\documentclass[11pt,reqno]{amsart}
\usepackage{fullpage}
\usepackage{mathrsfs}
\usepackage{amsfonts}
\usepackage{amsmath}
\usepackage{amsthm}
\usepackage{amssymb}
\usepackage{latexsym}
\usepackage[all]{xy}

\usepackage{palatino}
\usepackage{fullpage}

\usepackage[mathscr]{eucal}

\usepackage{xr-hyper}
\usepackage{hyperref}
\usepackage{multicol}

\theoremstyle{plain}
\newtheorem{thm}[equation]{Theorem}
\newtheorem{cor}[equation]{Corollary}
\newtheorem{prop}[equation]{Proposition}
\newtheorem{lem}[equation]{Lemma}

\theoremstyle{definition}
\newtheorem{defn}[equation]{Definition}

\theoremstyle{remark}
\newtheorem{quest}[equation]{Question}

\newtheorem{rem}[equation]{Remark}

\newtheorem{notation}[equation]{Notation}

\makeatletter
\renewcommand{\subsection}{\@startsection{subsection}{2}{0pt}{-3ex
plus -1ex minus -0.2ex}{-2mm plus -0pt minus
-2pt}{\normalfont\bfseries}} \makeatother

\numberwithin{equation}{subsection}

\newcommand{\br}[1]{{\mbox{$\left(#1\right)$}}}

\newcommand{\Lmod}[1]{#1\text{-}{\mathsf{mod}}}

\newcommand{\hdot}{{\:\raisebox{2pt}{\text{\circle*{1.5}}}}}
\newcommand{\idot}{{\:\raisebox{2pt}{\text{\circle*{1.5}}}}}

\DeclareMathOperator{\res}{{\mathsf{res}}}

\DeclareMathOperator{\sym}{\mathrm{Sym}}

\DeclareMathOperator{\im}{\mathrm{Im}}
\DeclareMathOperator{\supp}{\mathrm{Supp}}
\DeclareMathOperator{\coh}{{\mathrm{Coh}}}
\DeclareMathOperator{\Ker}{\mathrm{Ker}}
\DeclareMathOperator{\coker}{\mathrm{Coker}}
\DeclareMathOperator{\End}{\mathrm{End}}

\DeclareMathOperator{\ssym}{{\widehat{\opp{S}}}}

\DeclareMathOperator{\gr}{\mathrm{gr}}
\DeclareMathOperator{\hilb}{{\mathrm{Hilb}}}

\DeclareMathOperator{\Lie}{\mathrm{Lie}}

\DeclareMathOperator{\ggr}{{\widetilde{\mathrm{gr}}}}
\DeclareMathOperator{\dgr}{{\widetilde{\mathrm{dgr}}}}
\DeclareMathOperator{\Ad}{\mathrm{Ad}}
\DeclareMathOperator{\ad}{\mathrm{ad}}

\DeclareMathOperator{\mmp}{{\mathrm{Map}}}

\DeclareMathOperator{\rad}{\mathsf{rad}}

\DeclareMathOperator{\mo}{{\,\mathrm{mod}\,[\b,\b]}}

\newcommand{\dis}{\displaystyle}
\newcommand{\oper}{\operatorname }
\newcommand{\beq}{\begin{equation}\label}
\newcommand{\eeq}{\end{equation}}

\DeclareMathOperator{\Spec}{\mathrm{Spec}}
\DeclareMathOperator{\pr}{pr}

\newcommand{\iso}{{\;\stackrel{_\sim}{\to}\;}}
\newcommand{\cd}{\!\cdot\!}
\DeclareMathOperator{\Hom}{\mathrm{Hom}}
\DeclareMathOperator{\hhom}{{{\scr H}\!{\mathit{om}}}}

\def\ccirc{{{}_{\,{}^{^\circ}}}}
\newcommand{\bo}{\mbox{$\bigotimes$}}
\renewcommand{\o}{\otimes }
\newcommand{\bplus}{\mbox{$\bigoplus$}}

\newcommand{\kap}{{\kappa}}
\newcommand{\kkap}{{\wt\kappa}}
\newcommand{\si}{\sigma }

\renewcommand{\L}{{\scr L}}

\renewcommand{\t}{{\mathfrak t}}
\newcommand{\del}{{\boldsymbol\delta}}
\newcommand{\rese}{\operatorname{\textit{res}_\bbe}}

\newcommand{\fz}{{\pmb g}}

\newcommand{\wt}{\widetilde }

\newcommand{\eu}{{\operatorname{\mathrm{eu}}}}
\newcommand{\eps}{{\epsilon}}
\newcommand{\ttr}{{{\mathfrak T}^r}}
\newcommand{\tts}{{\overset{_{_\circ}}{\mathfrak T}}}
\newcommand{\ts}{{\overset{_{_\circ}}{\mathfrak t}}}
\newcommand{\Id}{{\operatorname{Id}}}
\newcommand{\zz}{{\mathfrak C}}
\newcommand{\xx}{{\mathfrak X}}
\newcommand{\xc}{{\mathfrak X}^\circ}

\newcommand{\fl}{{\mathfrak l}}

\renewcommand{\lll}{{\pmb\langle}}
\newcommand{\rrr}{{\pmb\rangle}}

\newcommand{\ha}{{\mathfrak W}}
\newcommand{\yy}{{\mathfrak x}}
\newcommand{\yr}{{\mathfrak r}}
\newcommand{\tx}{{\wt{\mathfrak X}}}
\newcommand{\zr}{{\mathfrak C}^{rs}}
\newcommand{\sign}{{\,\textit{sign}}}

\newcommand{\xr}{{\mathfrak X}^{rs}}
\newcommand{\fo}{{}'\!F}
\newcommand{\ft}{{}''\!F}
\newcommand{\so}{{{{}'\!{\wh S}}\,}}
\newcommand{\st}{{{}''\!{\wh S}^{\,}}}
\newcommand{\GG}{{\boldsymbol{\Gamma}}}
\newcommand{\prsign}{\pr^\textit{sign}}
\newcommand{\reo}{{}'\rees}
\newcommand{\ret}{{}''\rees}
\newcommand{\gro}{{}'\!\!\gr}
\newcommand{\grt}{{}''\!\!\gr}

\newcommand{\opp}{\operatorname}
\newcommand{\rees}{{\mathsf{R}}}
\newcommand{\ups}{\Upsilon}
\newcommand{\ccong}{\ \cong\ }
\newcommand{\ra}{\rees\max\,\C[\wbh]}
\newcommand{\bd}{{\mathbf d}}
\newcommand{\bde}{\Delta_\bbe }

\newcommand{\Del}{{\nabla}}
\newcommand{\erem}{\hfill$\lozenge$\end{rem}\vskip 3pt }
\def\npb{\noindent$\bullet\quad$\parbox[t]{149mm}}
\newcommand{\scr}[1]{\mathscr{#1}}
\newcommand{\lo}{\stackrel{L}{\mbox{$\bigotimes$}}}

\newcommand{\mhm}{Hodge module }
\renewcommand{\th}{{\theta}}
\newcommand{\iim}{{\boldsymbol{\imath}}}
\newcommand{\vsi}{{\boldsymbol{\varsigma}}}
\newcommand{\vvpi}{{\boldsymbol{\varpi}}}
\newcommand{\vka}{{\varkappa}}

\renewcommand{\min}{^{\opp{min}\!}}
\renewcommand{\max}{^{\opp{max}\!}}

\newcommand{\alo}{\aleph }

\newcommand{\bbde}{{\mathbf E}}
\newcommand{\bbo}{{\boldsymbol{0}}}
\newcommand{\bbss}{{\mathbf S}}
\newcommand{\ta}{\tau}
\newcommand{\bbs}{{\mathbf s}}
\renewcommand{\AA}{{\mathsf{A}}}
\newcommand{\HH}{{\mathbb H}}
\newcommand{\VV}{{\scr V}}
\newcommand{\fa}{{\mathfrak a}}

\def\ccirc{{{}_{^{\,^\circ}}}}

\newcommand{\la}{\lambda}
\newcommand{\be}{\beta }
\newcommand{\ddd}{{\mathbf D}}

\newcommand{\DD}{{\boldsymbol{\mathcal D}}}
\newcommand{\ctt}{\C[\tt]}
\newcommand{\tat}{\ta_1,\ta_2 }

\newcommand{\bim}{_\natural }

\newcommand{\ssc}{{\mathcal{S}}}
\renewcommand{\ss}{{\mathcal{S}}}

\newcommand{\gl}{{\mathfrak g\mathfrak l}}
\newcommand{\ff}{{\mathcal F}}
\newcommand{\filt}{F^{\operatorname{ord}}}
\newcommand{\ord}{{\operatorname{ord}}}
\newcommand{\hodge}{{{\operatorname{Hodge}}}}%\stackrel{}

\newcommand{\dd}{{\mathscr{D}}}
\newcommand{\oo}{{\mathcal{O}}}
\newcommand{\kk}{{\mathcal{K}}}
\newcommand{\fy}{{\mathfrak Y}}
\newcommand{\wh}{\widehat }
\renewcommand{\aa}{{\mathscr{A}}}

\newcommand{\mh}{{\mathsf{HM}}}
\renewcommand{\tt}{{\mathfrak T}}
\newcommand{\h}{{\mathfrak t}}
\newcommand{\uh}{{\mathfrak h}}
\newcommand{\hh}{{{\mathfrak T}}}

\newcommand{\dcoh}{D^b_{\operatorname{coh}}}
\newcommand{\pa}{\partial }
\renewcommand{\H}{{\scr H}}
\newcommand{\ZZ}{{\mathcal G}}

\newcommand{\bbe}{{\mathbf e}}
\newcommand{\bx}{{\mathbf x}}
\newcommand{\by}{{\mathbf y}}

\newcommand{\bh}{{\mathbf h}}
\newcommand{\bj}{{\mathbf j}}
\newcommand{\CC}{{\mathbb{C}^\times\times\mathbb{C}^\times}}
\newcommand{\CF}{{\mathcal F}}

\newcommand{\NN}{{\mathcal N}}
\newcommand{\tg}{{\wt\g}}

\newcommand{\nil}{{\operatorname{nil}}}

\newcommand{\al}{\alpha }
\renewcommand{\gg}{{\mathfrak G}}

\newcommand{\wbh}{{W\cd\bh}}
\newcommand{\rk}{{\mathbf{r}}}

\newcommand{\hc}{{Harish-Chandra }}

\newcommand{\TT}{{T}}
\newcommand{\I}{{{\boldsymbol{\mathcal I}}}}
\newcommand{\J}{{\mathcal J}}

\newcommand{\C}{\mathbb{C}}
\newcommand{\g}{\mathfrak{g}}

\renewcommand{\b}{\mathfrak{b}}

\newcommand{\norm}{{\operatorname{norm}}}
\newcommand{\red}{{\operatorname{red}}}
\newcommand{\La}{\Lambda }

\newcommand{\pp}{{\scr P}}
\newcommand{\rr}{{\scr R}}

\newcommand{\bb}{{\scr B}}

\newcommand{\mm}{{\mathcal M}}

\newcommand{\T}{{\mathcal T}}
\newcommand{\inv}{^{-1}}

\newcommand{\Z}{{\mathbb Z}}
\newcommand{\en}{{\enspace}}

\newcommand{\vi}{${\en\sf {(i)}}\;$}
\newcommand{\vii}{${\;\sf {(ii)}}\;$}
\newcommand{\viii}{${\sf {(iii)}}\;$}
\newcommand{\iv}{${\sf {(iv)}}\;$}

\newcommand{\sset}{\subset}
\newcommand{\sminus}{\smallsetminus}
\newcommand{\into}{\,\hookrightarrow\,}
\newcommand{\too}{\,\,\longrightarrow\,\,}
\newcommand{\mto}{\mapsto}
\newcommand{\onto}{\,\twoheadrightarrow\,}
\newcommand{\N}{{\mathcal{N}}}
\newcommand{\X}{{\mathbb X}}
\newcommand{\lel}{\prec}
\newcommand{\ccl}{\curlyeqprec }

\newcommand{\tgg}{{\wt{\mathfrak G}}}
\newcommand{\mmu}{{\boldsymbol{\mu}}}
\newcommand{\nnu}{{\boldsymbol{\nu}}}
\newcommand{\ppi}{{\boldsymbol{\pi}}}

\newcommand{\Ga}{\Gamma }

\newcommand{\om}{\omega }

\renewcommand{\sl}{{\mathfrak s\mathfrak l}}
\begin{document}
\title{{\large{\textsf{Isospectral  commuting variety,\,
 the Harish-Chandra}}}\\
\vskip4pt
{\large{\textsf{{\large$\boldsymbol{\mathcal D}$}-module,
and principal nilpotent pairs}}}}

\author{Victor Ginzburg}\address{
Department of Mathematics, University of Chicago,  Chicago, IL 
60637, USA.}
\email{ginzburg@math.uchicago.edu}

\begin{abstract} Let  $\g$ be  a complex reductive Lie algebra 
with Cartan algebra $\h.\ $ 
%In 1980's, 
Hotta and Kashiwara defined a
holonomic $\dd$-module $\mm$, on $\g\times\h$, called
Harish-Chandra module. We relate
$\gr\mm$, an associated
graded module with respect to
 a canonical  {\em Hodge filtration} on $\mm$,
to
the {\em isospectral
commuting variety}, a subvariety of $\g\times\g\times\h\times\h$
which is a ramified cover of the variety of pairs of commuting
elements of $\g$.
Our main result establishes an isomorphism
of $\gr\mm$ with  the structure
sheaf of $\xx_\norm$, the normalization of the isospectral
commuting variety.
We deduce, using Saito's theory of 
Hodge $\dd$-modules, that the scheme $\xx_\norm$
is    Cohen-Macaulay
and Gorenstein. This confirms a conjecture of M. Haiman.

Associated with any
{\em principal nilpotent pair} in $\g$, there is
a finite subscheme  of $\xx_\norm$.
The corresponding coordinate ring is a bigraded 
 finite  dimensional Gorenstein algebra that
affords the regular representation of the Weyl
group. The socle of that algebra is a 1-dimensional space
generated by a remarkable $W$-harmonic polynomial on $\t\times \t$.
In the special case where $\g=\gl_n$
the above algebras are closely related to 
the $n!$-theorem of Haiman and  our  $W$-harmonic
polynomial reduces to the Garsia-Haiman polynomial.
Furthermore,  in the $\gl_n$ case, the sheaf $\gr\mm$
gives rise  to a vector bundle on the Hilbert scheme
of $n$ points in $\C^2$ that turns out to be isomorphic to the 
{\em Procesi bundle}.
Our results were used by I. Gordon to obtain
a new proof of
 positivity of the Kostka-Macdonald polynomials established
earlier by Haiman.
\end{abstract}
\maketitle
{\small
\tableofcontents
}

\section{Introduction}
%sss

\subsection{Notation}\label{not} We work over the ground
field  $\C$ of complex numbers and we write
$\o=\o_\C.$\footnote{A more complete Index of Notation is
given at the end of the paper.}

By a  scheme $X$ we  mean a scheme of
finite type over $\C$. We write $X_\red$ for the corresponding
 reduced scheme and $\psi:\ X_\norm\to X_\red$
for the normalization map
(if $X_\red$ is irreducible).
Let  $\oo_X$ denote the structure
sheaf of $X$, resp.  $\kk_X$ the canonical sheaf
(if $X$ is Cohen-Macaulay), and
  $\dd_X$  the sheaf of 
 algebraic differential
operators on $X$ (if $X$ is smooth). 
Write $\C[X]=\Ga(X,\oo_X)$, resp. $\dd(X)=\Ga(X,\dd_X)$,
for the  algebra of
global sections.
Let $T^*X$ denote (the total space of) the 
cotangent bundle
on a smooth variety $X$.

Given an algebraic group $K$ and a $K$-action on $E$,
we write $E^K$ for the set of $K$-fixed points.
In particular, for a $K$-variety $X$,
one has the subalgebra $\C[X]^K\sset \C[X],$ resp.
$\dd(X)^K\sset \dd(X)$, of $K$-invariants.

Throughout the paper,  we fix
a connected complex reductive group $G$ with
Lie algebra $\g$. 
Let $T\sset G$ be a maximal torus,
 $\t=\Lie T$  the corresponding Cartan subalgebra 
of $\g$, and  $\rk=\dim\t$ the rank of $\g$.
Write  $N(T)$ for the normalizer of $T$ in $G$ so, $W=N(T)/T$ 
is the Weyl group.
The group $W$ acts   on $\t$ via the reflection representation
and it acts on  $\wedge^\rk\t$
by the sign character  $w\mto \sign(w)$.
We   write
$E^{\sign}$ for the $\sign$-isotypic component of 
  a $W$-module $E$.

\subsection{Definition of the \hc module}\label{r_sec}
We will use
a special notation $\DD:=\dd_{\g\times\t}$
for the sheaf of  differential operators
on  $\g\times\t$.
We have $\Ga(\g\times\t, \DD)=\dd(\g\times\t)=\dd(\g)\o\dd(\t)$
where  $\dd(\g)$, resp. $\dd(\t)$, is the algebra of polynomial differential operators
on the vector space $\g$, resp. on $\t$.
The  subalgebra of $\dd(\g)$, resp. of $\dd(\t)$, formed 
by the   differential operators  with
constant coefficients may be identified with
$\sym\g$, resp. with $\sym\t,$ the corresponding
symmetric algebra.

Let
the group $G$ act
on $\g$ by the adjoint action.
\hc  \cite{HC} defined a `radial part' map
$
\rad:\ \dd(\g)^G\to\dd(\t)^W.$
This is
an algebra homomorphism such  that its restriction to 
the subalgebra of $G$-invariant polynomials,
resp.  of $G$-invariant
constant coefficient  differential operators,
reduces to
the Chevalley isomorphism
$\C[\g]^G\iso \C[\t]^W,$ 
resp. $(\sym\g)^G\iso (\sym\t)^W.$

%Similarly, one has a linear map
%$\ad:\ \g=\g^* \to \g^*\o\g^*\sset \C[\g^*]\o\C[\g^*]=\C[\g^*\times\g^*]
%=\C[\gg]$ 
%obtained by dualizing the commutator map
%$\kap: \g\o\g\to\g$. We write $\ad\g\sset\C[\gg]$
%for the image of the map $\ad$ so defined.

Given $a\in \g$, one may view
the  map $\ad a:\ \g\to\g,\ x\mto [a,x]$ as a (linear)  vector
field on $\g$, that is, 
 as a first order differential operator on $\g$.
The assignment $a\mto\ad a$ gives
a linear map $\ad: \g\to\dd(\g)$ with image $\ad\g$.
Thus, one can form a left ideal
 $\DD\,(\ad\g\o 1)\ \sset\ \DD$.

\begin{defn}\label{mm_def} 
The \hc module is a left $\DD$-module defined as follows
\beq{M}
\mm\
:=\ \DD\big/\br{\DD\,(\ad\g\o 1)+\DD\,\{u\o1-1\o \rad(u),\ u\in \dd(\g)^G\}}.
\eeq
\end{defn}

\begin{rem} For a useful interpretation of this formula see also
\eqref{mmbimod}.\erem

According to an important result of
 Hotta and Kashiwara \cite{HK1}, 
 the \hc module is a simple holonomic $\DD$-module
of `{\em geometric origin}', cf. Lemma \ref{supp}(ii) below.
This implies that  $\mm$
comes equipped with a natural structure of
{\em Hodge module} in the sense
of M. Saito \cite{Sa}. In particular,
there is a canonical Hodge filtration on $\mm$,
see \S\ref{hc_hodge}.
Taking an associated graded sheaf with respect to
the  Hodge filtration produces a coherent sheaf
$\ggr^\hodge\mm$ on $T^*(\g\times\t)$.

The support of the sheaf $\ggr^\hodge\mm$  turns out to be
closely related to the commuting scheme of the Lie algebra $\g$, see 
Theorem \ref{mthm} below.
The main idea of the paper is to exploit the powerful
theory of  Hodge modules to deduce
new results  concerning commuting schemes using
information about the  sheaf $\ggr^\hodge\mm$.

\begin{rem}
Definition \ref{mm_def} was motivated by,
but is not identical to, the definition 
of  Hotta and Kashiwara, see \cite{HK1},
formula (4.5.1). The equivalence of the two definitions
follows from Remark \ref{I=I}, of \S\ref{hk} below.
\end{rem}

\subsection{Main results}\label{main_subsec}
Put  $\gg=\g\times\g$ and  let $G$ act diagonally on $\gg$.
The {\em commuting scheme} $\zz$ is defined as the {\em scheme-theoretic}
zero fiber of the commutator map $\kap:\ \gg\to\g,$ $(x,y)\mto[x,y]$.
Thus, $\zz$ is  a $G$-stable closed  subscheme of 
$\gg$; set-theoretically, one has
$\zz=\{(x,y)\in\gg\mid [x,y]=0\}$.
The scheme 
$\zz$ is known to be generically reduced and irreducible,
cf. Proposition \ref{zzbasic} below.
It is a long standing open problem whether or not
this scheme is reduced.

Let  $\tt:=\t\times\t\sset \gg$.
It is clear that
$\tt$ is an $N(T)$-stable closed subscheme of $\zz$ and
the resulting $N(T)$-action  on $\tt$ factors through the diagonal
action of the Weyl group  $W=N(T)/T$. Therefore,
  restriction of polynomial functions
gives algebra maps 
\beq{Jo} 
\res:\ \C[\gg]^G\, \onto\,\C[\zz]^G\,\to\,\C[\tt]^W.
\eeq

The 
{\em isospectral commuting
variety} is defined to be the algebraic set:
\beq{xxDef} \xx=\{(x_1,x_2,t_1,t_2)\in \zz\times\tt\mid
P(x_1,x_2)=(\res P)(t_1,t_2),\en \forall P\in\C[\zz]^G\}.
\eeq
We
view $\xx$ as a  {\em reduced} closed  subscheme of
$\gg\times\tt$, cf. also  Definition \ref{xxdef}.

To proceed further, we fix   an invariant
bilinear form $\langle-,-\rangle:\ \g\times\g\to\C$.
This gives  an isomorphism 
$\g\iso\g^*,\ x\mto \langle x,-\rangle$,
 resp. $\t\iso\t^*,\ t\mto -\langle t,-\rangle\ $
(the minus sign in the last formula is related to the
minus sign that appears in the anti-involution $v\mto v^\top$
considered in \S\ref{hc_filt}).
Thus, one gets an identification $T^*(\g\times\t)=\gg\times\tt$
so one may view
$\ggr^\hodge\mm$ as a coherent sheaf on $\gg\times\tt$.

One of the main results of the paper,
whose proof will be completed in \S\ref{xyzsec}, reads

\begin{thm}\label{mthm} There is a natural
$\oo_{\gg\times\tt}$-module isomorphism
$\psi_*\oo_{\xx_\norm}\iso \ggr^\hodge\mm.$
\end{thm}

This theorem combined with some deep results of
Saito \cite{Sa}, to be reviewed in \S\ref{hodged}, yields the following 
theorem that confirms a conjecture of M. Haiman, \cite[Conjecture
7.2.3]{Ha3}.

\begin{thm}\label{main_thm}  $\ \xx_\norm$ is a Cohen-Macaulay
and Gorenstein variety with trivial
canonical sheaf.
\end{thm}

Theorem \ref{main_thm} will
 be deduced from Therem \ref{mthm} in \S\ref{main_res}.

\begin{cor}[\S\ref{smooth_pf}]\label{zzcm} The scheme $\zz_\norm$ is  Cohen-Macaulay.
\end{cor}

\begin{cor}[\S\ref{LSsec}]\label{LS}  The  $\dd_\g$-module
$\dd_\g/\dd_\g\cd\ad\g$ comes equipped with
a canonical filtration $F$ such that one has
an isomorphism
$\ggr^F(\dd_\g/\dd_\g\cd\ad\g)\cong \psi_*\oo_{\zz_\norm}$,
of $\oo_\gg$-modules.
\end{cor}

The last corollary implies (see \S\ref{LSsec}) the following result
that has been proved earlier by
Levasseur and Stafford  \cite[Theorem 1.2]{LS2}
in a totally  different way.
\begin{cor}\label{ls_cor} \vi
$\dd(\g)/\dd(\g)\ad\g$ is a Cohen-Macaulay (non-holonomic) left
$\dd(\g)$-module.

\noindent
\vii  The natural {\em right} action
of the algebra $\C[\g]^G$  makes $\dd(\g)/\dd(\g)\ad\g\ $ a {\em flat}
 right $\C[\g]^G$-module.
\end{cor}

\subsection{Group actions}\label{actions}
One has a natural  $G\times W$-action
on $\g\times\t$,
 resp. on $\gg\times\tt$, 
where the group   $G$ acts  on the first factor and 
 the group $W$ acts on the second factor.
There is also a $\C^\times$-action on
$\g$, resp. on $\g\times\t$,
by dilations. So, we obtain
a $\C^\times\times\C^\times$-action on $\gg=T^*\g$,
resp.
on $\gg\times\tt=(\g\times\t)\times(\g\times\t)=
T^*(\g\times\t)$,
such that the standard $\C^\times$-action by dilations along the fibers of 
the cotangent bundle corresponds, via the
above identification,
to the action 
of the subgroup $\{1\}\times\C^\times
\sset\C^\times\times\C^\times$.
Thus, we have made the  space $\g\times\t$   a $G\times
W\times\C^\times$-variety, resp. the  space  $\gg\times\tt$ a
$G\times
W\times\C^\times\times\C^\times$-variety.

The scheme $\xx$ is clearly  $G\times
W\times\C^\times\times\C^\times$-stable. The resulting $G\times
W\times\C^\times\times\C^\times$-action on $\xx$ 
induces one on $\xx_\norm$
since a  reductive group action  can always 
be lifted canonically to
the normalization, \cite{Kr}, \S4.4.

On the other hand, the 
group $G\times W\times\C^\times$ acts on $\g\times\t$ and
the  \hc module $\mm$ has the natural structure of
a $G\times W\times\C^\times$-equivariant
$\DD$-module.
The  Hodge filtration on $\mm$ is canonical,
therefore, this $G\times W\times\C^\times$-action
respects the filtration.
Hence, the
group $G\times W\times\C^\times$ acts naturally
on $\ggr^\hodge\mm$. There is also an additional
$\C^\times$-action on $\ggr^\hodge\mm$ that comes
from the grading. Thus, combining all these actions together,
one may view $\ggr^\hodge\mm$ as a
 $G\times W \times\C^\times\times\C^\times$-equivariant
coherent sheaf on $\gg\times\tt$.

The isomorphism of Theorem \ref{mthm}
respects the  $G\times W \times\C^\times\times\C^\times$-equivariant
structures.
The equivariant
structure on the sheaf $\oo_{\xx_\norm}$
makes the coordinate ring $\C[\xx_\norm]$
 a bigraded locally
finite $G\times W$-module.

In \S\S\ref{dima_sec},\ref{xyzsec}, we will construct
a "DG resolution" of 
$\xx_\norm$, a (derived) "double analogue" of
the Grothendieck-Springer resolution,
cf. Remark \ref{dgres}. We show that
 the DG algebra
of global sections of our resolution
is acyclic in nonzero degrees.
Using this, a standard
application of  the
 Atiyah-Bott-Lefschetz fixed
point theorem yields the  following result.
\begin{thm}[see {\cite[\S2.4]{BG}}]\label{character_cor} The bigraded $T$-character
of  $\C[\xx_\norm]$ is  given by the formula
$$
\chi^{\CC\times T}(\oo_{\xx_\norm})=
\frac{1}{(1-q )^\rk(1-t )^\rk}\cdot
\sum_{w\in W}\ 
w\left(\prod_{\al\in R^+}\frac{1-qt\,
e^\al}{(1-e^{-\al})(1-q\,e^\al)(1-t\,e^\al)}\right).
$$
\end{thm}
Here, $R^+$ denotes the set of positive roots of $\g$ and
the product in
the right hand side of the formula is understood as a formal power series of the form
$\sum_{m,n\geq0}\ a_{m,n}\cdot q^mt^n$ where
the coefficients $a_{m,n}$ are viewed as elements
of the representation ring of the torus $T$.

We refer to
\cite{BG} for the proof and some combinatorial applications of the above
theorem.
\subsection{A coherent sheaf on the commuting scheme}\label{dsec}
The first projection $\gg\times\tt\to\gg$
restricts to a
map $p:\ \xx\to\zz$. Therefore, the
composite $\xx_\norm\to\xx
\to\zz$
factors through  a  morphism $p_\norm:\ \xx_\norm\to\zz_\norm.$
It is immediate to see that $p_\norm$ is a
 {\em finite} $G
\times\C^\times\times\C^\times$-equivariant
morphism and that the group $W$ acts along the fibers of this
morphism.

Let $\rr:=(p_\norm)_*\oo_{\xx_\norm}$. This is  a  
$G\times\C^\times\times\C^\times$-equivariant
coherent sheaf of $\oo_{\zz_\norm}$-algebras.
The action of 
the group $W$ along the fibers of the map $p_\norm$ 
 gives a $W$-action on $\rr$
by $\oo_{\zz_\norm}$-algebra automorphisms.
Therefore, for any finite dimensional
$W$-representation $E$, 
 one has  a coherent  sheaf $\rr^E:=(E\o\rr)^W$,
 the $E$-isotypic component of $\rr$.
Equivalently, in terms of  the contragredient
representation  $E^*$, we have $\rr^E=\Hom_W(E^*,\rr).$

The sheaf $\rr$
enjoys the following  properties, see \S\ref{pfvb}.

\begin{cor}\label{vectbun}  \vi   The sheaf
$\rr$ is   Cohen-Macaulay and we
 have 
$$\ \oo_{\zz_\norm}\cong\rr^W,\quad\opp{resp.}\quad
 \kk_{\zz_\norm} \cong\rr^{\sign}.
$$

\noindent
\vii  There  is a
 $G\times W\times \CC$-equivariant isomorphism
$\rr\cong \hhom_{\oo_{\zz_\norm}}(\rr,\, \kk_{\zz_\norm})$.
Furthermore, for any  finite dimensional
$W$-module $E$, this gives an isomorphism
$$
\rr^{E^*\o\sign}\ \cong\  \hhom_{\oo_{\zz_\norm}}(\rr^E,\, \kk_{\zz_\norm}).
$$
\end{cor}

Given $x\in \g$, let
 $\g_x$ denote the centralizer of $x$ in $\g$.
Similarly, write
 $\g_{x,y}=\g_x\cap\g_y$ for
 the centralizer of a pair  $(x,y)\in\zz$ in $\g$.
We call an element $x\in\g$, resp.  a pair   $(x,y)\in\zz$,
{\em regular} if we have $\dim\g_x=\rk$,
resp.  $\dim\g_{x,y}=\rk.$
Let $\g^r$, resp. $\zz^r$, be the set of regular elements
of $\g$, resp. of $\zz$.  
One shows that the set $\zz^r$ is a Zariski open
and dense subset of $\zz$ which is equal
to the smooth
locus of the scheme $\zz$,  cf. Proposition \ref{zzbasic} below.

There is a coherent sheaf $\fz$ on $\zz_\norm$,
the "{\em universal stabilizer sheaf}", such that the
geometric fiber of $\fz$ at  each point 
 is the Lie algebra
of the isotropy group of that point under the
$G$-action,
cf. \S\ref{bei} for a more rigorous definition.
Any pair $(x,y)\in \zz^r$ may be
  viewed  as a point
of $\zz_\norm$. The sheaf
$\fz|_{\zz^r}$ is 
  locally free; its
  fiber at any
 point $(x,y)\in\zz^r$ equals, by definition, the  vector space $\g_{x,y}$.

Part (i) of the following theorem 
says that the sheaf $\rr$ gives an algebraic vector
bundle on  $\zz^r$ that has very interesting structures.
Part (ii) of the theorem 
provides a description of  the 
isotypic components 
 $\rr^{\wedge^{\!^s\!}\t}=(\wedge^{\!^s\!}\t\o\rr)^W$,
coresponding to the wedge powers $\wedge^{\!^s\!}\t,\
s \geq 0,$ of the reflection
representation of $W$, in terms of the sheaf $\fz$.

\begin{thm}[\S\S\ref{smooth_pf},\ref{txr}]\label{rrsmooth}
\vi The restriction
of the sheaf $\rr$ to $\zz^r$
 is a locally free sheaf.
Each fiber of the corresponding algebraic vector bundle
is a finite dimensional algebra that affords
the regular representation of 
the group ~$W$.

\vii For any $s\geq0$, there is a natural $G\times\CC$-equivariant  isomorphism
$\rr^{\wedge^s\t}|_{\zz^r}\ \cong\ \wedge^s\fz|_{\zz^r}.$
\end{thm}

\subsection{Small representations}\label{small_intro}
Let $L$ be a finite dimensional  rational $G$-representation.
Given a Lie subalgebra ${\mathfrak a}\sset\g$,
we put $L^{\mathfrak a}:=\{v\in L\mid av=0,\ \forall a\in
{\mathfrak a}\}$. In particular, we have that $L^\t$ is the zero weight space of $L$.
In \S\ref{unistab}, we introduce a coherent sheaf
$L^\fz$ on $\zz_\norm$ such that the geometric fiber of
$L^\fz$ at any sufficiently general point
$(x,y)\in \zz_\norm$ is  the vector space
$L^{\g_{x,y}}$.
 
Following A. Broer  \cite{Br}, we call
$L$ {\em small} if  the set of weights of $L$
is contained in the root lattice of $\g$ and,
moreover, $2\al$  is not a weight of $L$, for any
root $\al$. Part (i) of our next theorem 
provides a description of  $W$-isotypic
components of the sheaf $\rr$ which correspond to the
$W$-representation in 
 the zero weight space of a small $G$-representation.

\begin{thm}\label{charb}Let   $E$
be a $W$-representation, resp. $L$ be a
{\em small} $G$-representation. Then,  we have 
\vskip 2pt

\vi There is a canonical  isomorphism
$\dis L^\fz\cong \rr^{L^\t}$, of  $G\times\CC$-equivariant 
$\oo_{\zz_\norm}$-sheaves.

\vii
 Restriction to
$\tt\sset\zz_\norm$
induces bigraded $\dis\C[\tt]^W$-module isomorphisms:
$$
\Ga(\zz_\norm,\ \rr^E)^G\ \iso\ (E\o\C[\tt])^W,
\quad\oper{resp.}\quad
(L\o \C[\zz_\norm])^G\ \iso\ (L^\t\o\C[\tt])^W.
$$

\viii For any  ~$s\geq0$, one has an isomorphism
$\Ga(\zz^r,\ \wedge^s\fz)^G\ccong 
(\wedge^s\t\o\C[\tt])^W$.
\end{thm}

Part (i) of the theorem will be proved in \S\ref{txr}
and parts (ii)-(iii)  will be proved in \S\ref{pfvb2}.

\begin{rem}\label{remb}
\textsf{(a)}\en  The adjoint representation $L=\g$ is small.
In this case, one has $L^\fz=\fz$.
So, the above theorem yields a sheaf isomorphism
$\fz\cong\rr^\t$ and a bigraded $\C[\tt]^W$-module isomorphism
$(\g\o \C[\zz_\norm])^G\cong(\t\o\C[\tt])^W$.

\textsf{(b)} Let $G=PGL(V)$ and let $n=\dim V$. 
The natural $GL(V)$-action on   $(V^*)^{\o
n}\o \wedge^n V$ descends to $G$ and the resulting
$G$-representation $L$ is known to be  small. Furthermore, 
the zero weight space of that representation
is isomorphic, as a $W$-module,
to  the regular representation of the Symmetric
group $W=S_n$.  Hence,  for $L=(V^*)^{\o
n}\o \wedge^n V$, we have $\rr\cong\rr^{L^\t}.$
Therefore,  Theorem \ref{charb}(i) yields an isomorphism 
$\rr\cong ((V^*)^{\o n}\o \wedge^n V)^\fz.$

Using  the tautological injective morphism $u:\ L^\fz\into L\o
\oo_{\zz_\norm}$
and writing $u^*$ for the dual morphism 
one  obtains the following {\em selfdual} diagram of sheaves on $\zz_\norm$:
$$
\xymatrix{
V^{\o n}\o \wedge^n V^*\o\kk_{\zz_\norm}\
\ar[r]^<>(0.5){u^*}&
\ \hhom_{\oo_{\zz_\norm}}(\rr, \kk_{\zz_\norm})\ \cong\ \rr\
\ar@{^{(}->}[r]^<>(0.5){u}&\
(V^*)^{\o n}\o \wedge^n V\o\oo_{\zz_\norm},
}
$$
where the isomorphism in the middle is due to Corollary \ref{vectbun}(ii).

The above diagram is closely related to formula (40) in
\cite[Proposition 3.7.2]{Ha1}.
\end{rem}

\subsection{Principal nilpotent pairs}\label{pnp_sec}
Given  a regular point  $\bx=(x_1, x_2)\in \zz$,
let  $\rr_{\bx}$ be  the
  fiber at $\bx$ of (the algebraic vector
bundle on  $\zz^r$ corresponding to)
 the locally free sheaf $\rr$. 
 By definition, one has 
$\rr_{\bx}=\C[p\inv_\norm({\bx})]$ where
$p\inv_\norm({\bx})$, 
the  scheme theoretic fiber of the morphism $p_\norm$
over  $\bx$, is a $W$-stable (not necessarily reduced) finite
 subscheme of $\xx_\norm$.
Thus, $\rr_{\bx}$  is a finite dimensional algebra 
equipped  with a $W$-action.

The  $W$-module $\rr_\bx$ is isomorphic
 to the regular representation of $W,$
by Theorem \ref{rrsmooth}(i).
In particular, one has
 ${\dim\rr_\bx^W=\dim\rr_\bx^{\sign}=1}$.
The line $\rr_\bx^W$ is clearly spanned by the unit of the algebra
$\rr_\bx$. Further, one has a canonical  map
 $\rr_\bx\to \rr_\bx^{\sign},\ r\mto r^\sign,$
 the $W$-equivariant projection to  the  isotypic component of the sign
 representation.
This map
gives, thanks to the  isomorphism 
$\kk_{\zz_\norm}\cong\rr^{\sign}$
of Corollary \ref{vectbun}(i),
 a nondegenerate  trace on the algebra $\rr_\bx$.
In other words,  the assignment $r_1\times r_2\mto$ $(r_1\cdot r_2)^\sign$
provides a nondegenerate symmetric bilinear pairing on $\rr_\bx$.

The most interesting fibers of the sheaf $\rr$
are, in a sense, the fibers over {\em principal nilpotent pairs}.
Following \cite{pnp}, we call
a regular pair $\bbe=(e_1,e_2)\in\zz^r$   a principal nilpotent pair
if there exists
a rational  homorphism $\CC\to G,\ (\tau_1,\tau_2)\mto g(\tau_1,\tau_2)$
such that one has
\beq{pnpdef}
\tau_i\cdot e_i\ =\ \Ad g(\tau_1,\tau_2)(e_i)\quad i=1,2,
\qquad \forall \tau_1,\tau_2\in\C^\times.
\eeq

Given a  principal nilpotent pair $\bbe$, 
we introduce 
 a `twisted'
$\CC$-action 
 on $\gg$ given, for  $(\tau_1,\tau_2)\in \CC$, by the formula
$$
(x,y)\mto
 (\tau_1,\tau_2)\bullet_\bbe(x,y)=\big(\tau_1\cdot\Ad g(\tau_1,\tau_2)\inv(x),\
\tau_2\cdot\Ad g(\tau_1,\tau_2)\inv(y)\big).
$$

The   $\bullet_\bbe$-action on $\gg$, combined
with the usual $\CC$-action on $\tt$ by
dilations of the two factors $\t$, gives a  $\CC$-action on $\gg\times\tt$.
The subvarieties $\zz$ and $\xx$ are clearly  $\bullet_\bbe$-stable.
Lifting the actions to normalizations, one gets a
$\CC$-action (to be referred to as a {$\bullet_\bbe\text{-action}$} again)
 on $\zz_\norm$, resp. on
$\xx_\norm$.
The map $p_\norm:\ \xx_\norm\to\zz_\norm$ is   
$\bullet_\bbe$-equivariant. 

Equations \eqref{pnpdef} force  $e_1,e_2$ be nilpotent elements of $\g$.
The scheme  $p\inv_\norm(\bbe)$ is   nonreduced and it has a single closed point, the 
element
$(\bbe,0)\in \gg\times\tt$.
The point $\bbe\in\zz^r$  is, by construction,
 a fixed point of the  $\bullet_\bbe$-action
on $\zz_\norm$.
Hence,  $p\inv_\norm(\bbe)$ is a $\bullet_\bbe$-stable subscheme of
$\xx_\norm$. We conclude that
the coordinate ring $\rr_\bbe=\C[p\inv_\norm(\bbe)]$ is
a local algebra and the $\CC$-action on $p\inv_\norm(\bbe)$
gives a   $\Z^2$-grading $\rr_\bbe=\bigoplus_{m,n\in\Z}\
\rr_\bbe^{m,n}$ on that algebra.

Associated with the nilpotent pair $\bbe=(e_1,e_2)$
there is a pair of commuting semisimple
 elements of $\g$ defined by
 $h_s:=\frac{\partial g(\tau_1,\tau_2)}{\partial \tau_s}\big|_{\tau_1=\tau_2=1},\
s=1,2.$ 
Let $\g^1=\g_{h_2},$ resp. $\g^2=\g_{h_1}$
(note that the indices are "flipped").
Thus, $\g^s,\ s=1,2,$ is a Levi subalgebra of $\g$ 
such that $e_s\in \g^s.$
The Lie algebra
 $\g_{h_1, h_2}=\g^1\cap\g^2$ is known to be a Cartan subalegebra of $\g$.
So,  we may (and will) put $\t:=\g_{h_1, h_2}$.
Let $R^+_s$  be the set of positive roots,
resp.  $W_s$
the Weyl group, of the reductive Lie algebra $\g^s$.

The pair $\bh=(h_1,h_2)$ is regular, \cite[Theorem 1.2]{pnp}; furthermore,
the fiber $p\inv_\norm(\bh)$ is a reduced finite subscheme of $\gg\times\tt$.
Specifically, writing $\wbh$ for the $W$-orbit of the
 element $\bh\in\tt,$ one has a bijection
$\wbh\iso p\inv_\norm(\bh),\
w(\bh)\mto (\bh, w(\bh)).$
Thus, the algebra $\rr_\bh=\C[p\inv_\norm(\bh)]$ is  a semisimple algebra
isomorphic to
$\C[\wbh]$, the coordinate ring of the set ~$\wbh$.

Let $\C^{\leq m}[\t],\ m=0,1,2,\ldots,$ be the 
space of polynomials on $\t$ of degree $\leq m$.
We introduce a pair of ascending filtrations  on
the algebra $\C[\tt]=\C[\t]\o\C[\t]$
defined by $\fo_{m}\C[\tt]=\C^{\leq m}[\t]\o\C[\t]$, resp.  
$\ft_{m}\C[\tt]=\C[\t]\o\C^{\leq m}[\t].$
The algebra $\C[\wbh]$ is a quotient of
the algebra $\ctt$ hence it inherits from $\ctt$ a
 pair  of  quotient filtrations $\fo_\hdot\C[\wbh]$
and $\ft_\hdot\C[\wbh]$, respectively.
We  further define   bifiltrations
$$F_{m,n}\ctt:=\fo_m\ctt\cap\ft_n\ctt,\quad\,
F_{m,n}\C[\wbh]:=\fo_m\C[\wbh]\cap\ft_n\C[\wbh],\en
m,n\geq 0,$$
 on the algebras $\ctt$ and $\C[\wbh]$, respectively.
Let $\gr^F\ctt$, resp. $\gr^F\C[\wbh]$,
be  an associated {\em bi}graded
algebra, 
see \S\ref{pnpsec} for more details.

One of the  central results
of the paper is the following theorem
motivated, in part, by \cite[\S 4.1]{Ha3}. 
Part (i) of the theorem describes how  $\C[\wbh]=\rr_\bh$, a
semisimple   Gorenstein algebra,
 degenerates to the bigraded Gorenstein algebra $\rr_\bbe$.

\begin{thm}\label{n!} \vi There is a $W$-equivariant
$\Z^2$-graded  algebra isomorphism
$\rr_\bbe \ccong \gr^F\C[\wbh].
$
\vskip2pt

\vii We have
$\rr^{i,j}_\bbe=0$ unless
$\ 0\leq i\leq \bd_1\en \&\en 0\leq j\leq\bd_2,\ $
where $\bd_s:=\# R^+_s,\ s=1,2.$
\end{thm}	

The proof of this theorem occupies \S\S\ref{rees_ctt}-\ref{npf}.
The main idea of the proof is to use the semisimple pair
$\bh$ to produce a 2-parameter deformation of the
point $\bbe$ inside $\zz^r$. 
One may  pull-back  the
sheaf $\rr|_{\zz^r}$ to the parameter space of the deformation.
This way, we obtain a
 locally free sheaf  on $\C^2$. On the other hand,
a construction based on Rees algebras
gives another  locally  
free sheaf on $\C^2$ such that its fiber over
the origin is the vector space $\gr^F\C[\wbh]$.
Using a careful analysis (Lemma \ref{rees}) based on the theory
of $W$-invariant polynomials on $\t$
we obtain an isomorphism between the restrictions
of the two 
sheaves in question to the complement of the origin in $\C^2$
(Proposition \ref{cruc}).
We then exploit the fact that
an isomorphism between
the restrictions to the punctured plane of two {\em locally  free} sheaves
automatically extends across the puncture.

From Theorem \ref{n!} we deduce, cf.
\S\ref{pnpappl},

\begin{cor}\label{kostka} 
There is a natural
$W$-module isomorphism (cf. \cite[Proposition 4.1.2]{Ha3}):
$$ \bplus_{m\geq 0}\ \rr_\bbe^{0,m}\ccong
\C[W/W_1],
\quad\opp{resp.}\quad
\bplus_{n\geq 0}\ \rr_\bbe^{n,0}\ccong
\C[W/W_2].
$$
\end{cor}

One has the following criteria for the variety $\xx$ be normal
(hence, by Theorem \ref{main_thm},
 also Cohen-Macaulay and Gorenstein)
at the  point $(\bbe,0)$, see \S\ref{pnpappl}:

\begin{cor}\label{normlem} The following properties
of a principal nilpotent pair $\bbe$ are equivalent:

\vi The restriction map $\ctt=\Ga(\zz_\norm,\rr)^G\to \rr_\bbe$  is surjective;

\vii The natural projection  $\gr^F\ctt\to\gr^F\C[\wbh]$
 is surjective;

\viii The variety $\xx$ is normal at the point $(\bbe,0)$.
\end{cor}

In the special case of the group
$G=GL_n$, 
Theorem \ref{n!} follows from the work of M. Haiman, \cite[\S 4.1]{Ha3}.
Moreover, in this case, thanks to Haiman's result on
the normality of the isospectral Hilbert scheme (\cite{Ha1}, Proposition 3.8.4)
 and to the classification
of principal nilpotent pairs in the Lie algebra $\gl_n$ 
(see \cite[Theorem 5.6]{pnp}), one knows that   the  equivalent
properties of
Corollary \ref{normlem} hold true for
{\em any} principal nilpotent pair.
We remark also that  Haiman shows
that  the validity of the Cohen-Macaulay property
of the scheme $\xx$ at each principal nilpotent pair
in the Lie algebra $\gl_n$
is actually equivalent
to the validity of the $n!$-theorem, see \cite[Proposition 3.7.3]{Ha1}.

On the other hand,
 Haiman produced
an example, see \cite[\S 7.2.1]{Ha3}, of a principal nilpotent
pair $\bbe$ for the  Lie algebra
$\g={\mathfrak s\mathfrak p}_6$
where the analogue of the $n!$-theorem  fails,
hence   the 
corresponding homomorphism
 $\gr\ctt\to\gr\C[\wbh]$ is  not surjective
and the scheme $\xx$ is  not normal
at the point $(\bbe,0)$. Another example is
provided by the {\em exceptional}  principal nilpotent pair
in the simple Lie algebra of type  $\mathbf E_7$ discussed below.

\subsection{The polynomial $\bde$}\label{bde_sec}
From 
the isomorphism $\rr_\bbe=\gr^F\C[\wbh]$, of Theorem \ref{n!},
we see
 that  $\gr^F\C[\wbh]$ is a Gorenstein algebra and
 that the line
$(\gr^F\C[\wbh])^\sign$ is  the socle
of that algebra.
In most cases, one can actually obtain a more explicit
description of the socle.

To explain the meaning of the words "most cases", we recall that
the  simple Lie algebra of type  $\mathbf E_7$
has one `exceptional'  conjugacy class
of  principal nilpotent pairs  $\bbe=(e_1,e_2)$ such that each of the 
nilpotent elements
$e_1$ and $e_2$ has Dynkin labels 
$\, \left[{{\hphantom{10}0\hphantom{010}}\atop{101010}}\right].\, $
Following \cite[Definition 4.1]{pnp},
we call a
 principal nilpotent pair $\bbe=(e_1,e_2)$,
of an arbitrary reductive Lie algebra $\g$,
 {\em non-exceptional}
provided none of the components of $\bbe$ corresponding
to the simple factors of $\g$ of
type $\mathbf E_7$ belong to the
exceptional  conjugacy class of  principal nilpotent pairs.

Let $\check h_s=\langle h_s,-\rangle,\ s=1,2,$
be a linear function on $\t$ that corresponds to
the element $h_s$ via the  invariant form.

The following result provides a simple
description of the socle of the algebra $\gr^F\C[\wbh]$
in the case of nonexceptional nilpotent  pairs.

\begin{thm}[see \S\ref{pftr}]\label{trace}
For any
  non-exceptional   principal nilpotent pair  $\bbe$,
 we have 
$$
(\gr^F\C[\wbh])^\sign\ =\ \gr_{\bd_1,\bd_2}\C[\wbh].$$

Morover, a base vector of the  1-dimensional vector space on the right
is provided by
the image under the map
$\gr^F\C[\tt]\to\gr^F\C[\wbh]$ of the
class of the following  bihomogeneous   polynomial:
\beq{bde}
\bde:=
\sum\nolimits_{w\in W}\
\sign(w)\cdot w(\check h_1^{\bd_1}\o \check h_2^{\bd_2})\
\in\ \C[\tt]^\sign.
\eeq
\end{thm}

\begin{rem}
The polynomial $\bde$ was first introduced  in \cite{pnp}.
It is a $W$-harmonic polynomial on $\tt$ that provides a natural
generalization 
to the case of arbitrary reductive Lie algebras
of the Garsia-Haiman polynomial on $\C^n\times\C^n$.
The proof of Theorem \ref{trace} is based
on the properties of  $\bde$ established in \cite[\S4]{pnp}.
The relevance of the 
notion of non-exceptional pair is due 
to \cite[Theorem 4.4]{pnp}, one of
the main results of {\em loc cit},  which says that
$\bde\neq 0$ holds if and only if the principal nilpotent pair $\bbe$ is
non-exceptional.
\erem

One may   pull-back the function $\bde$ 
via the projection $\xx_\norm\to\tt$.
The resulting $W$-alternating function on $\xx_\norm$ 
gives a $G$-invariant section, $\bbs_\bbe$,
of the  sheaf $\rr^\sign.$ Let
$\bbs_\bbe(\bbe)\in \rr^\sign_\bbe$ denote
the value of the section  $\bbs_\bbe$
at the point $\bbe$.

\begin{cor}[see \S\ref{pnpappl}]\label{section} For a
 non-exceptional principal nilpotent pair $\bbe$,
one has: $\bde(\bh)\neq0$  and $\bbs_\bbe(\bbe)\neq0$.
Moreover, we have that
$\rr_\bbe^\sign=\rr^{\bd_1,\bd_2}_\bbe=\C\cdot\bbs_\bbe(\bbe)$
is the socle of the algebra $\rr_\bbe$.
\end{cor}

Let
$\lll -,-\rrr:\ \tt\times\tt\to\C$
be the natural $W\times W$-invariant  bilinear form
given 
by the formula $\lll \bx,\by\rrr :=\langle
x_1,y_1\rangle+\langle x_2,y_2\rangle$,
for  any $\bx=(x_1,x_2),\, \by=(y_1,y_2)\in\tt$.

We define a holomorphic
function $\bbde$ on $\tt\times\tt$ as follows
$$\bbde(\bx,\by):=\sum\nolimits_{w\in W}\ \sign(w)\cdot
e^{\lll \bx,\, w(\by)\rrr},
\qquad \bx,\by\in\tt.
$$

It is clear that, for any $\bx,\by$, we have $\bbde(\bx,\by)=
\bbde(\by,\bx)$.
Observe also that the function $\bbde$ is $W$-invariant
with respect to the diagonal action on $\tt\times\tt$
 and, for any fixed $\bx\in\gg$, the function
$\bbde(\bx,-)=\bbde(-,\bx)$ is a $W$-alternating holomorphic
function on $\tt$.

Similarly to the above, one can pull-back the function $\bbde$ 
via the projection $\xx_\norm\times\xx_\norm\to\tt\times\tt$.
The resulting function on $\xx_\norm\times\xx_\norm$
gives a $G\times G$-invariant
{\em holomorphic} section, $\bbss$, of the coherent
sheaf $\rr^\sign\boxtimes\rr^\sign$ on $\zz\times\zz$.

The following result shows that the section  $\bbss$ 
provides a 
canonical {\em holomorphic} interpolation between 
the {\em algebraic} sections
$\bbs_\bbe$ associated with
 various, not necessarily $G$-conjugate, principal nilpotent pairs $\bbe$
of the Lie algebra $\g$.

\begin{prop}[\S\ref{pnpappl}] \label{exp}
For any non-exceptional principal nilpotent
pair $\bbe$, 
 in $\rr^\sign\boxtimes\rr^\sign_\bbe$,
resp. in $\rr^\sign_\bbe\boxtimes\rr^\sign$,
 one has
$$
\bbss|_{\zz\times\{\bbe\}}\ =\ 
c\cdot\bbs_\bbe\boxtimes
\bbs_\bbe(\bbe),\en\text{resp.}\en
\bbss|_{\{\bbe\}\times\zz}\ =\ 
c\cdot\bbs_\bbe(\bbe) \boxtimes
\bbs_\bbe\quad\text{where}\en
c=\frac{1}{\bd_1!\cd\bd_2!\cd\bde(\bh)}.
$$
\end{prop}

\subsection{Relation to work of M. Haiman}\label{mhai}
Let $\hilb^n(\C^2)$ be  the Hilbert scheme
 of $n$ points in $\C^2$.
In his work on the $n!$-theorem, Haiman introduced
a certain {\em isospectral Hilbert scheme}
$\wt\hilb^n(\C^2)$,
a reduced finite scheme over 
$\hilb^n(\C^2)$, see
\cite{Ha1}. The main
result of {\em loc cit} says
 that  $\wt\hilb^n(\C^2)$ is
a normal,  Cohen-Macaulay and  Gorenstein scheme.

Now, let $G=GL_n$.
It turns out that there is a Zariski open dense subset
$\zz^\circ\sset\zz^r$ such that
the projection $p_\norm :\
p_\norm\inv(\zz^\circ)\ \to\ \zz^\circ$
is closely related to the
 projection $\wt\hilb^n(\C^2)\to
\hilb^n(\C^2)$, see \S\ref{hai_sec}. Using this relation, we are able to
deduce from our Theorem \ref{main_thm}
that the {\em normalization}
of the isospectral Hilbert scheme
 is  Cohen-Macaulay and  Gorenstein,
see  Proposition \ref{haim}.
 Unfortunately, 
our approach does not seem to yield an independent
proof of  normality of  the 
isospectral Hilbert scheme, while the proof of  normality given
in \cite[Proposition 3.8.4]{Ha1} is based on the `polygraph theorem'
\cite[Theorem 4.1]{Ha1}, a key
technical result of \cite{Ha1}.

Nonetheless, we are able to use the locally free sheaf $\rr|_{\zz^r}$ 
to construct a rank $n!$ algebraic vector bundle $\pp$
on $\hilb^n(\C^2)$ whose fibers afford
the regular representation of the Symmetric group.
The results  of Haiman  \cite{Ha1} then insure
that the vector bundle
$\pp$ is isomorphic, {\em a posteriori}, to the
{\em Procesi bundle}, cf. Corollary \ref{procesi}.

We note that the properties of the
vector bundle  $\pp$ that we construct 
are often sufficient (without the knowledge of
the isomorphism with the  Procesi bundle) for applications.
This is so, for instance, in the  new proof
of the  positivity of Kostka-Macdonald polynomials found
recently by I. Gordon \cite{Go}.
Another situation of this kind is provided
by Theorem \ref{rres_thm} below. See also
 various combinatorial results established in \cite{BG}.
\medskip
 
Let
$V=\C^n$ be the fundamental representation of $G=GL_n$
and  let $\VV$ be the tautological rank $n$ vector
bundle on $\hilb^n(\C^2)$, see \S\ref{VT}.

Given an integer $m\geq 0$, let
$\C^m[V]$ be
the space of degree $m$ homogeneous polynomials
on $V$, and let $S_m$
be the symmetric group on $m$ letters.
The group $S_m$  acts
naturally on $V^{\o m}$, resp. on
$\VV^{\o m}$, hence also
on $\C[\xx_\norm]\o\C^m[V]\o V^{\o m}$, resp. on
$\pp\o\VV^{\o m}$, via the action on the last tensor
factor.

We also have the group $W=S_n$ and a natural
$W\times\C^\times\times\C^\times$ action on $\xx_\norm$, resp. on
$\pp$, hence 
on $\C[\xx_\norm]\o\C^m[V]\o V^{\o m}$, resp. on
$\pp\o\VV^{\o m}$, via the action on the first
factor. Finally, we let the subgroup  $SL_n\sset GL_n$
 act diagonally
on $\C[\xx_\norm]\o\C^m[V]\o V^{\o m}$.
The  $SL_n$-action
commutes with the previously defined actions  of other groups.
Thus,  the space $\big(\C[\xx_\norm]\o\C^m[V]\o V^{\o m}\big)^{SL_n}$,
of $SL_n$-invariants, acquires the natural structure of
 an $S_m\times W$-equivariant bigraded
$\C[\tt]$-module.

\begin{thm}[see \S\ref{VT}]\label{rres_thm}
There  is a 
$W\times S_m$-equivariant
bigraded $\C[\tt]$-module isomorphism
$$
\big(\C[\xx_\norm]\o \C^m[V]\o  V^{\o m}\big)^{SL_n}\
\iso\
\Ga(\hilb^n(\C^2),\ \pp\o\VV^{\o m}),\quad \forall m\geq 0.
$$
\end{thm}

We remark that the
space of global sections that appears on the right hand side of the
isomorphism  above  is
the "polygraph space" that played
a key role in the work of Haiman.
In \cite[Theorem 3.1]{Ha4}, Haiman
gives a  bigraded character formula
for the polygraph space
in terms of  Macdonald polynomials.  On the other hand,
our Theorem  \ref{character_cor}
may be used to find the bigraded character
of the vector space $\big(\C[\xx_\norm]\o \C^m[V]\o  V^{\o m}\big)^{SL_n}$.
Combining these results with Theorem \ref{rres_thm}, one
obtains
a certain explicit identity that
relates the  formal power series in the right hand side
of the formula of Theorem  \ref{character_cor}
to Macdonald polynomials, see ~\cite{BG} for more details.

A particularly nice special case occurs in the
situation  where the integer $m$, in Theorem \ref{rres_thm},
 is a multiple of $n$,
i.e., such that $m=kn$ for some integer $k\geq 1$.
Let $A^k\sset\C[\tt]$ be the vector space
spanned by the products of $k$-tuples of elements of $\C[\tt]^{\text{sign}},$
the space of $W$-alternating polynomials.
The bigrading on $\ctt$ induces one on $A^k$ and we have

\begin{prop}\label{aa} For any $k\geq 1$, there
is a natural
bigraded 
isomorphism (see  \S\ref{gg}):
$$
\big(\C[\zz_\red]\o \C^{kn}[V]\big)^{SL_n}\ 
\ \iso\  A^k.$$
\end{prop}

This result may be viewed as an example of
 a situation involving the isomorphism of  Theorem \ref{rres_thm}
where taking normalizations is unnecessary.

\subsection{Layout of the paper} In section \ref{main_sec}
we begin with  basic properties of commuting 
and isospectral  varieties. We then review
the material involving holonomic $\dd$-modules
and (polarizable) Hodge modules that will
be used in the paper.
We show that the Harish-Chandra module $\mm$
is a simple holonomic $\dd$-module
and give $\mm$ the structure of a   Hodge module.
We explain that Saito's results on  Hodge modules
imply easily that the coherent sheaf $\gr^\hodge\mm$
is Cohen-Macaulay and Gorenstein.
This allows us to reduce  the
proof of Theorem \ref{mthm} 
to an isomorphism $\gr^\hodge\mm|_{\xx^{rr}}\cong
\oo_{\xx^{rr}}$, where ${\xx^{rr}}$ is a certain
open subset of $\xx$, and to a  codimension estimate
involving ${\xx^{rr}}$.
We 
finish the section by showing
how  Theorem \ref{mthm} implies  Theorem \ref{main_thm}.

In section \ref{int}, we recall the defininition of
the Grothendieck-Springer resolution
and construct
its `double analogue'. We  introduce a DG algebra $\aa$ such that its
homogeneous components are locally free sheaves
on the double analogue of the Grothendieck-Springer resolution.
The spectrum of the DG algebra $\aa$
may be thought of as a DG resolution of the isospectral
commuting variety.

In section \ref{sympl_sec}, we recall the Hotta-Kashiwara construction
of  the Harish-Chandra module $\mm$ as a direct image of the
structure sheaf  of the Grothendieck-Springer resolution.
We use the  Hotta-Kashiwara construction to obtain
a description of the sheaf $\gr^\hodge\mm$
in terms of the  DG algebra $\aa$. From that  description,
we deduce the above mentioned
 isomorphism $\gr^\hodge\mm|_{\xx^{rr}}\cong
\oo_{\xx^{rr}}$.

In section \ref{dima_sasha}, we give the definition of the
`universal stabilizer' sheaf. Further, we associate
with any finite dimensionsal representation $L$
of the Lie algebra $\g$
a coherent sheaf $L^{\fz}$, on $\g^r$.
 In the case where the representation $L$ is
small, we obtain, using a refinement
of an idea due to Beilinson and Kazhdan,
a description of the sheaf $L^{\fz}$ 
in terms the `universal 
zero weight space' of $L$. 
At the end of the section we indicate how to modify
the construction in order to obtain a similar description of
 $L^{\fz}$ for a not necessarily small representation
$L$.

In section \ref{END}, we introduce a  stratification
of the commutating variety and use it to prove
the above mentioned dimension estimate involving the set
$\xx^{rr}$ which is required
for our proof of  Theorem \ref{mthm}.
We then prove Theorems \ref{rrsmooth} and \ref{charb}
by adapting the construction of \S\ref{dima_sasha}
to a `double' setting. We also prove Corollary \ref{vectbun}.

Section \ref{pnpsec} is devoted to the proofs of all the
results stated in \S\S\ref{pnp_sec}-\ref{bde_sec} of the
Introduction.

In section \ref{hai_sec}, we consider the case of the
group $GL_n$. We apply our results about commuting
varieties to the geometry of the isospectral Hilbert scheme.
We discuss relations with the work of M. Haiman and
prove the results stated in \S\ref{mhai}.

Finally, in section \ref{some}, we prove Corollary \ref{LS}
and  Corollary \ref{ls_cor}.
\medskip

We remark that the proof of the main results, Theorem
\ref{mthm} and Theorem \ref{main_thm}, is rather long. The argument
involves several ingredients of different nature.
The discussion of these  igredients is spread over a number of
sections of the paper. Our   division into sections
has been made according to the nature of the material rather
than to the order in which that  material is actually
used in the proof. 

To make the logical structure 
of our arguments more clear, below is an  outline of
 the main steps of the
proof of  Theorem
\ref{mthm} and Theorem \ref{main_thm}, listed
in  the {\em logical} order:
\medskip

{\small
\begin{description}
\item[\S\S\ref{main_res}, \ref{ss3}] Define a Zariski open smooth
subset $\xx^{rr}\sset\xx$ and show that the set $\xx\sminus\xx^{rr}$
has codimension $\geq 2$ in $\xx$ (Corollary \ref{xxcor}).

\item[\S\S\ref{tx}, \ref{tx_sec}] Introduce the `double analogue' of  
 the Grothendieck-Springer resolution and show that it is 
an isomorphism over $\xx^{rr}$ (Corollary \ref{transv}).

\item[\S\S\ref{hc_filt}, \ref{hc_hodge}] Show that $\mm$ is a simple
$\DD$-module with the natural structure of a Hodge module.

\item[\S\S\ref{filt_sec}, \ref{homsec}] 
Use the  Hotta-Kashiwara
construction  of
$\mm$ to
 get a description (see Corollary \ref{dima22}) of $\ggr^\hodge\mm$ in terms of
the double analogue of   the Grothendieck-Springer resolution.

\item[\S\ref{xyzsec}] Establish
an isomorphism: $\
\ggr^\hodge\mm|_{\xx^{rr}}\ \cong\ \oo_{\xx^{rr}}\ $ (see Proposition
\ref{res_lem}
and \S\ref{xyzsec}).

\item[\S\ref{hc_hodge}]
Deduce from Saito's theory   that the sheaf $\ggr^\hodge\mm$
is  Cohen-Macaulay and selfdual (Corollary \ref{serre}).

\item[\S\ref{main_res}] Deduce  Theorem \ref{mthm} from
Lemma \ref{ei}, using  Proposition \ref{res_lem}
and Corollary \ref{xxcor}\ (which were proved at an earlier stage).
 Deduce  Theorem \ref{main_thm} from Theorem \ref{mthm}
and Corollary ~\ref{serre}.

\end{description}
}
\subsection{Acknowledgements}
{\small I am  grateful
to Dima Arinkin and Gwyn Bellamy for  useful discussions.
Special thanks are due to  Iain Gordon for showing
me a draft of \cite{Go} well before that paper
was finished, for many  stimulating discussions, and for a carefull
reading of (many) preliminary versions of the present paper.
I thank Eric Vasserot for pointing out a connection
of Theorem \ref{character_cor} with a result in
\cite{SV}, cf. also \S\ref{dima_sec}. Vladimir Drinfeld
and Mark Haiman have helped me with some questions
of algebraic geometry.

This work was supported in part by an NSF award DMS-1001677.}

\section{Analysis of the Harish-Chandra module}\label{main_sec}
\subsection{Isospectral varieties}\label{ysec}
Let  $\g^{rs}\sset \g^r$ be the set
of regular  semisimple elements, resp.  $\zz^{rs}\sset\zz^r$ be  the set
of pairs $(x,y)\in \zz^r$ such that both $x$ and $y$ are
 semisimple.

Basic properties of the commuting scheme $\zz$
 may be summarized
as follows.
\begin{prop}\label{zzbasic} \vi The set $\zz^{rs}$ is 
 Zariski open and dense in $\zz$. Thus, $\zz$ is a generically reduced and 
irreducible scheme; moreover, we have $\dim\zz=\dim\g+\rk.$

\vii The smooth
locus of the scheme $\zz$ equals $\zz^r;\ $
for any $(x,y)\in\zz$ one has $\dim\g_{x,y}\geq\rk$.
\end{prop}

Here, part (i) is due  to Richardson \cite{Ri1}.
For the proof of part (ii) see e.g. \cite[Lemma 2.3]{Po},
or Remark \ref{zzregpf} of section \ref{END}  below. 
\medskip

Let $\g/\!/G:=\Spec \C[\g]^G$, resp. $\zz/\!/G:=\Spec \C[\zz]^G$,
be the
categorical quotient. The natural restriction homomorphism
$\C[\g]^G\to\C[\t]^W,$ resp. $\C[\zz]^G\to\C[\tt]^W$,
cf. \eqref{Jo}, gives  morphisms of schemes
$\t\to\t/W\to\g/\!/G,$ resp. $\tt\to\tt/W\to\zz/\!/G$.
\begin{rem} The morphism
$\t/W\to\g/\!/G$ is an isomorphism by the Chevalley
restriction theorem. One can show, using that the 
map $\C[\zz]^G\to\C[\tt]^W$ is surjective by a  theorem  of Joseph \cite{Jo},
that the morphism $\tt/W\to\zz/\!/G$
induces  an isomorphism $\tt/W\iso[\zz/\!/G]_\red$.
It is expected that the   scheme $\zz/\!/G$ is in fact reduced.
This is known to be so in the special case of the group $G=GL_n$, 
see \cite[Theorem ~1.3]{GG}.
\erem

%a diagram 
%\beq{X}
%\xymatrix{
%\tt\ \ar@{->>}[r]&\ \tt/W\
% \ar@{=}[rr]^<>(0.5){\eqref{Jo}}&&\ [\zz/\!/G]_\red\
%\ar@{^{(}->}[r]&\ \zz/\!/G\ &\ \zz.\ar@{->>}[l]
%}
%\eeq

Next, we form a fiber product $\yy:=\g\times_{\g/\!/G}\h,$ resp. 
 $\zz\times_{\zz/\!/G}\tt$, 
a closed $G\times W$-stable
subscheme of $\g\times\t,$ resp.  of $\gg\times\tt$.
It is clear that the set $\yy^{rs}:=\g^{rs}\times_{\g/\!/G}\h$
 is    a smooth Zariski open
subset of $\yy$, resp. $\xr:=\zz^{rs}\times_{\zz/\!/G}\tt$
 is  a smooth Zariski open
subset of $\zz\times_{\zz/\!/G}\tt$.
The first projection $\yy\to\g,$ resp.
$\zz\times_{\zz/\!/G}\tt\to\zz$,
is a $G$-equivariant {\em finite} morphism.
The group
$W$ acts along the fibers of this morphism.

Part (i)  of the following lemma is  Corollary \ref{xdense}
of \S\ref{END} below, and part (ii)  is clear.

\begin{lem}\label{xxbasic}
 \vi The set
$\yy^{rs}$,
resp.
$\xr$,
 is an irreducible dense subset of $\yy$, resp. of $\ \zz\times_{\zz/\!/G}\tt$.

\vii The first projection  $\yy^{rs}\to\g^{rs},$ resp.
$\xr\onto \zr$, is
a Galois covering
with Galois group ~$W$.
\end{lem}

The scheme $\yy$ is known to be a reduced normal
complete intersection in $\g\times\t$,
cf. eg \cite{BB}.
On the contrary, the scheme $\zz\times_{\zz/\!/G}\tt$
is not  reduced
already in the case $\g=\sl_2$.
This motivated  M. Haiman, \cite[\S8]{Ha2}, \cite[\S 7.2]{Ha3},
to introduce the following definition, which is equivalent
to formula \eqref{xxDef} of the Introduction.

\begin{defn}\label{xxdef} The
{\em isospectral commuting variety}
is defined as
$\dis\xx:=[\zz\times_{\zz/\!/G}\tt]_\red,$
a  {\em reduced} fiber product. Let
 $p :\ \xx\to\zz$, resp. $p_{_\tt}:\ \xx\to\tt$, 
denote the first, resp. second, projection.
\end{defn}

Lemma \ref{xxbasic}(i)
shows that $\xx$ is an irreducible variety and that
we may (and will) identify
the set $\xr$ with a smooth Zariski open and dense subset of $\xx$.

Let $\yr:=\{(x,t)\in \yy\mid x\in \g^r\}$.
This is  a Zariski open and dense subset
of $\yy$ which is
contained in the smooth locus of $\yy$
(since the differential of the adjoint quotient
map $\g\to\g/\!/G$ is known to have maximal rank at any point of
$\g^r$).
Let $N_{\yr}$ be  the total space of the conormal
bundle on $\yr$ in $\g\times\t$ and let
$\overline{N_{\yr}}$ be the closure of $N_{\yr}$  in $T^*(\g\times\t)$.

In general, let  $X$ be a smooth variety. An
irreducible reduced subvariety $\Lambda \sset T^*X$
is said to be {\em Lagrangian} if 
the tangent space to $\La$ at any smooth
point of $\La$
is a Lagrangian 
subspace of the tangent space to $T^*X$  with respect to the
canonical symplectic 2-form on ~$T^*X.$

\begin{lem}\label{conormal} In $T^*(\g\times\t)=\gg\times\tt$,
we have $\overline{N_{\yr}}=\xx$.
In particular, $\xx$ is a $\C^\times\times\C^\times$-stable Lagrangian subvariety
and  $N_{\yr}$ is a smooth Zariski open and dense subset of $\xx$.
\end{lem}
\begin{proof} Let $\t':=\t\cap \g^r$
and $\xx':=\{(x,y,t_1,t_2)\in\xx\mid
x\in\g^{rs}\}$. We claim that
$\xx'\sset N_\yr$. To see this we
observe  that the assignment
$gT \times (t_1,t_2) \mto \big(\Ad g(t_1),\, \Ad
g(t_2),\,t_1,\,t_2\big)$
yields  an isomorphism
$(G/T)\times \t'\times\t \iso \xx',$ see
Lemma \ref{codim}.
Thus, by $G$-equivariance, it suffices
to show that every point of $\xx'$
of the form $(t_1,t_2,t_1,t_2)$
is contained in $N_\yr$.
Any  tangent vector to $\yr$ at the point
$(t_1,t_1)\in\g\times\t$ has the form $\xi=([a,t_1]+t, t)\in \g\times\t$,
for some $a\in\g$ and some $t\in\t$.
Further, the covector that corresponds to the
 element $(t_1,t_2,t_1,t_2)$ via our
identification $\gg\times\tt\iso T^*\g\times T^*\t=T^*(\g\times\t)$
is the linear function $\xi^*:\ \g\times\t:\
(u,v)\mto \langle  t_2, u\rangle-\langle t_2,v\rangle.$
Therefore, using the invariance of the
bilinear form $\langle-,-\rangle$, we find
$$\langle\xi^*,\xi\rangle=
\langle  t_2, [a,t_1]+t\rangle-\langle t_2,t\rangle=
\langle  t_2, [a,t_1]\rangle=\langle a,[t_1,t_2]\rangle=0.$$

Thus, we have proved that $\xx'\sset N_\yr$.
Moreover, it is clear that $\xx'$ is an open subset of
$\xx\cap N_\yr$ and, we have  $\dim\xx'=
\dim (G/T\times\t'\times\t)=\dim(\g\times\t)
=\frac{1}{2}\dim T^*(\g\times\t)=\dim N_\yr$. 
Now, Lemma
\ref{xxbasic}(i) implies that $\xx$ is 
reduced and irreducible. Hence, $\xx'$ is
a dense subset  both in $\xx$ and in 
$N_\yr$. It follows that
$\xx=\overline{\xx'}=\overline{N_\yr}$.
\end{proof}

\subsection{}\label{ddfilt}
Let $X$ be a smooth variety and $q:\ T^*X\to X$ the cotangent bundle.
The  sheaf $\dd_X$ comes equipped with
an ascending filtration $\filt_\idot\dd_X$  by the order of
 differential operator. 
For the associated graded sheaf,
one has a canonical isomorphism $\gr^\ord\dd_X\cong
q_*\oo_{T^*X}.$

Let  $M$ be a $\dd_X$-module.
An
ascending filtration
 $F_\idot M$ such that $\filt_i\dd_X\cd F_j M\sset
F_{i+j}M,$ $\forall i,j,$
is said to be {\em good} if
$\gr^F M$, an associated graded module,
is a coherent $q_* \oo_{T^*X}$-module.
In that case, there is a canonically defined
 coherent sheaf $\ggr^F M$ 
on $T^*X$  such that one has an
isomorphism $\gr^F M=q_*\ggr^F M$ 
of   $q_* \oo_{T^*X}$-modules.
We write $[\supp(\ggr^F M)]$ for the {\em support cycle}
of the sheaf $\ggr^F M$, a linear
combination of the irreducible components
of  the support of $\ggr^F M$ counted with multiplicites.
This  is  an algebraic
cycle in $T^*X$ which  is known to be independent
of the choice of a good filtration on $M$, cf.
\cite{Bo}.

Recall that  $M$ is called {\em holonomic}
if one has $\dim \supp(\ggr^F M)=\dim X$.
In such a case,  each irreducible
component of $\supp(\ggr^F M)$, viewed as a reduced
variety, is a Lagrangian
subvariety of $T^*X$.
Holonomic $\dd_X$-modules form an abelian category
and the assignment $M\mto [\supp(\ggr^F M)]$
is additive on short exact sequences of holonomic modules,
cf. eg. \cite{Bo},  \cite{HTT}.
From this, one obtains

\begin{lem}\label{CC}  If $M$ is a 
$\dd_X$-module such that
the cycle $[\supp(\ggr^F M)]$ equals the fundamental cycle of
a Lagrangian subvariety, taken with multiplicity 1,
then $M$ is a {\em simple} holonomic $\dd_X$-module.
\end{lem}

Given  a  morphism $f:\ X\to Y$  of
smooth varieties, let
$\dd_{X\to Y}:=
\oo_X\bo_{f\inv\oo_Y}f\inv\dd_Y$.
Assuming $f$ is proper, there is a 
direct image functor 
$\int^R_f=Rf_*(-\lo_{\dd_X}\dd_{X\to Y})$
between  bounded derived categories of
coherent {\em right}
$\dd$-modules on $X$ and $Y$, respectively.
The corresponding functor on left $\dd$-modules
is then defined by
$
\int_f M= \kk\inv_Y\bo_{\oo_Y}\int_f^R (\kk_X\bo_{\oo_X}M)$, see 
 \cite[pp.22-23, 69]{HTT} for more details.

In the special case where  $f: X\into Y$
is a closed immersion,
one has  $\oo_X\bo_{f\inv\oo_Y}f\inv\dd_Y=\dd_Y|_X$
so, we get $\int^R_f\kk_X=f_*(\kk_X\o_{\oo_X} (\dd_Y|_X))$.
This $\dd_Y$-module
comes equipped with a  natural order filtration
 defined  by the formula
$F^\ord_m\big(\int^R_f\kk_X\big)=f_*\big(\kk_X\o_{\oo_X} (F^\ord_m\dd_Y)|_X\big),\
m\geq0$.
Thus, writing  $q:\ T^*Y\to Y$ for the
cotangent bundle, in this case one obtains
\beq{shiftgr}\ggr^\ord\br{\int^R_f\kk_Y}=
q^*(f_*\kk_X).
\eeq

\subsection{Filtered $\dd$-modules and duality}\label{hodged}
Below, we will use rudiments of the formalism of filtered derived categories.
Let
$$
E:\ 
\xymatrix{
\ldots\ \ar[r]^<>(0.5){d_{k-2}}&\
 E^{k-1}\ \ar[r]^<>(0.5){d_{k-1}}&
\ E^k\ \ar[r]^<>(0.5){d_{k}}&
\ E^{k+1}\ \ar[r]^<>(0.5){d_{k+1}}&\ldots,
}$$
be a filtered complex in an abelian category.
One has the following useful 

\begin{defn}\label{strict_def}The filtered complex $E$ is said to be {\em strict}
if the  morphism
$d_k: F_jE^k\to \im d_{k}\cap F_j E^{k+1}$ is surjective,
for any $k,j\in\Z$. 
\end{defn}
This notion only depends on the quasi-isomorphism
class of $E$ in an appropriate filtered derived category.

Given a filtered complex $E$, there is an induced
filtration on each cohomology group
$H^k(E),\ k\in\Z$.
It is immediate from the construction
of the standard spectral sequence of a filtered complex
that the  filtered complex $E$ is strict if and only if
the spectral sequence
$H^\hdot(\gr E)\ \Rightarrow\ \gr  H^\hdot(E)$  degenerates
at the first page,
cf. \cite[Lemma 3.3.5]{La}. In that case,
one has a canonical isomorphism
$\gr  H^\hdot(E)\cong H^\hdot(\gr E)$.

Following G. Laumon and M. Saito, for any smooth algebraic variety $X$,
one has
an exact (not abelian) category of
filtered left  $\dd_X$-modules and also
the corresponding  derived category.
Thus, let $\Lmod{{\mathsf F}\dd_X}$ be
an additive category whose objects are $\dd_X$-modules $M$ 
 equipped with a good filtration $F$. Further, abusing
the notation slightly, we let 
 $D^b_{coh}({{\mathsf F}\dd_X})$ be
the  triangulated category 
whose objects are isomorphic to
bounded complexes $(M,F),$ of filtered $\dd_X$-modules, 
such that each cohomology group $\H^i(M,F)$ is
an object of $\Lmod{{\mathsf F}\dd_X}$, cf. \cite{La}, \cite{Sa}.

In his work \cite{Sa}, M. Saito defines
a semisimple abelian category $\mh(X)$ of polarizable
 Hodge modules, see \cite[\S5.2.10]{Sa}. The data of a polarizable  \mhm includes,
in particular,  a holonomic $\dd_X$-module
 $M$ with regular singularities
and a good filtration $F$ on $M$
called {\em Hodge filtration}.\footnote{Part of the data 
involving
 "polarization"
will play no role in the present work.
For this reason, 
from now on, we will refer to  polarizable
 Hodge modules as `Hodge modules', for short.}
 Thus,
$(M,F)$ is a filtered $\dd_X$-module;
abusing 
notation, we write $(M,F)\in \mh(X)$
and let $\ggr(M,F)$ denote the corresponding
coherent sheaf on $T^*X.$

Let
${\mathbb D}(-)=R\hhom_{\dd_X}(-,\dd_X\o_{\oo_X} K_X\inv)[\dim X]\ $
be the standard Verdier duality functor
on $\dd$-modules, cf. \cite[\S2.6]{HTT}.
Laumon and  Saito have upgraded Verdier duality
to a triangulated contravariant duality functor
 ${\mathsf F}{\mathbb D}$ on the category $D^b_{coh}({{\mathsf
 F}\dd_X})$,
cf. \cite[\S4]{La}, \cite[\S2]{Sa}.
Furthermore, Saito showed that, for any $(M,F)\in \mh(X)$,
the filtered complex ${\mathsf F}{\mathbb D}(M,F)\in D^b_{coh}({{\mathsf F}\dd_X})$
is strict,
 see \cite{Sa}, Lemma 5.1.13.

Let $\coh Z$ denote the abelian category
of coherent sheaves on  a scheme $Z$
and let  $\dcoh(Z)$ denote
 the corresponding
bounded derived category.

The assignment
$(M,F)\mto \ggr^F M$ gives a functor
$\Lmod{{\mathsf F}\dd_X}$ $\to \coh T^*X$
that can be extended to a triangulated functor
$\dgr:\ 
D^b_{coh}({{\mathsf F}\dd_X})\to \dcoh(T^*X),$
cf. \cite{La}, (4.0.6).
It is essentially built into the construction of 
duality on  filtered derived categories that
 the functor $\dgr$
intertwines the duality  ${\mathsf F}{\mathbb D}$ on  
$D^b_{coh}({{\mathsf F}\dd_X})$
with the  Grothendieck-Serre duality on $\dcoh(T^*X)$,
 \cite[\S4]{La}, \cite[\S2]{Sa}.

Now let  $(M,F)\in \mh(X)$. Then $M$ is a holonomic
$\dd$-module. Hence ${\mathbb D}(M)$, viewed as an object
of the derived category,
is quasi-isomorphic to its 0th cohomology.
Therefore,  the strictness of the
filtered complex ${\mathsf F}{\mathbb D}(M,F)$
implies that the complex $\dgr({\mathsf F}{\mathbb D}(M,F))$
has nonvanishing cohomology in at most one degree.
It follows, in view of the above,
 that  the
 Grothendieck-Serre dual of the object $\ggr(M,F)\in \dcoh(T^*X)$
has  nonvanishing cohomology in at most one degree
again.
 Moreover, the  cohomology sheaf
in that degree is isomorphic
to $\ggr({\mathbb D}(M))$.
In particular,   one concludes  that
$\ggr(M,F)$ is a Cohen-Macaulay sheaf on $T^*X$, for any $(M,F)\in
\mh(X)$,
see 
\cite{Sa}, Lemma 5.1.13.

\subsection{The order filtration on the Harish-Chandra module}\label{hc_filt}
Observe that $\dd(\g)\cdot\ad\g$, a left ideal of  the algebra
$\dd(\g)$,
is stable under the  $G$-action on $\dd(\g)$
induced by the adjoint action on $\g$.
Multiplication in the algebra  $\dd(\g)$
gives  the
quotient  $\AA:=
\dd(\g)^G/[\dd(\g)\cdot\ad\g\big]^G$ 
a natural  algebra structure called
{\em quantum Hamiltonian reduction},
cf. eg. \cite[\S 7.1]{GG}. Furthermore, it
 gives $\dd(\g)/\dd(\g)\cdot\ad\g$
 the natural structure of an
$(\dd(\g),\AA)$-bimodule,  \cite{LS3}, \cite[\S 7.1]{GG}.

It is immediate from definitions
that
the radial part map $\rad$ considered in \S\ref{r_sec}
descends to a well defined  map
$\AA\to \dd(\t)^W$.
According to an  important result, due to 
 Levasseur and Stafford \cite{LS1}-\cite{LS2}
and Wallach  \cite{Wa}, the latter map
is
an algebra isomorphism. Thus,
one may view  $\dd(\g)/\dd(\g)\cdot\ad\g$
as  an
$(\dd(\g),\dd(\t)^W)$-bimodule.
It is easy to see that, taking global sections
on each  side of formula
 \eqref{M}, yields
an isomorphism, cf.
\cite[p. 1109]{LS3},
\beq{mmbimod}
\xymatrix{
\Ga(\g\times\t,\mm)\ \ar@{=}[r]^<>(0.5){\Xi}&\ [\dd(\g)/\dd(\g)\cd\ad\g]\,
\bo_{\dd(\t)^W}\, \dd(\t).
}
\eeq

Here, the object on the
left has the structure of a {\em left} $\dd(\g)\o\dd(\t)$-module and
the object on the right  has the structure of  a
$(\dd(\g),\dd(\t))$-bimodule.
These structures  are related  via the isomorphism
$\Xi$ as follows
$\Xi[(u\o v)m]=u\,\Xi(m)\,v^\top,$ 
for any $m\in \Ga(\g\times\t,\mm)$ and any
$u\o v
\in \dd(\g)\o\dd(\t)$ 
where $v\mto v^\top$ is an
 {\em anti}-involution of the algebra
$\dd(\t)$ given by  $t\mto t,\,\frac{\pa}{\pa t}\mto
-\frac{\pa}{\pa t}$.

According to formula \eqref{M} 
 the \hc module has the form $\mm=\DD/\I$
where  $\I$ is a left ideal of $\DD$.
The order filtration on  $\DD$ restricts to a filtration
on $\I$ and it also induces a quotient filtration 
$\filt_\idot\mm$ on $\DD/\I$.
Using the identifications $T^*(\g\times\t)=
\gg\times\tt$ and
 $\ggr^\ord\DD=\oo_{\gg\times\tt},$ we get
$\ggr^\ord\mm=\oo_{\gg\times\tt}/\ggr^\ord\I$
where $\ggr^\ord\I$, the associated graded
ideal, is a subsheaf of ideals of $\oo_{\gg\times\tt}$,
not necessarily reduced, in general.
A relation between
$\ggr^\ord\I$ and $\J\sset \oo_{\gg\times\tt}$,
the ideal sheaf
of the (non reduced) subscheme $\zz\times_{\zz/\!/G}\tt\sset
\gg\times\tt$, is  provided by
part (i) of the lemma below.

\begin{notation}\label{dx}
We put
$\del_\t:=\prod_{\al\in R^+}\, \al$ and $\t^r:=\t\cap \g^r$.
Let $dx\in \kk_\g$, resp. $dt\in \kk_\t$, be a constant
volume form on ~$\g$, resp. on $\t$. Thus, $dx\,dt$
is a section of $\kk_{\g\times\t}$.

Write  $\bj:\ \yr\into\g\times\t$
for the  locally closed imbedding and let
 $\bj_{!*}\oo_\yr$ be the minimal extension, see \cite{Bo}, of
 the structure sheaf $\oo_\yr$ viewed  as a $\dd_\yr$-module.
\end{notation}

\begin{lem}\label{supp} \vi One has inclusions
$\J\sset\ggr^\ord \I\sset\sqrt{\J}$.

\vii The \hc module $\mm$  is a  simple holonomic $\DD$-module,
specifically, one has an isomorphism
$\mm\cong \bj_{!*}\oo_\yr$ and an equality $[\supp(\ggr^\ord\mm)]=[\xx]$
of algebraic cycles in $\gg\times\tt$.
\end{lem}
\begin{proof} 
To simplify notation, it will be convenient below
to  work with spaces of global
sections rather than with sheaves.
Thus, we put
$I=\Ga(\g\times \t, \I)\sset \dd(\g\times\t),$
resp. $J=\Ga(\gg\times\tt, \J)\sset \C[\gg\times\tt]$.
The  vector space $\g\times\t$ being affine, 
it is sufficient to prove an analogue 
of the lemma for the ideals $\gr I$ and $J$
where throughout the proof we put $\gr:=\gr^\ord$.

Equip the space $\dd(\g)/\dd(\g)\cd\ad\g$ with the
quotient filtration. Then, taking associated graded
spaces on each side of isomorphism \eqref{mmbimod} yields the following
chain of
graded maps, where $pr$ stands for the natural projection
$$
\left(\frac{\gr\dd(\g)}{\gr[\dd(\g)\cd\ad\g]}\right)
\underset{\gr\dd(\t)^W}\bo\ \gr\dd(\t)
\!\xymatrix{\ar@{->>}[r]^<>(0.5){pr}&}\!
\gr\left[\left(\frac{\dd(\g)}{\dd(\g)\cd\ad\g}\right)\,
\underset{\dd(\t)^W}\bo\ \dd(\t)\right]
\!\xymatrix{\ar@{=}[r]^<>(0.5){\gr\Xi}&}\!
\gr\Ga(\g\times\t,\mm).
$$

We have $\gr\dd(\g)=\C[\gg]$.
Observe that the map $\gr(\ad):\ \g\to \gr\dd(\g)$
may be identified with the map $\kap^*: \g=\g^*\to\C[\gg]$,
 the pull-back morphism
induced by the commutator map $\kap$. 
By definition, one has
$\C[\gg]/\C[\gg]\cd\kap^*(\g)=\C[\zz]$.
Therefore, using the inclusion $\gr\dd(\g)\cd\gr(\ad)(\g)\sset \gr[\dd(\g)\cd\ad\g]$
we obtain a chain of
graded maps
\beq{a}
\C[\zz]=\C[\gg]/\C[\gg]\cd\kap^*(\g)
=\gr\dd(\g)/\gr\dd(\g)\cd\gr(\ad)(\g)
\onto \gr\dd(\g)/\gr[\dd(\g)\cd\ad\g].
\eeq
We let $a$ denote the composite  map in \eqref{a}.

Now, the radial part map $\rad: \dd(\g)^G\to\dd(\t)^W$ 
is known to respect the order filtrations, \cite[\S3]{Wa}. Moreover,
$\gr(\rad): \gr\dd(\g)^G\to\gr\dd(\t)^W,$ the associated graded map, is nothing but
the algebra map $\res: \C[\gg]^G\to\C[\tt]^W$ in 
~\eqref{Jo}. Thus, combining the maps in the two displayed formulas
above and
writing $b: \C[\tt]\iso\gr\dd(\t)$ for the standard isomorphism,
we obtain  the following graded surjective   maps

$$\xymatrix{
\C[\zz]\,\bo_{{\C[\tt]^W}}\,\C[\tt]\
\ar@{->>}[r]^<>(0.5){a\o b}&
\ {\Large\mbox{$\left(\frac{\gr\dd(\g)}{\gr[\dd(\g)\cdot\ad\g]}\right)$}}
\underset{\gr\dd(\t)^W}\bo\ \gr\dd(\t)\
\ar@{->>}[rr]^<>(0.5){\gr\Xi\ \ccirc\, pr}&&
\ \gr\Ga(\g\times\t,\mm).
}
$$

Further, by definition, we have
$\C[\gg\times\tt]/J=\C[\zz\times_{\zz/\!/G}\ \tt]=
\C[\zz]\bo_{\C[\zz]^G}\C[\tt].$ Thus,
the composite map in the last formula gives a
graded  surjective morphism 
$$
\xymatrix{
\C[\gg\times\tt]/J=\C[\zz]\bo_{\C[\zz]^G}\C[\tt]\
\ar@{->>}[rrr]^<>(0.5){\gr\Xi\, \ccirc\, pr\ccirc(a\o b)}&&&\ \gr\Ga(\g\times\t,\mm)
=\gr\dd(\g\times\t)/\gr I.
}
$$

This proves the inclusion
 $J\sset \gr I$. Hence, set theoretically, we have
$\supp(\ggr^\ord\mm)\sset~\xx$.

We know that
 $\xx^{rs}$ is an  open dense
{\em smooth}  subset of
 $\zz\times_{\zz/\!/G}\tt$, by Lemma \ref{xxbasic}(i).
Hence, the ideal $J$ is
generically reduced. 
Further, we know that $\xx$ is a Lagrangian subvariety (Lemma
\ref{conormal}) and that the dimension of 
any irreducible component
of the support of the sheaf  $\ggr\mm$ 
is $\geq\dim\xx.$ We conclude that
either $\mm=0$ or else we have that
 $[\supp(\ggr\mm)]=[\xx]$, so $\xx$
is the only irreducible component
of $\supp(\ggr\mm)$ and 
this component occurs with multiplicity 1.
In the latter case,
 $\mm$ must be a simple holonomic
$\dd$-module, by Lemma \ref{CC}.
Below, we mimic the proof of \cite[Proposition 4.7.1]{HK1}
to show that $\mm\neq0.$

Let 
$pr_\g$ and $pr_\t$ denote the natural
projections of $\yy^{rs}$
to $\g^{rs}$ and to $\t^r$, respectively.
The restriction of the imbedding
$\bj$ to the open dense subset $\yy^{rs}\sset\yy$
gives a {\em closed}  imbedding $\bj^{rs}: \yy^{rs}\into
\g^{rs}\times\t^r$.
According to parts (i) and (iii) of
Proposition \ref{tb} (of \S\ref{tx} below),
there is a nowhere vanishing $G$-invariant section
$\om\in\kk_{\yy^{rs}}$ of the canonical bundle.
The element  $(dx\, dt)\inv\o (\bj^{rs})_*\om$
is therefore a nonzero  $G$-invariant
section of the $\dd$-module
$\int_{\bj^{rs}}\oo_{\yy^{rs}}$
on $\g^{rs}\times\t^r$, see \S\ref{ddfilt} and \cite[p.22]{HTT}.

The radial part map $\rad$ is known to
  have the following property, \cite{HC}:
\beq{resdel}
\del_\t\cdot u(f)|_\t=(\rad u)(\del_\t\cdot f|_\t),
\quad \forall u\in \dd(\g)^G,\ f\in\C[\g]^G.
\eeq

Using the above formula and 
an equation  $pr_\g^*(dx)=(pr_\t^*\del_\t)\cdot\om$,
see Proposition \ref{tb}(iii) of \S\ref{int} below,
one verifies easily 
that the  section $(dx\, dt)\inv\o (\bj^{rs})_*\om$
 is annihilated by the ideal
$\I|_{\g^{rs}\times\t^r}$, 
Therefore, similarly to \cite[\S4.7]{HK1},
we conclude that
 the assignment $\dd_{\g^{rs}\times\t^r}\to
\int_{\bj^{rs}}\oo_{\yy^{rs}},\
1\mto (dx\, dt)\inv\o (\bj^{rs})_*\om,\ $
descends to a well defined nonzero $\dd$-module morphism
\beq{h} \mm|_{\g^{rs}\times\t^r}\ =\
(\DD/\I)|_{\g^{rs}\times\t^r}
\ \too\ \mbox{$\int_{\bj^{rs}}$}\oo_{\yy^{rs}}.\eeq

We conclude that $\mm\neq0$. It follows,
as we have shown above, that $\mm$ is a simple $\dd$-module.
Hence, the map \eqref{h} must be an isomorphism.
Moreover, we have
$\mm\cong
\bj^{rs}_{!*}\oo_{\yy^{rs}}= \bj_{!*}\oo_{\yr}$.
Part (ii) of Lemma \ref{supp} follows.

To complete the proof of part (i), we observe
that the section $\om$ provides
a trivialization of the
canonical bundle $\kk_{\yy^{rs}}.$
This implies, thanks to the isomorphism in \eqref{h}
and formula \eqref{shiftgr}, that
we have
$(\ggr\mm)|_{\xx^{rs}}\cong\oo_{N_{\yy^{rs}}}$.
Hence,  any function
$f\in \gr I$
viewed as a function
on $\gg\times\tt$
vanishes on the set $\xx^{rs}.$
The set $\xx^{rs}$
being Zariski dense in $\zz\times_{\zz/\!/G}\ \tt$,
by Lemma \ref{xxbasic}(i),
we deduce that the function $f$ vanishes
on the zero set of  the ideal $J$.
Hence, $f\in \sqrt{J}$
by Hilbert's Nullstellensatz.
The inclusion $\gr I\sset\sqrt{J}$ follows.
\end{proof}

\subsection{Hodge filtration on the Harish-Chandra module}\label{hc_hodge}
The minimal extension  $\bj_{!*}\oo_\yr$ has a canonical
 structure
of    Hodge  $\DD$-module, \cite{Sa}, p.857, Corollary 2.
This makes  the \hc module $\mm$ a    Hodge module
 via the isomorphism of Lemma \ref{supp}(ii).
Let $F^\hodge\mm$ be the
Hodge filtration on $\mm$ and $\ggr^\hodge\mm$ be an 
associated graded sheaf.
Observe further that the $\DD$-module $\bj_{!*}\oo_\yr$ is
isomorphic to its Verdier dual
so, we have ${\mathbb D}(\mm)\cong\mm$.

Thus, from the discussion of \S\ref{hodged} and  Lemma \ref{supp}(ii), we conclude

\begin{cor}\label{serre} The
sheaf $\ggr^\hodge\mm$ is a 
Cohen-Macaulay coherent $\oo_{\gg\times\tt}$-module
which is isomorphic to its
 Grothendieck-Serre dual, up to a shift.
In addition,  we have $\supp(\ggr^\hodge\mm)=\xx$.
\end{cor}

\begin{rem}\label{shift} The normalization of the Hodge filtration  that we
use in this paper  differs by a degree shift from the
one used by Saito \cite{Sa}.

Our normalization is determined
by the requirement that, for any closed immersion $f: X\into Y$,
of smooth varieties,
the Hodge filtration on the
right $\dd_Y$-module $\int^R_f\kk_X$
be equal to the order filtration introduced
at the end of section  \ref{ddfilt}.
Thus, we have $F^\hodge_{-1}\big(\int^R_f\kk_X\big)=0$,
which is not the case
 in  Saito's normalization,
see \cite[formula (3.2.2.3) and Lemma 5.1.9]{Sa} and 
\cite[p.\,222]{HTT}.

Degree shifts clearly do not
affect the validity of  Corollary \ref{serre}.
\erem

In the previous subsection,
we have considered the order filtration 
on the \hc module $\mm$. 
The isomorphism in
\eqref{h} implies, in view of Remark \ref{shift}, that the order and
the Hodge filtrations agree
on the open dense subset $\g^{rs}\times\t^r$.
We do not know if these two  filtrations agree on the
whole of $\g\times\t$.

The following result, to be proved in \S\ref{pf},  provides a partial answer.

\begin{prop}\label{ordhodge} With our normalization of
the Hodge filtration on $\mm$, for any $ k\geq 0$,
one has an inclusion
$F^\ord_k\mm\sset F^\hodge_k\mm.$
\end{prop}

According to Lemma \ref{supp}, the isomorphism  of Theorem 
\ref{mthm}
fits into the following chain of  morphisms of
coherent sheaves on $\gg\times\tt$:
$$
\xymatrix{
\oo_{\zz\times_{\zz/\!/G}\ \tt}\ 
\ar@{->>}[r]&\ \ggr^\ord\!\mm\ \ar@{->>}[r]&
\ \oo_\xx\ \ar@{^{(}->}[r]&
\ \psi_*\oo_{\xx_\norm}\ \ar@{=}[rr]^<>(0.5){\text{Theorem \ref{mthm}}}&&
\ \ggr^\hodge\mm.
}
$$

\subsection{Outline of proof of Theorems \ref{mthm}
and \ref{main_thm}}\label{main_res} 
We begin with the following
standard result.
\begin{lem}\label{ei} Let $X$ be 
 an irreducible  scheme  and $\psi:\ X_\norm\to
X_\red$ the normalization map.
Let $j:\ U\into X$
be a Zariski open imbedding such that $U$ is smooth
and the set $X\sminus U$ has
codimension $\geq 2$ in $X$.
Let $\CF$ be a  Cohen-Macaulay
sheaf on $X$ such that $j^*\CF\cong \oo_U$.
 
 Then, there are natural isomorphisms of $\oo_X$-modules
\beq{dr}\CF\ \cong\  j_* \oo_U\ \cong\  \psi_*\oo_{X_\norm}.
\eeq
\end{lem}

\begin{proof} Let $Z=X\sminus U$ and let $\H^k_Z(\CF)$
 denote the $k$th cohomology sheaf of $\CF$ with support in $Z$. 
One has a standard long exact sequence of
cohomology with support
$$0\too \H_Z^0(\CF)\too \CF \stackrel{a}\too j_* j^*\CF
 \too \H^1_Z(\CF)\too\ldots$$
where $j_*$ stands for
a (nonderived) sheaf theoretic push-forward
and $a$ is the canonical adjunction morphism.

A maximal Cohen-Macaulay
sheaf has no nonzero torsion subsheaves, \cite[\S 21.4]{E}.
Therefore, the sheaf $\CF$ is actually an $\oo_{X_\red}$-module.
In addition, we have $\H^0_Z(\CF)=0$.
Further,
a well known general result says that,  for any
maximal  Cohen-Macaulay sheaf $\CF$ on $X$ and any closed subscheme $Z\sset X$,
one has a vanishing
 $\H^k_Z(\CF)=0$ for all $k<\dim X-\dim Z$
(the result can be reduced to the
case of a local ring where it follows e.g.
from \cite[Theorem A.4.3]{E}). Applying this result
in our case  and using
the codimension $\geq 2$ assumption of the
lemma, we get that $\H^1_Z(\CF)=0$.
Thus, the morphism $a$ in the   long exact sequence
above  is an isomorphism. We deduce  that
$\CF\cong j_* j^*\CF\cong j_* \oo_U,$ the first
isomorphism in \eqref{dr}.

Observe next that the algebra structure on $\oo_U$
makes $j_*\oo_U$ an $\oo_{X_\red}$-algebra.
Let $Y:=\Spec(j_*\oo_U)$ be the relative
spectrum of that algebra. Thus, $Y$ is a scheme equipped with a 
morphism $f: Y\to {X_\red}$  that restricts to an isomorphism
$f:\ f\inv(U)\iso U;$ by definition,
we have $j_* \oo_U=f_*\oo_Y$.
The scheme  $Y$ is   reduced and irreducible,
since the algebra $j_*\oo_U$ clearly
has no zero divisors. The codimension  $\geq 2$ assumption
implies that $j_*\oo_U$ is a coherent $\oo_{X_\red}$-module
hence $f$ is a finite birational morphism.
Hence, the set $Y\sminus f\inv(U)$
has codimension  $\geq 2$ in $Y$.
We conclude that the  scheme $Y$ is smooth in codimension
1 and that it is  Cohen-Macaulay,
thanks to the
isomorphism $\CF\cong j_* \oo_U.$ 
It follows, by Serre's criterion, that
$Y$  is a normal variety; moreover,  $f=\psi$ is the normalization
map
so, we have $j_* \oo_U=f_*\oo_Y=\psi_*\oo_{X_\norm}$.
\end{proof}

We need the following definition,  motivated in
part by \cite[Lemma 3.6.2]{Ha1}. 

\begin{defn}\label{U}
Let $\zz_i,\ i=1,2,$ be the set of pairs $(x_1,x_2)\in\zz$
such that  $x_i$ 
is a regular element of $\g$.
 Put $\zz^{rr}=\zz_1\cup \zz_2$.
\end{defn}

Clearly, each of the sets $\zz_i,\ i=1,2,$ is an open subset of $\zz^r$.
Thus, $\zz^{rr}$ is an open subset of $\zz$ which is contained in the smooth locus of
$\zz$.
Furthermore, in
\S\ref{ss3} we  will  prove

\begin{lem}\label{irr} The set 
$\zz\sminus \zz^{rr}$  has codimension
$\geq 2$ in $\zz$.
\end{lem}

\begin{cor}\label{xxcor} The set $ \xx^{rr}:=p \inv(\zz^{rr})$
 is a  smooth
Zariski open  subset of  $\xx$; furthermore, the set 
$\xx\sminus \xx^{rr}$  has codimension
$\geq 2$ in $\xx$.
\end{cor}
\begin{proof}[Proof of Corollary]
Let $\xx_i:=p \inv(\zz_i),\ i=1,2$.
According  to  Lemma \ref{conormal},
we have $\xx_1=N_\yr$, a Zariski open subset
contained in the smooth locus of $\xx$.
By symmetry, the set $\xx_2$ is
contained in the smooth locus of $\xx$ as well.
Therefore, $ \xx^{rr}=\xx_1\cup \xx_2$
 is a  smooth
Zariski open  subset of  $\xx$. Furthermore,
the map $p$ being finite, it follows from
Lemma \ref{irr} that the set 
$\xx\sminus \xx^{rr}$  has codimension
$\geq 2$ in $\xx$.
\end{proof}

Write $j:\ \xx^{rr}\ \into\ \xx$
for the   open
imbedding. A key role in the proof of
Theorem   \ref{mthm}
is played by the following result 

\begin{prop}\label{res_lem}
There is a natural isomorphism $j^*(\ggr^\hodge\mm)\cong\oo_{\xx^{rr}}$.
\end{prop}

The proof of this proposition will occupy most of sections \ref{int}
and \ref{sympl_sec}. 
Our approach is based on the Hotta-Kashiwara construction of
$\mm$ via the Springer resolution, see \S\ref{xyzsec}.

\begin{rem}\label{fourier} We observe that the
isomorphism
$\mm=\bj_{!*}\oo_\yr$ of Lemma \ref{supp}
and the fact that the  canonical bundle
$\kk_\yr$ is trivial (see Proposition \ref{tb}(i), (iii))
imply
an isomorphism $(\ggr^\hodge\mm)|_{N_\yr}\cong \oo_{N_\yr}$ 
of coherent sheaves, see \eqref{shiftgr}. 
Note further that we have $\xx_2=\sigma(\xx_1)$
where $\sigma$ is the following involution
\beq{inv}
\sigma:\en \gg\times\tt\ \leftrightarrow\ \gg\times\tt,
\quad
(x_1,x_2) \times (t_1,t_2)
\ \leftrightarrow\ (x_2,x_1) \times (t_2,t_1).
\eeq

Therefore, it would be tempting to try to deduce 
 Proposition  \ref{res_lem} from  the 
 isomorphism $(\ggr^\hodge\mm)|_{N_\yr}\cong \oo_{N_\yr}$
by proving an isomorphism
  $\sigma^*(\ggr^\hodge\mm)\cong \ggr^\hodge\mm$.
Such an approach is motivated by the observation
that  the $\DD$-module $\mm$ is isomorphic to its
 Fourier transform in the sense of $\dd$-modules,
which is an immediate consequence of Definition \ref{mm_def}.
The functor  $\sigma^*:\ \coh T^*(\g\times\t)\to \coh T^*(\g\times\t)$
may be viewed as a
 `classical analogue' of the  Fourier transform  of $\dd$-modules
on $\g\times\t$.
Thus, 
 Proposition  \ref{res_lem}
would follow from the
invariance of $\mm$ under
 the  Fourier transform, had we known a general result saying that
the functor $\ggr^\hodge(-)$ 
on ($\C^\times$-monodromic) Hodge $\dd$-modules
on a vector space, 
commutes  with Fourier transform.
Unfortunately,  such a result is not
available at the moment of writing of this paper.
Indeed, this seems to be a  very difficult question 
(we are grateful to C. Sabbah for information on this subject).
\erem

\begin{proof}[Proposition \ref{res_lem} and Lemma \ref{irr} imply Theorem \ref{mthm}]
Let $\J^\hodge\sset \oo_{\gg\times\tt}$ be the annihilator
of  $\ggr^\hodge\mm$ viewed as an   $\oo_{\gg\times\tt}$-module
and  let $X\sset \gg\times\tt$
be a closed subscheme defined by the ideal $\J^\hodge.$
By Corollary \ref{serre}, the sheaf $\ggr^\hodge\mm$ is  Cohen-Macaulay
and, set theoretically, one has $X=\xx$.
Further, the isomorphism
$\mm\cong\bj_{!*}\oo_\yr$ of Lemma \ref{supp}
implies that the ideal $\J^\hodge$ is generically reduced.
Hence, $X$ is reduced and we have  $X=\xx$, since
$\xx$ is reduced by definition.

Thus, thanks to Corollary \ref{xxcor}, we are in the setting  of
Lemma \ref{ei}, where we take $\CF=\ggr^\hodge\mm$
and let $j:\ U=\xx^{rr}\ \into\ X=\xx$
be the   open
imbedding. 
We see that
Theorem \ref{mthm} is a direct consequence of
Lemma \ref{ei} combined with Proposition 
\ref{res_lem}.
\end{proof}

\begin{proof}[Theorem \ref{mthm} implies Theorem \ref{main_thm}]
 By Theorem \ref{mthm}, we have $\ggr^\hodge\mm\cong
 \psi_*\oo_{\xx_\norm}$.
 Further, thanks to Corollary \ref{serre},
we know that the sheaf $\ggr^\hodge\mm$ is Cohen-Macaulay
and, moreover, it is isomorphic to its
Grothendieck-Serre dual, up to a shift.
Thus, the sheaf $\psi_*\oo_{\xx_\norm}$ has
similar properties.
Since, Grothendieck's duality commutes with finite
morphisms, we deduce that $\oo_{\xx_\norm}$
is a  Cohen-Macaulay sheaf which is, moreover,
isomorphic to its Grothendieck-Serre dual,
that is, to the dualizing sheaf $\kk_{\xx_\norm}$, up to a shift.
Therefore, we have $\oo_{\xx_\norm}\cong\kk_{\xx_\norm}$,
completing the proof.
\end{proof}

\section{Springer resolutions}\label{int}

\subsection{An analogue of the Grothendieck-Springer
resolution}\label{tx} 
Let $\bb$ be the flag variety, the variety
of all Borel subalgebras $\b\sset\g$.
Motivated by Grothendieck and Springer,
we introduce the following incidence varieties
$$
\tg:=\{(\b,x)\in\bb\times\g\mid x\in\b\},\en\text{resp.}\en
\tgg:=\{(\b,x,y)\in\bb\times\g\times\g\mid
x,y\in\b\}.$$

The first projection
  makes $\tg$, resp. $\tgg$,
a sub vector bundle of the trivial vector bundle
$\bb\times\g\to\bb$, resp. $\bb\times\gg\to\bb$. 
Given a  Borel subgroup  $B\sset G$ with Lie algebra $\b=\Lie B$,
we have $\bb\cong G/B$. That gives
 a $G$-equivariant vector bundle isomorphism 
$\tg\cong G\times_{B}\b,$ resp. $\tgg\cong G\times_{B}(\b\times\b)$.
Thus, $\tg$ and $\tgg$ are smooth connected varieties.

Recall that, for any pair $\b,\b'$ of  Borel
subalgebras
of $\g$, there
is a canonical isomorphism $\b/[\b,\b]\cong\b'/[\b',\b']$,
cf. eg. \cite[Lemma 3.1.26]{CG}.
Given a  Cartan subalgebra $\t\sset\b$,
the composite  $\t\into\b\onto \b/[\b,\b]$
yields an isomorphism $\t\iso\b/[\b,\b]$.
Therefore, the  assignment $(\b',x)\mto x{\,\mathrm{mod}\,[\b',\b']}\in\b'/[\b',\b'],$
resp. $(\b',x,y)\mto  (x{\,\mathrm{mod}\,[\b',\b']},\
y{\,\mathrm{mod}\,[\b',\b']})
\in\b'/[\b',\b']\times\b'/[\b',\b'],$
gives
a well defined smooth morphism $\nu:\ \tg\to\h$,
resp. $\nnu:\ \tgg\to\hh$.

Finally, we introduce a projective $G$-equivariant morphism
 $\mu:\ \tg\to\g,\ 
(\b,x)\mto x$,
resp. $\mmu:\ \tgg\to \gg,\ (\b,x,y)\mto(x,y)$.
The map $\mu$ is known as  the
{\em Grothendieck-Springer resolution}.

%Write $\T_X$ for the tangent sheaf on a smooth variety $X$.

\begin{prop}\label{tb}
\vi The image of the map $\mu\times\nu:\tg\to\g\times\t$ is
contained in $\yy=\g\times_{\g/\!/G}\t$.
The resulting morphism
$\pi:\tg\to\yy$ is a  resolution 
of singularities, so $\pi\inv(\yr)\iso\yr$
is an isomorphism.

\vii The image of the map $\mmu\times\nnu:
\tgg\to\gg\times\tt$ is
contained in $\gg\times_{\gg/\!/G}\tt$.
The resulting map 
 $\tgg\onto[\gg\times_{\gg/\!/G}\tt]_\red$
is a proper birational morphism.

\viii The canonical bundle on $\tg$ has a natural trivialization
by a nowhere vanishing $G$-invariant section $\omega\in \kk_\tg$ such that one has
$\mu^*(dx)=(\nu^*\del_\t)\cdot
\omega$, cf. Notation \ref{dx}.

\iv We have $\kk_{\tgg}=\wedge^{\dim\bb}\ {\mathfrak q}^* \T_\bb$
(here $\T_\bb$ is the tangent sheaf
and  ${\mathfrak q}:\ \tgg\to\bb$ is the  projection).
\end{prop}

Part (i) of Proposition \ref{tb}  is well-known,
cf. \cite{BB}, \cite{CG}. The 
descriptions of canonical bundles in 
parts (iii)-(iv) are straightforward.
The equation of part (iii) that involves
the section $\omega$
appears eg. in  \cite{HK1}, formula (4.1.4).
This equation is, in essence, nothing but Weyl's 
classical integration formula.

Part (ii)  of Proposition \ref{tb} is an immediate consequence
of  Proposition
\ref{nonormal}, of \S\ref{ss1} below.
Lemma \ref{lss} implies, in particular, that the image of the
map $\mmu$ equals the set of pairs $(x,y)\in\gg$ such that
$x$ and $y$ generate a solvable Lie subalgebra of $\g$.
We remark also that the statement of Proposition \ref{tb}(ii) is a variation of
 results concerning the
null-fiber of the adjoint
quotient map $\gg\to\gg/\!/G$, see
\cite{Ri2}, \cite{KW}.
\qed

\begin{quest} Is  the variety $[\gg\times_{\gg/\!/G}\tt]_\red$
normal, resp. Cohen-Macaulay ?
\end{quest}

\subsection{Symplectic geometry interpretation}\label{sympl}
The map $(\b,x)\mto (\b, x, \nu(\b, x))$ gives
 a closed immersion $\eps:\ \tg\into \bb\times\g\times\t$.
The image of this immersion  is a smooth subvariety 
$\eps(\tg)\sset \bb\times\g\times\t$. 
 Let  $\La$
be the total space of the
conormal bundle of that subvariety.
 Thus,
 $\La$
is a smooth $\C^\times$-stable Lagrangian subvariety
in $T^*(\bb\times\g\times\t)=T^*\bb\times(\gg\times\tt)$.
Let $\pr_{\La\to T^*\bb}:\ \La\to
T^*\bb$, resp. $\pr_{\La\to\gg\times\tt} 
:\ \La\to\gg\times\tt,$
denote the restriction to $\La$ of
the projection of $T^*\bb\times\gg\times\tt$
to the first, resp. along the first, factor.

Let $\N$ be  the variety of nilpotent elements
of $\g$ and put
$\wt\N:=
\{(\b,x)\in\bb\times\g\mid x\in[\b,\b]\}$.
There is
 a natural isomorphism $T^*\bb\cong\wt\N$
of $G$-equivariant vector bundles on $\bb$
that identifies the cotangent space
 at a point $\b\in\bb$ with the vector space
$(\g/\b)^*\cong [\b,\b]$,
cf. \cite[ch. 3]{CG}. Let $\Phi:\ \wt\N\iso T^*\bb$ be an
isomorphism obtained by composing the above isomorphism
 with the sign involution
along the fibers of the vector bundle ~$T^*\bb$.

Restricting the commutator map $\kap$
to each Borel subalgebra $\b\sset\g$ yields
 a morphism
$\kkap:\ \tgg\to \wt\N,\
(\b,x,y)\mto (\b,\,[x,y])$.
Let $\Psi$ denote the map
$(\Phi\ccirc\kkap)\times\mmu\times\nnu:\ \tgg\to
T^*\bb\times\gg\times\tt$.

\begin{prop}\label{conorm} The map $\Psi$ yields
an isomorphism $\tgg\iso \La$ that fits into a
 commutative diagram
\beq{psi}
\xymatrix{
\gg\times\tt\ \ar@{=}[d]^<>(0.5){\Id}&&
\ \tgg\ \ar@{=}[d]^<>(0.5){\Psi}
\ar[rr]^<>(0.5){\kkap}\ar[ll]_<>(0.5){\mmu\times\nnu}
&&\ \wt\N\ \ar@{=}[d]^<>(0.5){\Phi}\\
\gg\times\tt\ &&\ \La\ \ar[rr]^<>(0.5){\pr_{\La\to T^*\bb}}
\ar[ll]_<>(0.5){\pr_{\La\to\gg\times\tt}}&&
\ T^*\bb
}
\eeq
\end{prop}

\begin{proof}
Fix $\b\in\bb$ and $x\in\b$.
So,
$(\b,x)\in\tg$
and the corresponding point in $\bb\times\g\times\h$
is given by the triple $ u=(\b,\,x,\,x\mo)$.
The fiber of the tangent
bundle $\TT(\bb\times\g\times\h)$
 at the point
$ u$ may be identified with the vector
space $\TT_u(\bb\times\g\times\h)=(\g/\b)\times\g\times\h$.
Hence,
the tangent space to the submanifold
$\eps(\tg)$ equals
$$\TT_u(\eps(\tg))=
\big\{(\alpha\,\operatorname{mod}\b,\,[\alpha,x]+\beta,\,\beta\mo)\
\in\ 
(\g/\b)\times\g\times\h\enspace\big|\en \alpha\in\g,\ \beta\in\b\big\}.
$$

Now, write $\langle-,-\rangle$ for an invariant bilinear form on $\g$
and use it to identify the fiber  of the cotangent
bundle  $\TT^*(\bb\times\g\times\h)$ at 
$ u$  with the vector
space $[\b,\b]\times\g\times\h$.
Let $( n, y,h)\in [\b,\b]\times\g\times\h$
be a point of that vector space.
Such a point belongs to $\La_u$,
 the fiber at $u$ of the conormal bundle on the subvariety  $\eps(\tg)$,
if and only if the following equation holds
\beq{eq}
\langle\alpha, n\rangle+\langle[\alpha,x]+\beta,\, y\rangle+\langle\beta\mo,\ h\rangle=0
\quad\forall \alpha\in\g,\,\beta\in\b.
\eeq

Taking $\alpha=0$ and applying 
equation \eqref{eq}  we get $\langle[\b,\b],\, y\rangle=0$.
Hence,
$ y\in\b$ and  $h=-y\mo$.
Next, for any $\alpha\in\g$,
we have $\langle[\alpha,x],\, y\rangle=\langle\alpha,\,[x, y]\rangle$.
Hence, for $\beta=0$ and any $ n\in[\b,\b],\  y\in\b$,
equation \eqref{eq} gives
$$0=\langle\alpha, n\rangle+\langle[\alpha,x], y\rangle=\langle
\alpha,\, n+[x, y]\rangle\quad\forall\alpha\in\g.$$
It follows that $ n+[x, y]=0$. We conclude
 that 
\beq{qq}
\La_u=\{(-[x,y],\,y,\,y\mo)\in[\b,\b]\times\g\times\h,\quad y\in\b\}.
\eeq

We have a projection
$\varpi:\ \tgg\to\tg,\ (\b,x,y)\to (\b,x)$ along the last factor.
The vector space in right hand side of \eqref{qq} is equal to
the image of the set $\varpi\inv(\b,x)$ under the map
$\Psi=
(\Phi\ccirc\kkap)\times\mmu\times\nnu$.
We conclude that the map $\Psi$ gives
an isomorphism of vector bundles $\tgg\to \tg$
and $\La\to \eps(\tg),$ respectively.
\end{proof}

\subsection{The scheme $\tx$}\label{tx_sec}
We use the notation
$\iim:\ X\into T^*X$ for the zero section
of the cotangent bundle on a variety $X$.

We consider the following commutative diagram
where the vertical map $\mu|_{_{\wt\N}}$ is known as the Springer resolution,
\beq{kk}
\xymatrix{
\bb=\mu\inv(0)\ \ar[d]^<>(0.5){\mu} \ar@{^{(}->}[rr]^<>(0.5){\iim:\
\b\,\mto(\b,0)}&&\ 
\wt\N\
\ar@{->>}[d]^<>(0.4){\mu|_{_{\wt\N}}}&&\ \tgg\ \ar[ll]_<>(0.5){\kkap}
\ar[d]^<>(0.5){\mmu}\ar[rr]^<>(0.5){\nnu}&&\ \tt \ar[d]\\
\{0\}\  \ar@{^{(}->}[rr]&&
\ \NN\ \ar@{^{(}->}[r]&\ \g\ &\ \gg\ \ar[l]_<>(0.5){\kap}\ar[rr]&&\
\gg\times_{\gg/\!/G}\tt
}
\eeq

Let $\tx :=\kkap\inv\big(\iim(\bb)\big)\sset \tgg$, a scheme theoretic
preimage
of the zero section. Set theoretically, one has
\beq{txform}\tx =
\{(\b,x,y)\in\bb\times\g\times\g\mid
x,y\in \b,\en [x,y]=0\}.
\eeq
Diagram \eqref{kk} shows that the morphism $\mmu$ maps $\tx$ to $\zz$,
resp.
 $\mmu\times\nnu$ maps $\tx$ to
$\zz\times_{\zz/\!/G}\hh.$
It follows from Proposition \ref{conorm}
 that the map $\Psi$ induces
an isomorphism of schemes
$$\tx=\kkap\inv\big(\iim(\bb)\big)\en
\stackrel{\Psi}\cong\en
\pr_{\La\to T^*\bb}\inv\big(\iim(\bb)\big)=\La\,\cap\,(\bb\times
\gg\times\tt)$$
where the scheme structure on each side is that of a
scheme theoretic preimage.

Let $\wt\xx^{rr}:=\mmu\inv(\zz^{rr})$, a Zariski 
open subset of the scheme $\tx=\kkap\inv\big(\iim(\bb)\big)$.

\begin{lem}\label{codim1}  The  differential
of the morphism $\kkap: \tgg\to \wt\N$
is surjective at any point
 $(\b,x,y)\in\wt\xx^{rr};$ in particular,
the set $\wt\xx^{rr}$ is contained in the
smooth locus of the scheme $\tx$.
\end{lem}
\begin{proof}
Let $x\in\b$ and write $\ad_\b x$ for the map
$\b\mto [\b,\b], u\mto
[x,u]$. One has $\dim\b-\dim\g_x\leq\dim\b-\dim\Ker(\ad_\b x)=
\dim\im(\ad_\b x)\leq\dim[\b,\b]$.
For $x\in\b\cap\g^r,$ 
we have $\dim\g_x=\dim\t$ and  the above inequalities
yield $\dim[\b,\b]=\dim\b-\dim\g_x\leq\dim(\im\ad_\b x)
\leq\dim[\b,\b]$. Thus, in this case, 
$\im(\ad_\b x)=[\b,\b]$ i.e. the map $\ad_\b x$ is surjective.

Now, let $(\b,x,y)\in\wt\xx^{rr}.$ 
Without loss of generality,
we may assume that  $x$ is a regular element of $\g$.
The differential
of the commutator map $\kap:\
\b\times\b\to[\b,\b]$ at the point $(x,y)\in \b\times\b$
is a linear map $d_{\b,x,y}\kap:\ \b\oplus\b\to[\b,\b]$
given by the formula
$d_{\b,x,y}\kap:\ (u,v)\mto \ad_\b x(u)-\ad_\b y(v)$.
We see that
$\im(\ad_\b x)\subseteq \im(d_{\b,x,y}\kap).$
By the preceding paragraph,
we deduce that the
map $d_{\b,x,y}\kap$ is surjective.
The lemma follows from this by
$G$-equivariance.
\end{proof}

\begin{rem} The scheme $\tx$ is not
irreducible, in general, cf. \cite{Ba}. So, the open set $\tx^{rr}$
 is not necessarily dense in $\tx$.
\erem

We use Proposition \ref{conorm}
to identify the set $\wt\xx^{rr}\sset\tx\sset\tgg$
with a  subset of  $\La\cap(\bb\times\gg\times\tt)$.
Recall the notation  $\xx^{rr}=p\inv(\zz^{rr})$.

\begin{cor}\label{transv} \vi The varieties
$\La$ and $\bb\times\gg\times\tt$ meet transversely
at any point of  $\wt\xx^{rr}$,
so  $\wt\xx^{rr}$ is a smooth  Zariski open subset
of $\La\cap(\bb\times\gg\times\tt)$.

\vii  The  morphism
$\ppi:\wt\xx^{rr}\to \xx^{rr},$  the restriction of the
 morphism $\mmu\times\nnu$ to the
set $\wt\xx^{rr},$
is an isomorphism of algebraic varieties.
\end{cor}
\begin{proof} Part (i) is equivalent to the statement of 
Lemma \ref{codim1} . 
To prove (ii), let $X$ denote the
preimage of $\zz^{rr}$ under
the first projection $\zz\times_{\zz/\!/G} \tt\to\zz$,
so we have $X_\red=\xx^{rr}$.
It is clear
that the  morphism
$\mmu\times\nnu$  maps $\wt\xx^{rr}$
to $X$.
We claim
that 
the resulting map 
$\ppi:\ \wt\xx^{rr}\to X$
is a set theoretic
bijection.
Indeed, it 
is surjective, by Lemma \ref{xxbasic}(i),
since the image of this map
contains the set $\zz^{rs}\times_{\zz/\!/G} \tt$
and $\mmu\times\nnu$, hence also $\ppi,$ is a proper morphism.
To prove injectivity, we interpret the
map $(\b,x,y)\mto (x,y)$ as a composition
of the imbedding $\tx\into\tg\times\g,\ (\b,x,y)\mto (\b,x)\times y$
and the map $\pi\times\Id_\g:\
\tg\times\g\to\yy\times\g$.
This last map gives a bijection
between the set $\yr\times\g$ and its preimage
in $\tg\times\g$, thanks to Proposition 
\ref{tb}(i). Our claim follows.

Recall that $\wt\xx^{rr}$ is a smooth scheme by 
 Lemma \ref{codim1}. 
Hence, the  scheme theoretic image of $\wt\xx^{rr}$
under the morphism $\ppi$ 
is actually contained 
in $\xx^{rr}=X_\red$.
The reduced scheme $\xx^{rr}$ is
smooth by Corollary \ref{xxcor}.
Thus, 
$\ppi: \wt\xx^{rr}\to \xx^{rr}$ is a   morphism
of smooth varieties,
which is a set theoretic bijection.
Such a morphism is necessarily an isomorphism,
by  Zariski's main theorem,
and part (ii) follows.

There is an alternative proof that the
morphism $\ppi: \wt\xx^{rr}\to \xx^{rr}$
is \'etale 
based on symplectic geometry. In more detail,
put $X=\bb$ and $Y=\g\times\t$
and let $\eps: \tg\into X\times Y$ be 
the imbedding, cf. \S\ref{sympl}.
We have smooth locally closed subvarieties
$\yr\sset \g\times\t,\ \pi\inv(\yr)\sset\tg$
and
$Z:=\eps(\pi\inv(\yr))\sset X\times Y$,
respectively. We use  Proposition \ref{conormal}, resp.
 Proposition \ref{conorm},
to identify  $\xx_1$ with $N_\yr$,
resp. $\ppi\inv(\g^r\times\g)\cap\La$ with $N_Z$.

We know that
the projection $X\times Y\to Y$
induces an isomorphism $Z\iso\yr,$
by Proposition \ref{tb}(i).
Now, one can prove a general result saying 
that, in this case, the map $(X \times T^*Y)\cap N_Z\
\to\ N_{\yr}$
induced by the projection
$T^*X\times T^*Y\to T^*Y$
is \'etale at any point
where  $N_Z$
meets the subvariety $X \times T^*Y\sset T^*(X\times Y)$
transversely. The latter condition
holds in our case
thanks to part (i) of the Corollary.
This implies the isomorphism
$\ppi:\ \ppi\inv(\zz_1)\iso \xx_1$.
The isomorphism $\ppi:\ \ppi\inv(\zz_2)\iso \xx_2$
then follows by symmetry.
\end{proof}

\subsection{A DG algebra}\label{dima_sec} In this subsection, we
construct
a sheaf $\aa$ of DG $\oo_{\tgg}$-algebras  such that
$\H^0(\aa)$,
the zero cohomology sheaf, is isomorphic
to the structure sheaf of the closed subscheme $\tx\sset\tgg.$
To explain the construction,
let  $\T:=\T_\bb$ and write $\iim: \bb\into T^*\bb$ for the
zero section, resp. $q: T^*\bb\to \bb$
for the projection. 

The sheaf $q^*\T^*$ on $T^*\bb$
comes equipped with  a canonical Euler section  $\eu$
such that, for each  covector $\xi\in T^*\bb,$ the value of $\eu$ at the
point
$\xi$ is equal to $\xi$.
Further, 
there is  a standard   Koszul complex $\ldots\to  \wedge^3q^*\T\to\wedge^2q^*\T\to
q^*\T\to0$
with differential $\partial_\eu$
 given  by contraction  with $\eu$.
The  complex $(\wedge^\hdot q^*\T,\ \partial_\eu)$ is a  locally
free resolution of the sheaf $\iim_*\oo_\bb$ on $T^*\bb$.

We may
use the isomorphism $T^*\bb\cong\wt\N$ to view the morphism $\kkap$
as a map $\tgg\to  T^*\bb$.
Hence, the  pull-back of the Koszul
complex above via the map $\kkap$
is a complex of locally free sheaves on $\tgg$
that represents the
object $L\kkap^*(\iim_*\oo_\bb)\in\dcoh(\tgg)$.

Let  ${\mathfrak q}: \tgg\to\bb$
denote  the first   projection, so we have
$\kkap^*q^*\T={\mathfrak q}^*\T$
and one may identify $\kkap^*\eu$, the pull-back of the section $\eu$,
 with a section of ${\mathfrak q}^*\T^*.$
For each $n\geq 0,$ let $\aa_n=
\wedge^n\,{\mathfrak q}^*\T=\wedge^n\,\kkap^*q^*\T$,
where  $\aa_0:=\oo_{\tgg}$.
Contraction with $\kkap^*\eu$ gives
a  differential
$\partial_{\kkap^*\eu}:\,\aa_\hdot\to\aa_{\idot-1}$.
Thus, we may (and will) view $\aa:=\bplus_{n\geq 0}\aa_n$
as a sheaf of coherent DG $\,\oo_\tgg$-algebras, with multiplication given by the
wedge product and with the differential
$\partial_{\kkap^*\eu}$,
a graded derivation of degree $(-1)$.

\begin{notation} Write $\H^j({\mathcal F})\in\coh X$
for the $j$th cohomology sheaf of an object
${\mathcal F}\in \dcoh( X)$.
\end{notation}

By construction, one has
 $\H^0(\aa,\partial_{\kkap^*\eu})=\oo_{\tx}$,
cf. \eqref{txform}.
Thus, one may view $\aa$ as the structure sheaf
of a certain DG scheme, a "derived analogue" of the
scheme $\tx$.

The DG algebra $\aa$ has also appeared, in an implicit
form, in a calculation
in \cite{SV}.

\begin{rem}\label{adual}
The DG algebra 
 $\aa$ is concentrated in degrees
$0\leq i\leq d:=\dim \bb$ and we have $\aa_d=\kk_\tgg$, by
 Proposition \ref{tb}(iv). 
It follows that $\aa$ is a  {\em self-dual} DG algebra in the sense
that
 multiplication in $\aa$ yields
an isomorphism of complexes
$$
\aa_{d-\hdot}\iso\hhom_{_{\oo_{\tgg}}}(\aa_\idot, \aa_d)=
 \hhom_{_{\oo_{\tgg}}}(\aa_\idot, \kk_{\tgg}).
$$
\end{rem}
\section{Proof of the main theorem}\label{sympl_sec}
\subsection{The Hotta-Kashiwara construction}\label{hk}
The structure sheaf $\oo_\tg$ has an obvious  structure of
holonomic left $\dd_\tg$-module. So, 
one has $\int_{\mu\times\nu}\oo_\tg$, 
 the direct image of this $\dd_\tg$-module 
via the projective morphism $\mu\times\nu$,
cf. \S\ref{tx}.
Each cohomology group
$\H^k(\int_{\mu\times\nu}\oo_\tg)$ is a holonomic
$\DD$-module  set theoretically
supported on the variety $\yy\sset \g\times\t$.

 Hotta and Kashiwara proved the following important result,
\cite[Theorem 4.2]{HK1}.

\begin{thm}\label{pimm} For  any
$k\neq 0$, we have
$\H^k(\int_{\mu\times\nu}\oo_\tg)=0$;
furthermore,
there is a natural
isomorphism
$\H^0(\int_{\mu\times\nu}\oo_\tg)
\cong\bj_{!*}\oo_\yr\ $ of
$\DD$-modules.
\end{thm}

The $\DD$-module
$\H^0(\int_{\mu\times\nu}\oo_\tg)$
has a canonical nonzero section $(dx\, dt)\inv\bo(\mu\times\nu)_*\om$,
cf.  Proposition
\ref{tb}(iii).  Hotta and Kashiwara \cite[\S 4.7]{HK1} show  that
this section is annihilated by the  left ideal
$$
\I':=\DD\cdot(\ad\g\o 1)+\DD\cdot\{P-\rad(P)\mid P\in \C[\g]^G\}+
\DD\cdot\{Q-\rad(Q)\mid Q\in(\sym\g)^G\}
$$
(we have exploited their argument in the proof of Lemma 
\ref{supp}).

Furthermore,  it is proved in \cite[\S 4.7]{HK1} that
$\DD/\I'$ is a simple $\DD$-module and
 the assignment $u\mto u[(dx\,
dt)\inv\bo(\mu\times\nu)_*\om]$ yields a
$\DD$-module isomorphism
$\DD/\I'\iso \H^0(\int_{\mu\times\nu}\oo_\tg).$
 An alternative, more direct, proof of
an analogous isomorphism may be found in \cite{HK2}.

\begin{rem}\label{I=I} Hotta and Kashiwara 
defined the \hc module as the quotient
$\DD/\I'$. Comparing
formula \eqref{M} with the definition of the
ideal $\I'$ we see that
one has  $\I'\subseteq\I$.
This  inclusion of left ideals yields a surjection $\DD/\I'\onto\DD/\I=\mm$.
Since $\DD/\I'$ is a simple $\DD$-module,
this surjection must be an isomorphism.
So, Definition \ref{mm_def} is equivalent, {\em a posteriori},
to the one used by  Hotta and Kashiwara.
\end{rem}

\subsection{Direct image for filtered $\dd$-modules}\label{filt_sec}
Let  $f:\ X\to Y$ a  proper morphism.
The direct image functor $\int_f$ can be upgraded to
a functor
$D^b_{coh}({{\mathsf F}\dd_X})\to
D^b_{coh}({{\mathsf F}\dd_Y}),$
$(M,F)\mto \int_f(M,F)$
between filtered derived categories, cf. \cite[\S4]{La},
 \cite[\S2.3]{Sa}.
The latter functor is known to commute with the
associated graded functor $\dgr(-)$.
We will only need a special
case of this result for  maps of the form
$f: X\times Y\to Y,$ the  projection along
a proper variety $X$. 
In such a case, one has a diagram
\beq{fdiag}
\xymatrix{
T^*Y\ &&\ X\times T^*Y\ \ar@{->>}[ll]_<>(0.5){\vvpi }
\ar@{^{(}->}[rrr]^<>(0.5){\vsi=\iim\times\Id_{T^*Y}}
&&&\ T^*X\times T^*Y=T^*(X\times Y).
}
\eeq
Here, $\vvpi$  denotes the second projection and
 $\iim$ is the zero section. 

The relation between the functors $\dgr(-)$ and $\int_f$
is provided by the following result,
see \cite{La}, formula (5.6.1.2).
\begin{thm}
\label{fgr} Let $f: X\times Y\to Y$ be the
second projection where
$X$ is proper.
Then, for any $(M,F)\in D^b_{coh}({{\mathsf F}\dd_{X\times Y}})$,
in $\dcoh(T^*Y)$ there is a functorial
isomorphism
$$\dgr\mbox{$\left(\int_f (M,F)\right)$}=
R\vvpi _*\big[(\kk_X\boxtimes \oo_{T^*Y})\ 
\bo_{\oo_{X\times T^*Y}}
L\vsi^*
\dgr(M,F)\big].
\eqno\Box$$
\end{thm}

The cohomology sheaves $\H^\hdot\big(\int_f(M,F)\big)$
 of the filtered complex $\int_f(M,F)$
come equipped 
with an induced filtration. However,
Theorem \ref{fgr} is not sufficient, in general, for
describing $\ggr\H^\hdot\big(\int_f(M,F)\big)$,
the associated graded sheaves.
A theorem of Saito stated below says that, in
the case of Hodge modules,  Theorem \ref{fgr} 
{\em is} indeed  sufficient for that.

Let $f: X\to Y$ be a proper morphism of smooth
varieties. A  \mhm  on $X$
may be viewed as an object $ (M,F)\in D^b_{coh}({{\mathsf F}\dd_X})$,
so there is a well defined object
 $\int_f(M,F)\in  D^b_{coh}({{\mathsf F}\dd_Y})$.

One of the main results of Saito's theory 
reads,
see \cite[Theorem 5.3.1]{Sa}:
\begin{thm}\label{saito} 
For any
 $(M,F)\in \mh(X)$ and any {\em projective} morphism
  $f: X\to Y$, the filtered complex $\int_f(M,F)$
is  strict, cf. Definition \ref{strict_def}, and each cohomology group
$\H^j(\int_f(M,F))$ has the natural structure of a 
\mhm  on ~$Y$.
\end{thm}

In the situation of the theorem, we refer to
 the induced filtration on 
 $\H^j(\int_f(M,F)),\ j=0,1,\ldots$,
as the Hodge filtration
and let $\ggr^\hodge\H^j(\int_f(M,F))$
denote the associated graded coherent sheaf on $T^*Y$.
Similar notation will be used for right $\dd$-modules.

\subsection{Key result}\label{filtmmG}
We recall the setup of \S\ref{sympl}.
Thus, we have the immersion
$\eps:\ \tg\into \bb\times\g\times\t,$ $(\b,x)\mto (\b, x,\, \nu(\b, x))$
and we write $\La\sset T^*(\bb\times\g\times\t)$ for the total space of the
conormal bundle on $\eps(\tg).$
We will view the structure sheaf $\oo_\La$
as a coherent sheaf on $T^*\bb\times\gg\times\tt$
supported on $\La$.

In the special case where
$X=\bb$ and $Y=\g\times\t$ diagram \eqref{fdiag}
takes the form
\beq{cohisodiag}
\xymatrix{
\La\ \sset 
\ T^*\bb\times\gg\times\tt\ &
\ \bb\times\gg\times\tt\ \ar@{_{(}->}[l]_<>(0.5){\vsi}
\ar@{->>}[r]^<>(0.5){\vvpi}\ &\ \gg\times\tt.
}
\eeq

\begin{thm}\label{cohiso} All nonzero cohomology
sheaves of the object
$R\vvpi _*L\vsi^*
\,\oo_\La\in \dcoh(\gg\times\tt)$ vanish and, we have
$$\H^0\br{R\vvpi _*L\vsi^*
\,\oo_\La}=\ggr^\hodge\mm.$$
\end{thm}
\begin{proof}
Let $E:=\int_\eps^R \kk_\tg$,
 a  right Hodge  $\dd$-module on $\bb\times\g\times\t$.
Using the notation of  Remark \ref{shift}, from \eqref{shiftgr}, we obtain
 $\gr^\hodge E=q^*(\eps_* \kk_\tg)=
q^*(\eps_*\oo_\tg),$ where in the last equality
we have used that the canonical bundle on $\tg$
is  trivial, see Proposition \ref{tb}(iii).
Therefore,
for $\int_\eps\oo_\tg=\kk_{\bb\times\g\times\t}\inv\o E$,
the corresponding left $\dd_{\bb\times\g\times\t}$-module,  we find
\beq{gre}
L\vsi^*\big(\ggr^\hodge\br{\int_\eps\oo_\tg}\big)=
L\vsi^* q^*(\kk_{\bb\times\g\times\t}\inv\o\eps_* \oo_\tg)=
L\vsi^* q^*(\kk_\bb\inv\boxtimes\oo_{\g\times\t})=
\kk_\bb\inv\o L\vsi^*\oo_\La,
\eeq
where we have used simplified notation
$\kk_\bb\inv\o(-)=(\kk_\bb\inv\boxtimes \oo_{\gg\times\tt})\ 
\bo_{\oo_{\bb\times \gg\times\tt}}(-)$.

We factor the map $\mu\times\nu$ as
a composition of the closed imbedding $\eps$
and a proper projection $f:\ \bb\times\g\times\t\to \g\times\t$
along the first factor. We get
$\int_{\mu\times\nu}\oo_\tg=
\int_f\br{\int_\eps\oo_\tg}.$
Hence, applying Theorem \ref{fgr}
to the $\dd$-module $M=\int_\eps\oo_\tg$
and using \eqref{gre},
we obtain $\dgr\big(\int_{\mu\times\nu}\oo_\tg\big)=
\dgr\big(\int_f (\int_\eps\oo_\tg)\big)=
R\vvpi _* L\vsi^*\oo_\La.$
Thus, by  Theorem  \ref{saito}
applied to the Hodge module $\oo_\tg$,
we get
$\ggr^\hodge\H^j(\int_{\mu\times\nu}\oo_\tg)=
\H^j(\dgr\int_{\mu\times\nu}\oo_\tg)
=\H^j(R\vvpi _* L\vsi^*\oo_\La)$
for any $j\in\Z.$
We conclude
 that
$\H^j(R\vvpi _* L\vsi^*\oo_\La)=0$
for any $j\neq0$, thanks to Theorem \ref{pimm}.

Finally, we observe that
the Hodge structure
 on the minimal extention $\dd$-module
$\bj_{!*}\oo_\yr$ is  determined by the
 Hodge structure on 
$\oo_\yr$.   The morphism $\mu\times\nu:\ \tg\to \yy$
is generically an isomorphism. It follows that the
isomorphism of Theorem \ref{pimm}
respects the Hodge structures.
Therefore, the  isomorphism $\mm\cong 
\H^0(\int_{\mu\times\nu}\oo_\tg)$, which is based on
the isomorphism $\mm\cong\bj_{!*}\oo_\yr$ of Lemma \ref{supp}(ii),
also respects
the Hodge filtrations.
 We deduce that  $\ggr^\hodge\mm=\H^0(R\vvpi _* L\vsi^*\oo_\La)$.
\end{proof}

\subsection{}\label{homsec}
Let $X$ be a smooth  variety and
let $i_Y: Y\into X$,
resp. $i_Z: Z\into X$,
be closed imbeddings of smooth
subvarieties. 
Below, we will use the following simple 
\begin{lem}\label{standard}
In $\dcoh(X)$,
 there are canonical quasi-isomorphisms
\beq{xyz}(i_Y)_*Li_Y^*[(i_Z)_*\oo_Z]\ \cong\
(i_Y)_*\oo_Y\lo_{\oo_X}(i_Z)_*\oo_Z\
\cong\ (i_Z)_*Li_Z^*[(i_Y)_*\oo_Y].
\eeq
\end{lem}
\begin{proof} Let $\Delta_X: X\to X \times X$
be  the diagonal
imbedding and define a map $\Delta_{YX}:\ Y\to
Y\times X,\ y\mto (y, i_Y(y))$.
We have a cartesian diagram of closed imbeddings
$$
\xymatrix{
Y\ \ar[d]_<>(0.5){i_Y}\ar[rr]^<>(0.5){\Delta_{YX}}&&
\ Y\times X\ \ar[d]^<>(0.5){i_Y\times \Id_X}\\
X\ \ar[rr]^<>(0.5){\Delta_X}&& \ X\times X
}
$$

For any  $\ff\in\dcoh(X)$,
we have $[(i_Y)_*\oo_Y]\boxtimes\ff$ $=
(i_Y\times \Id_X)_*f^*\ff$, where
 $f:\ Y\times X \to X$ is  the second
projection. Therefore, 
we obtain
\begin{multline*}
[(i_Y)_*\oo_Y]\lo_{\oo_X}\ff=
\Delta_X^*([(i_Y)_*\oo_Y]\boxtimes\ff)=
\Delta_X^*(i_Y\times \Id_X)_*(f^*\ff)\\
=(i_Y)_*[\Delta_{YX}^*(f^*\ff)]=
(i_Y)_*(f\ccirc\Delta_{YX})^*\ff=(i_Y)_*i_Y^*\ff,
\end{multline*}
where the third isomorphism is a consequence
of proper base change with
respect to the cartesian square above.
Taking here $\ff:=(i_Z)_*\oo_Z$ yields the
isomorphism on the left of \eqref{xyz}.
The isomorphism on the right of \eqref{xyz}
is proved similarly.
\end{proof}

Associated with  diagram \eqref{kk},
there are derived functors
\beq{ddd}
 \xymatrix{\dcoh( \bb)\
\ar[r]^<>(0.5){\iim_*}&\ 
\dcoh( \wt\N)\ar[r]^<>(0.5){L\kkap^*}\
&\ \dcoh( \tgg)\ \ar[rr]^<>(0.5){R(\mmu\times\nnu)_*}
&&\ \dcoh(\gg\times\tt).}
\eeq

In \S\ref{dima_sec}, we considered
a Koszul complex 
$(\wedge^\hdot q^*\T,\ \partial_\eu)$, a  locally
free resolution of the sheaf $\iim_*\oo_\bb$ on $T^*\bb$.
Therefore, the
object $L\kkap^*(\iim_*\oo_\bb)\in \dcoh(\tgg)$
may be represented by the DG $\oo_\tgg$-algebra $(\aa,\ \partial_{\kkap^*\eu})=
\kkap^*(\wedge^\hdot q^*\T,\ \partial_\eu)$
introduced in \S\ref{dima_sec}.

The following result provides a link between the DG algebra $\aa$
and Theorem \ref{cohiso}.

\begin{lem}\label{amu} In $\dcoh(\gg\times\tt)$, there is a natural
isomorphism 
$$R\vvpi _*L\vsi^*((i_\La)_*\oo_\La)
\ \simeq\ R(\mmu\times\nnu)_*\aa.$$
\end{lem}

\begin{proof}
To simplify notation, we put $Z:=\bb\times \gg\times\tt.$
Further, write $pr_{T^*\bb}$,
resp. $pr_{\gg\times\tt}$,
for the projection of the variety
$T^*\bb\times\gg\times\tt$
to the first, resp.
along the first, factor. Clearly, we have
$\vsi_*\oo_Z=pr_{T^*\bb}^*(\iim_*\oo_\bb)$.
Also, using the notation of diagram
\eqref{psi}, we get
$\pr_{\La\to T^*\bb}=pr_{T^*\bb}\ccirc
i_\La$, resp. $\pr_{\La\to \gg\times\tt}
=i_\La\ccirc pr_{\gg\times\tt}.$
We deduce a chain of isomorphisms
$$
Li^*_\La(\vsi_*\oo_Z)=
Li^*_\La[pr_{T^*\bb}^*(\iim_*\oo_\bb)]
=L(i_\La\ccirc pr_{T^*\bb})^*(\iim_*\oo_\bb)
=L\pr_{\La\to T^*\bb}^*(\iim_*\oo_\bb).
$$

Next, we use the isomorphism
$\Phi: \wt\N\iso T^*\bb$, see \S\ref{sympl},
to identify
the imbedding $\bb\into \wt\N$
with the zero section $\iim: \bb\into T^*\bb$.
Thus, from the above chain
of isomorphisms, using  commutative  diagram \eqref{psi}, we get
\beq{xy}
R(\pr_{\La\to \gg\times\tt})_* Li_\La^*(\vsi_*\oo_Z)=
R(\pr_{\La\to \gg\times\tt})_*
L\pr_{\La\to T^*\bb}^*(\iim_*\oo_\bb)\cong
R(\mmu\times\nnu)_*L\kkap^*(\iim_*\oo_\bb).
\eeq

Now, we apply Lemma \ref{standard} in the
case where $X=T^*\bb\times\gg\times\tt$
and $Y=\La$. So, we have $i_Z=\vsi$
and the composite isomorphism in \eqref{xyz}
yields  $(i_\La)_*Li^*_\La[\vsi_*\oo_Z]=
\vsi_*L\vsi^*[(i_\La)_*\oo_\La]$.
Note  that $\vvpi=\vsi\ccirc pr_{\gg\times\tt},$
so $R\vvpi_*=(Rpr_{\gg\times\tt})_*\vsi_*$.
Hence, 
 we
obtain
\begin{multline*}
R\vvpi_*L\vsi^*[(i_\La)_*\oo_\La]\ =\ 
R(pr_{\gg\times\tt})_*\vsi_*\vsi^*[(i_\La)_*\oo_\La]\
\stackrel{\eqref{xyz}}{=\!=}\
R(pr_{\gg\times\tt})_*(i_\La)_*Li^*_\La[\vsi_*\oo_Z]\\
\stackrel{\eqref{xy}}{=\!=}\
R(\pr_{\La\to \gg\times\tt})_* Li_\La^*(\vsi_*\oo_Z)\ =\ 
R(\mmu\times\nnu)_*L\kkap^*(\iim_*\oo_\bb).
\end{multline*}

Here, the object on the right is $R(\mmu\times\nnu)_*\aa$,
and the lemma is proved.
\end{proof}

\begin{cor}\label{dima22} The 
sheaves $\H^k\big(R(\mmu\times\nnu)_*\aa\big)$
vanish for all $k\neq 0$ and
there is an $\oo_{\gg\times\tt}$-module isomorphism
$\H^0\big(R(\mmu\times\nnu)_*\aa\big)\cong\ggr^\hodge\mm$.
\end{cor}
\begin{proof} This is an immediate consequence
of  Theorem \ref{cohiso} and  Lemma \ref{amu}.
\end{proof}

\subsection{Completing the proof of Theorem \ref{mthm}}\label{xyzsec}
We recall that in order to complete
 the proof of Theorem \ref{mthm}
it remains to prove  Proposition \ref{res_lem}
and Lemma \ref{irr}. The proof of the lemma
will be given later, in \S\ref{ss3}.

\begin{proof}[Proof of Proposition \ref{res_lem}]
Let $\aa_0$ be the degree zero homogeneous component
of $\aa$.
We may view $\aa_0$
as a  DG algebra equipped with  zero differential
and concentrated in degree zero.
Thus, one has a natural DG algebra
imbedding $f:\ (\aa_0,0)\into (\aa,\partial_{\kkap^*\eu}) $.

To simplify notation, put $\ppi:=\mmu\times\nnu$.
Applying the functor
$R\ppi_*$
and using that $\aa_0=\oo_\tgg=
\ppi^*\oo_{\gg\times\tt}$, 
we obtain the following chain of DG algebra morphisms
\beq{bff}\xymatrix{
\oo_{\gg\times\tt}\ \ar[rr]^<>(0.5){\text{adjunction}}&&
\ R\ppi_*\ppi^*\oo_{\gg\times\tt}\ =\ 
R\ppi_*\aa_0\ \ar[rr]^<>(0.5){R\ppi_*(f)}&&
\ R\ppi_*\aa.
}
\eeq

Let $\gg^{rr}\sset \gg$ be a
 Zariski
open subset  such that $(\gg^{rr}\times\tt)\cap \xx=
\xx^{rr}$.
 From Lemma \ref{codim1}
and Corollary \ref{transv}(ii), we deduce
that the composite morphism
in \eqref{bff} induces an isomorphism
$\oo_{\xx^{rr}}\iso R\ppi_*\aa|_{\gg^{rr}\times\tt}$.
Combining the latter isomorphism with
the isomorphism
 $\H^0\big(R\ppi_*\aa\big)\cong\ggr^\hodge\mm$
of Corollary \ref{dima22} yields the required
isomorphism $\oo_{\xx^{rr}}\iso\jmath^*(\ggr^\hodge\mm)$.
\end{proof}

The statement of the next result was
 suggested to me by Dmitry Arinkin. 

\begin{thm}\label{dima}  There is 
 $G\times \C^\times\times\C^\times$-equivariant
DG $\oo_{\gg\times\tt}$-algebra quasi-isomorphism:\hfill\break
$\dis R(\mmu\times\nnu)_*\aa\simeq\psi_*\oo_{\xx_\norm}$.
\end{thm}

\begin{proof} The fact that $R(\mmu\times\nnu)_*\aa$ and
$\psi_*\oo_{\xx_\norm}$ are isomorphic as 
objects of $\dcoh(\gg\times\tt)$ is
an immediate consequence of 
 Theorem \ref{mthm}
and Corollary \ref{dima22}.

To establish the DG $\oo_{\gg\times\tt}$-{\em algebra} quasi-isomorphism,
one can argue as follows.
Let ${\mathcal F}:=\H^0\big(R(\mmu\times\nnu)_*\aa\big)$.
We claim that the sheaf  ${\mathcal F}$ is
 Cohen-Macaulay.
This follows  from Corollary \ref{dima22} since we
know that $\ggr^\hodge\mm$ is a  Cohen-Macaulay
coherent sheaf set theoretically supported on $\xx$,
see Corollary  \ref{serre}.

There is also an alternative 
proof of the claim that does not use the
Cohen-Macaulay property
of associated graded sheaves 
arising from Hodge modules.
That alternative 
proof is based instead on the self duality property of the DG algebra $\aa$,
see Remark \ref{adual}. The latter property, combined with the fact
that the morphism $\mmu\times\nnu$
 is  proper, implies
 that the object
$R(\mmu\times\nnu)_*\aa\in\dcoh(\gg\times\tt)$ is
isomorphic to its Grothendieck-Serre dual, up to a shift.
Therefore, the cohomology vanishing
from Corollary \ref{dima22} forces
 the sheaf 
${\mathcal F}=\H^0\big(R(\mmu\times\nnu)_*\aa\big)$ be Cohen-Macaulay.
The claim follows.

Now, according to the
proof of  Proposition \ref{res_lem}
given above,  the composite morphism
in \eqref{bff} induces a
$G\times\CC$-equivariant  $\oo_{\gg\times\tt}$-{\em algebra} isomorphism
$\oo_{\xx^{rr}}\iso
\jmath^*{\mathcal F}$.
Thus, Lemma \ref{ei} provides a $G\times\CC$-equivariant 
{\em algebra} isomorphism $\psi_*\oo_{\xx_\norm}\iso
{\mathcal F}$. 
\end{proof}

We observe next that $\HH^\hdot(\tgg,\aa)$,
the hyper-cohomology of the DG algebra $(\aa,\,\partial_{\kkap^*\eu})$,
acquires the canonical structure of a
graded commutative algebra.
From Theorem \ref{dima}, for any $k\in\Z$, we deduce
$$\HH^k(\tgg,\,\aa)=
\HH^k(\gg\times\tt,\ R(\mmu\times\nnu)_*\aa)=
H^k(\gg\times\tt,\ \psi_*\oo_{\xx_\norm}).
$$
The group on the right vanishes for any $k\neq 0$
since the scheme  $\gg\times\tt$ is
affine. So, we obtain

\begin{cor}\label{dima3}
The hyper-cohomology groups 
$\HH^k(\tgg, \aa)$ vanish  for all $k\neq 0$
and there is a  $G$-equivariant bigraded
$\C[\gg\times\tt]$-algebra isomorphism
$\HH^0(\tgg, \aa)\cong\C[\xx_\norm]$.\qed
\end{cor}

\begin{rem}\label{dgres}  Write ${\mathcal X}:=\Spec\aa$ for the DG scheme
associated with the  DG algebra $\aa$
 in the sense
of derived algebraic geometry, cf. \cite{TV}.
The  DG scheme ${\mathcal X}$ 
may be thought of as a `derived
analogue' of the scheme $\wt\xx$,
cf. \S\ref{tx_sec}. Then,
Corollary \ref{dima3} says  that the morphism $\mmu\times\nnu$ induces
a  DG-algebra quasi-isomorphism
$\C[\xx_\norm]\to R\Ga({\mathcal X},\ \oo_{\mathcal X})$.
This
may be interpreted
as 
saying that the
 DG scheme $\Spec\aa$ provides, in a sense, a `DG resolution'
of the variety $\xx_\norm$.
\end{rem}

\subsection{Proof of Proposition \ref{ordhodge}}\label{pf}
Let $f:\ \bb\times(\g\times\t)\to \g\times\t$
be the second projection.

\begin{lem}\label{intd} All nonzero cohomology
sheaves of the complex $\int_f^R\dd_{\bb\times\g\times\t}$
vanish and one has
an isomorphism $\H^0(\int_f^R\dd_{\bb\times\g\times\t})\cong
\dd_{\g\times\t}$ of right $\dd_{\g\times\t}$-modules.
\end{lem}
\begin{proof}
There is a  standard Koszul type
complex ${K}^\hdot$
with terms ${K}^j=\dd_\bb\o_{\oo_\bb}\wedge^{-j} \T_\bb,
\ j=0,-1,\ldots, -\dim\bb,$
that gives a  resolution of the structure
sheaf $\oo_\bb$, cf. \cite[Lemma 1.5.27]{HTT}.
Using that
$H^0(\bb,\oo_\bb)=\C$ and $H^k(\bb,\oo_\bb)=0$
for any $k\neq 0$
we deduce that $ R\Ga(\bb, \ {K}^\hdot)$ $=
R\Ga(\bb, \oo_\bb)=\C$. Therefore,
an explicit construction of direct image 
for {\em right} $\dd$-modules 
(see \cite[Proposition 1.5.28]{HTT} for 
its {\em left} $\dd$-module counterpart) shows that 
$\int_f^R\dd_{\bb\times\g\times\t}=$ 
$R\Ga(\bb, \ {K}^\hdot)\,\o_\C\,
\dd_{\g\times\t}=\dd_{\g\times\t}.$
\end{proof}

Next, we recall the setting of \S\ref{filtmmG}.
We observe that $E=\int^R_\eps\kk_\tg$
is a cyclic  right 
$\dd_{\bb\times\g\times\t}$-module 
generated by the section $\eps_*(\om\o1)$
where $1\in\dd_{\tg\to\bb\times\g\times\t}=
\dd_{\bb\times\g\times\t}|_\tg.$
Therefore, the assignment $1\mto \eps_*(\om\o1)$
can be extended to a surjective morphism
$\gamma:\ \dd_{\bb\times\g\times\t}\onto E$
of right  $\dd_{\bb\times\g\times\t}$-modules. 
The  quotient
filtration on $E$ induced by 
 the projection $\gamma$ is equal,
by our normalization of the Hodge filtration,
to the Hodge filtration  $F^\hodge E$,
see Remark \ref{shift}.
This filtration on $E$, resp. the order filtration on $\dd_{\bb\times\g\times\t}$,
makes $\int^R_f E$, resp. $\int^R_f\dd_{\bb\times\g\times\t}$,
a filtered complex. 
Applying the functor $\int^R_f$ to $\gamma$,
a morphism of filtered $\dd$-modules,
one obtains a morphism
$\int^R_f\gamma:\ \int^R_f\dd_{\bb\times\g\times\t}\to
\int^R_f E$  of filtered complexes.

We identify the sheaf $\dd_{\g\times\t}$
with $\dd_{\g\times\t}^{op}$ via the
trivialization of the canonical
bundle on $\g\times\t$ provided by the section  $dx\, dt$.
Thus, using Lemma \ref{intd},  we obtain
a chain of morphisms of {\em left} $\dd_{\g\times\t}$-modules
\beq{j}
\xymatrix{
\dd_{\g\times\t}=
\kk\inv_{\g\times\t}\o\H^0(\int_f^R\dd_{\bb\times\g\times\t})\ 
\ar[r]^<>(0.5){\int^R_f\gamma}&
\ \kk\inv_{\g\times\t}\o\H^0(\int^R_f E)
=\H^0(\int_{\mu\times\nu}\oo_\tg)=\mm.
}
\eeq

It is straightforward to see,
using the explicit formula $
u\mto u[(dx\,
dt)\inv\bo(\mu\times\nu)_*\om]$
for the isomorphism
$\mm\iso \H^0(\int_{\mu\times\nu}\oo_\tg),$
cf. \S\ref{hk},
 that the composite
morphism  in \eqref{j} is equal to  the
natural
projection $\dd_{\g\times\t}\onto
\dd_{\g\times\t}/\I=\mm$.

The proof of Lemma \ref{intd} shows
that
the  filtration on the $\dd$-module
$\H^0(\int^R_f\dd_{\bb\times\g\times\t})$
induced by the filtered structure on $\int^R_f\dd_{\bb\times\g\times\t}$
goes, under
the isomorphism of the Lemma,
to the standard order filtration on
the sheaf $\dd_{\g\times\t}$.
It follows that all the maps in 
\eqref{j} respect the filtrations.
Thus, writing $\bar\gamma: \dd_{\g\times\t}
\onto \mm$ for the
composite map in \eqref{j},
we get $\bar\gamma(F_k^\ord\dd_{\g\times\t})
\sset F^\hodge_k\mm$ for any $k\in\Z$.
Proposition \ref{res_lem} follows from this
since the order filtration $F^\ord\mm$
 was defined as the
quotient filtration on $\dd_{\g\times\t}/\I$.

\section{A generalization of a construction of Beilinson
and Kazhdan}\label{dima_sasha}

\subsection{The universal stabilizer sheaf}\label{unistab}
Given a $G$-action on an irreducible scheme $X$, one defines
the "{\em universal stabilizer}" scheme
${\mathcal G}_X$ as a scheme  theoretic preimage of the
diagonal in $X\times X$ under the
morphism $G\times X\to X\times X,\ (g,x)\mto (gx,x).$
Set theoretically, one has
${\mathcal G}_X=\{(g,x)\in G\times X\mid
gx=x\}$.
The group $G$ acts naturally on ${\mathcal G}_X$ by
$g_1: (g,x)\mto (g_1gg_1\inv, g_1x)$.
The second projection ${\mathcal G}_X\to X,
\ (g,x)\to x$ gives 
 ${\mathcal G}_X$  the natural
structure of a  $G$-equivariant group scheme over $X$.
The Lie algebra of the group scheme $\ZZ_X$ is 
a coherent
sheaf $\fz_X:=\Ker(\g\o\oo_X\to \T_X)$,
the kernel of the natural "infinitesimal action"
morphism.  Let
$\fz_X^*:=\hhom_{\oo_X}(\fz_X,\oo_X)$ be the 
(nonderived) dual of $\fz_X$.

Let $L$ be a  finite dimensional $\g$-representation.
%The action of $\g$ on $L$ gives a linear map
%$L\to\Hom_\C(\g, L).$
There are
natural morphisms of sheaves
\beq{vg}L\o\oo_X\ \too\  \Hom_\C(\g, L)\o\oo_X=  L\o\g^*\o \oo_X\ \too\ 
 L\o\fz_X^*.
\eeq
Here, the first morphism is induced by the linear map 
$L\to\Hom_\C(\g, L)$ resulting from the $\g$-action on $L$,
and the second morphism is induced by the sheaf imbedding
$\fz_X\to\g\o\oo_X$.

Let $L^{\fz_X}$ denote the kernel
of the composite morphism in \eqref{vg}.
Thus, $L^{\fz_X}$ is a $G$-equivariant coherent subsheaf of $L\o\oo_X$.
The  geometric fiber of
the sheaf $L^{\fz_X}$ at a sufficiently general point
$x\in X$ equals $L^{\g_x}.$

\begin{lem}\label{tors} For an irreducible  normal
variety $X$
with a $G$-action, we have

\vi The sheaf imbedding $L^{\fz_X}\to L\o\oo_X$ induces
an isomorphism $\Ga(X,\, L^{\fz_X})^G\iso (L\o\C[X])^G$, provided general points
of $X$ have connected stabilizers.

\vii Let $j: U\into X$ be an  imbedding
of a
Zariski open subset such that $\dim(X\sminus U)\leq\dim X -2.$
Then, the canonical morphism
$L^{\fz_X}\to j_*j^*(L^{\fz_X})$,
resp. $\fz_X\to j_*j^*(\fz_X)$, is an isomorphism.
\end{lem}
\begin{proof} For any $x\in X$ and any  $G$-equivariant
morphism $f: X\to L$, the element $f(x)\in L$ is fixed
by the group $G_x$. This implies (i).

To prove (ii), we use the isomorphism
 $\g\o\oo_X\iso j_*j^*(\g\o\oo_X)$,
due to  normality of $X$. Thus, we deduce
$$\fz_X=\Ker[j_*(\g\o\oo_U)\to \T_X]
=j_*\big(\Ker[\g\o\oo_U\to j^*\T_X]\big)
=j_*\fz_U.
$$

Next, let
$V\sset X$ be an open subset and let $s\in \Ga(V,\
j_*j^*(L^{\fz_X}))$. One may view $s$ as a morphism
$V\cap U\to L$. This  morphism can be extended to
a regular map $\bar s: V\to L$, 
since $X$ is normal.
Let $s'$ be the image of $\bar s$ under the
composite morphism in \eqref{vg}. Thus,
$s'\in\Ga(V,\ L\o\fz_X^*)$ and, by the definition of
the sheaf $L^{\fz_X}$, we have $s'|_U=0$.

Now, the dual of a coherent sheaf is
a  torsion free sheaf. Therefore,
 the sheaf $L\o\fz_X^*$ is torsion free.
Hence, $s'|_{V\cap U}=0$ implies  that $s'=0$. We deduce that
$\bar s$ is actually a section (over $V$) of the sheaf $L^{\fz_X}$, the kernel
of the map \eqref{vg}.
Since  $\bar s|_U=s$, we have proved that the morphism
$L^{\fz_X}\to j_*j^*(L^{\fz_X})$ is surjective. This morphism
is injective since  $L^{\fz_X}$, being a subsheaf of
a locally free sheaf $L\o\oo_X$, is a torsion free sheaf.
\end{proof}

\subsection{A canonical filtration}\label{canfilt}
Let $\X$ be the weight lattice of $\g$.
Given a Borel subalgebra $\b$, we  
identify elements of $\X$ with linear functions on  $\b/[\b,\b]$.
We let $R$ (= the set of roots), resp. $R^+$ (= the set of positive
roots)
be the set of nonzero weights  of the $\ad\b$-action
on $\g$, resp. on $\g/\b$. We introduce 
 a partial order
on $\X$ by setting $\la'\ccl \la$ if $\la-\la'$ is a sum
of positive roots.

Fix a finite dimensional $\g$-module $L$
and  a Borel subalgebra $\b$.
Let $L=\oplus_{\la\in\X}\ L_\t^\la$
be the weight decomposition of $L$ with respect to
a Cartan subalgebra $\t\sset\b$. 
For any $\la\in \X$, we put
 $L^{\ccl\la}_\t:=\bplus_{\la'\ccl \la}\ L_\t^{\la'}$,
resp. $L^{\lel\la}_\t:=\bplus_{\la'\ccl \la,\ \la'\neq \la}\ L_\t^{\la'}$.
This is a  $\b$-stable subspace of $L$ that does not depend on the
choice of a  Cartan subalgebra $\t\sset \b$.
Therefore, the above construction associates with
each Borel subalgebra $\b$ a canonically defined
$\b$-stable filtration, to be denoted by $L^{\ccl\hdot}_\b$, 
labeled by the partially ordered set $\X$.

One may let the Borel subalgebra $\b$ vary inside the
flag variety. The family $\{L_\b^{\ccl\hdot},\,\b\in\bb\}$,
of filtrations on $L$, then
gives  a filtration
$L_\bb^{\ccl\hdot}$ of the trivial sheaf
$L\o\oo_\bb$ by $G$-equivariant locally free
$\oo_\bb$-subsheaves. We put  $L^\la_\bb=L_\bb^{\ccl
\la}/L_\bb^{\lel\la}$
and let $L_\bb :=\bplus_{\la\in\X}\ L^\la_\bb$, an
 associated graded sheaf. This is a $G$-equivariant locally free
sheaf on ~$\bb$.

In the special case $L=\g$,
 the adjoint representation,
one has $L_\b^{\ccl 0}=\b$, resp.
$L_\b^{\lel 0}=[\b,\b]$, for any $\b\in\bb$.
Furthermore,
the above construction yields,  for each root $\al\in R$,
 a $G$-equivariant line bundle $\g^\al_\bb:=
\g_\bb^{\ccl \al}/\g_\bb^{\lel\al}.$
Thus, we have
\beq{lie}
\g_\bb\ :=\ \g^0_\bb\ \,\bplus\,
\  \big(\underset{^{\al\in R}}\oplus\ \g^\al_\bb\big).
\eeq
The Lie bracket on $\g$ induces a natural
fiberwise bracket on the locally free sheaf $\g_\bb$.
Further, it is easy to check that the invariant bilinear form on $\g$
induces
an isomorphism $\g^{-\al}_\bb\cong(\g^\al_\bb)^*$,
of line bundles on ~$\bb$.

Given a finite dimensional representation $L$,
the action map $\g\o L\to L$ induces, for any 
 $\al\in R$ and $\la\in\X$,
a well defined morphism
$\g^\al_\bb\o L^\la_\bb 
\to L^{\la+\al}_\bb$,
of $G$-equivariant  locally free sheaves on $\bb$.
This way, one obtains a natural 'action' morphism
$\g_\bb \o\ L_\bb  \to L_\bb$.

Note that any Borel subgroup $B\sset G$ acts
trivially on the vector space $L_\b^{\ccl 0}/L_\b^{\lel0}$.
It follows that there is a canonical
isomorphism $L_\b^{\ccl 0}/L_\b^{\lel0}=
L_{\b'}^{\ccl 0}/L_{\b'}^{\lel0}$,
for any  Borel subalgebras $\b$ and $\b'$.
In other words,
 the $\oo_\bb$-sheaf $L_\bb^0$
comes equipped with a canonical $G$-equivariant
trivialization. In the case of the adjoint representation
the canonical trivialization reads $\g^0_\bb=\uh\o\oo_\bb$,
where the vector space $\uh$ is referred to as the {\em universal Cartan algebra} of
$\g$. More generally, for an arbitrary representation $L$,
one has the  canonical  $G$-equivariant 
isomorphism $L^0_\bb \cong L^\uh\o\oo_\bb$,
where the vector space $L^\uh$ may be called
 the `{\em universal zero weight space}' of $L$.
The group  $G$ acts on $L^\uh\o\oo_\bb$ through its 
action on $\oo_\bb$. 

\begin{rem}\label{X} In the above discussion,
we have implicitly viewed elements  of the weight lattice $\X$ 
as linear functions on $\uh$, the  universal Cartan algebra.
\end{rem}

\subsection{$W$-action}\label{morphism_sec}
Let
 $\tg^r=\mu\inv(\g^r)$, a Zariski open subset
of $\tg$. 
An important role below will be played
by the following commutative diagram

\beq{cart}
\xymatrix{
\tg^r\ \ar[drr]_<>(0.5){\mu}\ar[rr]^<>(0.6){\pi}_<>(0.6){\sim}&&
\ \yr=\g^r\times_{\t/W}\t\ \ar[rr]^<>(0.5){\wt\gamma}\ar[d]^<>(0.5){pr}&&
\ \t\ \ar[d]^<>(0.5){\vartheta}\\
&&\ \g^r\ \ar[rr]^<>(0.5){\gamma}&& \g^r/G=\t/W,
}
\eeq
In this diagram, the map $pr$  is the first projection,
 $\vartheta$ is
the quotient map,
 $\wt\gamma$ is the second projection
and $\gamma$ is the adjoint quotient
morphism. 
The map $\pi$, in the diagram,  is the isomorphism from Proposition \ref{tb}(i).
Note that the square on the right of diagram \eqref{cart}
is cartesian, by definition.

Given a $\g$-representation $L$,
 we pull-back  the canonical
filtration on the  trivial sheaf ${L\o\oo_\bb}$ via the vector bundle  projection
$\tg\to\bb$.
We obtain a filtration $L_\tg^{\ccl\hdot}$, of the sheaf $L\o\oo_\tg$, by
 $G$-equivariant subsheaves. The sheaves 
$L^\la_\tg :=
L_\tg^{\ccl\la}/L_\tg^{\lel\la}$ are locally free,
and so is $L_\tg=\oplus_{\la\in\X}\ L^\la_\tg$,
 an associated graded sheaf. In particular, for
 $\la=0$ and $L=\g$,
the adjoint representation, one gets the sheaf
$\g^0_\tg=\uh\o\oo_\tg.$

The Weyl group $W$ acts on the fibers
of the map $pr$ in \eqref{cart}.
We may transport this action  via the isomorphism $\tg^r\iso\yr$
to get a $W$-action on $\tg^r$.
Thus, for any $w\in W$ and
$(\b,x)\in \tg^r$ there is a  unique Borel subalgebra $\b'$
such that we have $w(\b,x)=(\b',x)$.

Assume now that the  element $x$ is regular semisimple.
Then, $\g_x$ is a Cartan subalgebra of $\g$
contained in
$\b\cap \b'.$ Thus, one has the following chain of
isomorphisms
\beq{gx}\uh\ =\ \g^0_\tg\big|_{(\b,x)}\ =\ 
\b/[\b,\b]\ \stackrel{_\sim}{\leftarrow}\ 
\g_x\ \iso\  \b'/[\b',\b']\ =\ \g^0_\tg\big|_{w(\b,x)}\ =\ \uh.
\eeq

Here, the composite map from the copy of $\uh$ on the left to the
  copy of $\uh$ on the right  is the map
  $\uh\to\uh,\
h\mto w(h)$.  One also has the dual map $\uh^*\to\uh^*,\ \la\mto
w(\la)$.
It follows that, given  $\la\in \X$, there is an
analogoue of diagram \eqref{gx} 
for any $\g$-representation $L$; it reads:
$$
L^\la_\bb\big|_{(\b,x)}\ =\ 
L^{\ccl\la}_\b/L^{\lel\la}_\b
\ \stackrel{_\sim}{\leftarrow}\ 
L^\la_{\g_x}\ \iso\  
L^{\ccl w(\la)}_{\b'}/L^{\lel w(\la)}_{\b'}\ =\ 
L^{w(\la)}_\tg\big|_{w(\b,x)}.
$$

Now, we let the Borel subalgebra $\b$ vary.
We conclude that the above construction yields a canonical
isomorphism
\beq{w}
w^*(L^\la_\tg)\big|_{\tg^{rs}}\ \cong\
L^{w(\la)}_\tg\big|_{\tg^{rs}},\qquad
\forall\la\in \X,\ w\in W,
\eeq
  of $G$-equivariant locally free sheaves 
on  $\tg^{rs}:=\mu\inv(\g^{rs})$:

In the special case where $\la=0$, we have that
$L^0_\tg=L^\uh\o\oo_\tg$ is a trivial sheaf.
Hence, $w^*(L^0_{\tg^r})=L^0_{\tg^r}$, canonically.
Thus, the isomorphism in \eqref{w} yields a 
{\em canonical}
$W$-action on  the universal zero
weight space  $L^\uh$.

\subsection{The morphisms $\la^k$ and $\la^L$}\label{bei}
In this section, we are interested in the
universal stabilizer contruction
of \S\ref{unistab} in the special case where the group
$G$ acts on
 $X=\g^r$ by the adjoint action. Below, we will use simplified notation
$\fz:=\fz_{\g^r}$. Thus, $\fz$ is a rank $\rk$ locally free sheaf
on $\g^r$. 
The geometric fiber of the sheaf $\fz$
  at any point $x\in \g^r$  equals $\g_x$, the centralizer of $x$.

From now on, we let $L$ be a finite dimensional
rational
$G$-representation
such that the set of weights of $L$ is
{\em contained in the root lattice}.
In this case, the results of Kostant \cite{Ko}
insure that  $L^\fz$ is a  locally free sheaf on $\g^r$.

Let  $\mu^*(L^\fz)$ be the pullback of  $L^\fz$ via the 
map $\mu: \tg^r\to\g^r$. Thus,  both $\mu^*(L^\fz)$ and $L_{\tg^r}^{\ccl 0}$
are $G$-equivariant   locally free subsheaves of the trivial sheaf
$L\o \oo_{\tg^r}$.

\begin{lem}\label{centb}
The sheaf $\mu^*(L^\fz)$ 
is a subsheaf of $L_{\tg^r}^{\ccl 0}$ so,
for any Borel subalgebra $\b$ and any $x\in \g^r\cap \b$,
we have an inclusion $L^{\g_x}\sset L_\b^{\ccl 0}.$
\end{lem}

\begin{proof} We may (and will) identify the locally free sheaf
$\mu^*(L^\fz)$, resp. $L_{\tg^r}^{\ccl 0}$,
with the  corresponding vector bundle on $\tg^r$.
 Thus, both $\mu^*(L^\fz)$ and $L_{\tg^r}^{\ccl 0}$
are $G$-equivariant sub vector bundles of
the trivial vector bundle on $\tg^r$ with fiber $L$.

Let $(\b,x)\in\tg^{rs}$.
Then, $\g_x$ is a Cartan subalgebra contained in $\b$ and,  by definition, we have
$L^{\g_x}=L^0_{\g_x}\sset L_\b^{\ccl 0}.$
Hence, for any  $(\b,x)\in\tg^{rs}$, the fiber of the vector bundle
 $\mu^*(L^\fz)$ at the point $(\b,x)$ is contained in
 the corresponding fiber of the vector bundle
 $L_{\tg^r}^{\ccl 0}$.
The statement of the lemma follows from this by
continuity since the set $\tg^{rs}$ is dense in ~$\tg^r$.
\end{proof}

Thanks to the above lemma, one has the following chain of canonical
morphisms of locally free $G$-equivariant sheaves
on $\tg^r$:
\beq{lfree}
\mu^*(L^\fz) \too 
L^{\ccl 0}_{\tg^r}\ \onto\  L_{\tg^r}^{\ccl
0}/L_{\tg^r}^{\lel0}\ =\  L^0_{\tg^r}\ =\ 
L^\uh\o \oo_{\tg^r}.
\eeq

We are now going to  transport
  various sheaves on $\tg^r$ to $\yr$ via
 the isomorphism $\pi: \tg^r\iso\yr$ of  diagram \eqref{cart}.
This way,
one obtains a filtration $L_\yr^{\ccl \hdot}$ of the trivial 
sheaf $L\o\oo_\yr$ and an associated graded
sheaf $L_\yr=\oplus_{\la\in\X}\ L^\la_\yr$.
Further, we  factor the morphism $\mu$ in \eqref{cart}
as a composition $\mu=pr\ccirc\pi$.
Thus, we have $\mu^*(L^\fz)=\pi^*pr^*(L^\fz),$
where we put $L_\yr^\fz:=pr^*(L^\fz)$.
Transporting the composite
 morphism  in  \eqref{lfree}
via the isomorphism $\pi$
yields a morphism 
\beq{bbmap12}
pr^*(L^\fz)\ \too\ L^0_\yr=L^\uh\o\oo_\yr.
\eeq
Applying
adjunction to the above morphism, one obtains a morphism
%\beq{bbmap13}
$\la^L: L^\fz\to L^\uh\o pr_*\oo_\yr$,
%\eeq
 of locally free $G$-equivariant sheaves on $\g^r$.

 An important special case of the above setting is the case where
 $L=\g$,  the adjoint representation.
The weights of  the adjoint representation are clearly
contained in the root lattice. Furthermore,
 for  $L=\g$ one has 
$L^{\g_x}=\g_x$, hence, we have $L^\fz=\fz$.
Since $\g_\b^{\ccl 0}=\b$, we see that  Lemma \ref{centb} reduces in
 the case
$L=\g$
to  a well known result
saying that, for any $x\in \b\cap\g^r$, one has an inclusion   $\g_x\sset\b$.
Further, the composite morphism in  \eqref{lfree}
becomes 
a  morphism $pr^*\fz\to \uh\o\oo_\yr.$
We take exterior powers  of that morphism.
This gives, for each $k\geq 0,$ 
an  induced   morphism $pr^*(\wedge^k\fz)\to
\wedge^k\uh\o \oo_\yr$. Finally, applying
adjunction  one obtains a morphism
$\la^k:\ \wedge^k\fz\ \to\ \wedge^k\uh\o pr_*\oo_\yr$.

%Now, let $L$ be an arbitrary $\g$-representation. 
%Applying
%adjunction to the morphism  \eqref{bbmap12},
%resp.  to the morphism $pr^*(\wedge^k\fz)\to
%\wedge^k\uh\o \oo_\yr$, one obtains a morphism
%\beq{bbmap13}
%\la^L:\ L^\fz\ \to\ \gr^{0\!\!}L\o pr_*\oo_\yr,
%\quad\oper{resp.}\quad
%\la^k:\ \wedge^k\fz\ \to\ \wedge^k\uh\o pr_*\oo_\yr,\en k\geq 0.
%\eeq

Associated with any $W$-module $E$,
there is a   $G$-equivariant coherent sheaf
$(E\o pr_*\oo_\yr)^W$ on $\g^r$.
In particular, one may let $E=L^\uh$,
the universal zero weight space of a $G$-module $L$

\begin{lem}\label{bbcompare} 
\vi The sheaf $(E\o pr_*\oo_\yr)^W$ is locally free
for any $W$-module $E$.

\vii The   image of the morphism $\la^L$, resp.  $\la^k,\ k\geq 0,$
is contained in the corresponding sheaf of $W$-invariants, so 
the morphism in question factors through an injective morphism
\beq{bbmap14}
\la^L:\ L^\fz\ \to\ ( L^\uh\o pr_*\oo_\yr)^W,
\quad\oper{resp.}\quad
\la^k:\ \wedge^k\fz\ \to\ (\wedge^k\uh\o pr_*\oo_\yr)^W,\en k\geq 0.
\eeq

\viii On the  open dense set $\g^{rs}\sset\g^r$,
each of the morphisms in \eqref{bbmap14} is an isomorphism.
\end{lem}
\begin{proof}[Sketch of Proof] The statement of part (i) is well known
but we recall the proof, for completeness.
First, we apply base change
for the cartesian square in diagram \eqref{cart}.
This yields a chain of isomorphisms
$pr_*\oo_\yr=pr_*\wt\gamma^*\oo_\t=\gamma^*\vartheta_*\oo_\t$.
Here, the sheaf on the right is locally free
and, moreover, the Weyl group $W$ acts on the
geometric fibers of that sheaf via the regular
representation. We deduce that the sheaf $pr_*\oo_\yr$
has similar properties. Part (i) follows from this.

The proof of part (iii) is  straightforward.
Finally, part (ii) follows from  part (iii) `by  continuity',
using that the
sheaves $L^\fz$ and $\wedge^k\fz$ are locally free.
\end{proof}

The theorem below, which is the main result of this
section,  is inspired by an idea due to Beilinson and Kazhdan.
In \cite{BK}, the authors  considered the isomorphism of part (i)
of Theorem \ref{BBprop}
in the special case $k=1$ (equivalently, of part (ii) in the case $L=\g$).

\begin{thm}\label{BBprop}\vi 
For any $k\geq1,$
the morphism 
$\dis\la^k:\ \wedge^k\fz\ \iso\ (\wedge^k\uh\o pr_*\oo_\yr)^W$
is an isomorphism.

\vii For any {\em small} representation $L$, 
the morphism 
$\dis\la^L:\  L^\fz\  \iso\  ( L^\uh\o pr_*\oo_\yr)^W$
is an isomorphism.
\end{thm}

The proof of Theorem \ref{BBprop} will be completed in
\S\ref{BBproof}.

In \S\ref{vin} we sketch  a generalization of  part (ii) of the 
above theorem to the case
of an arbitrary, not necessarily small, representation $L$.

\subsection{Proof of Theorem \ref{BBprop}(ii)}\label{reformulate}
Let $\coh^G(\g^r)$ be the category of $G$-equivariant coherent sheaves
on $\g^r$. We begin with an easy

\begin{lem}\label{equiv2} The  functor 
$\dis
\GG:\ \coh^G(\g^r)\ \iso \ \Lmod{\C[\t]^W},
\ \
{\mathcal F}\ \mto\ 
\Ga(\g^r,\ {\mathcal F})^G
$ is an equivalence.
\end{lem}
\begin{proof} Recall the adjoint quotient morphism $\gamma$, cf. \eqref{cart}.
We have a diagram of functors
$$
\xymatrix{
\coh^G(\g^r)\
\ar@<0.5ex>[rrr]^<>(0.5){{\mathcal F}\ \mto\ (\gamma_*{\mathcal F})^G}&&&
\ \coh(\t/W)\
\ar@<0.5ex>[lll]^<>(0.5){\gamma^*}\ar[rr]^<>(0.5){\Ga(\t/W,-)}
&&\ \Lmod{\C[\t]^W}.
}
$$

It is known, thanks to results of Kostant \cite{Ko}, that
$\gamma$ is a smooth morphism, moreover, each fiber of that
morphism is a single $G$-orbit. It follows by equivariant
descent that the functor $ {\mathcal F}\ \mto\ (\gamma_*{\mathcal
F})^G$ 
is an equivalence, with $\gamma^*$ being its quasi-inverse.
The functor $\Ga(\t/W,-)$, in the above diagram,
is  an equivalence since the variety $\t/W$ is affine.
Finally, the functor $\GG$ is isomorphic to the composite functor
${\mathcal F}\ \mto\  \Ga(\t/W,\ (\gamma_*{\mathcal F})^G)\ =\
\Ga(\g^r,\ {\mathcal F})^G$.
It follows that $\GG$
 is also an equivalence.
\end{proof}

It is useful to  transport the 
morphisms in \eqref{bbmap14}  via the equivalence  $\GG$ of Lemma \ref{equiv2}.
To this end,
for a $G$-module $L$, 
we compute
\beq{GG1}
\GG(L^\fz)\ =\ \Ga(\g^r,\ L^\fz)^G\ =\ \Ga(\g,\ L^{\fz_\g})\ 
=\ (L\o \C[\g])^G,
\eeq
where the second equality holds by part (ii) of Lemma \ref{tors}
and the third equality holds by  part (i) of the same lemma.

Recall next that $\yy$ is known to be a normal  variety 
and the set $\yy\sminus\yr$ has codimension $\geq2$
in $\yy$. This yields natural isomorphisms
$\C[\g]\o_{\C[\t]^W}\C[\t] \ =\ \C[\yy]\ \iso\ \C[\yr]$.
In particular, we deduce
$\C[\yr]^G\ =\ (\C[\g]\o_{\C[\t]^W}\C[\t])^G \ =\ 
\C[\t].$

Thus, for any  $W$-representation $E$,
we find
\begin{align}
\GG((E\o pr_*\oo_\yr)^W)\ &=\ 
\Ga(\g^r,\ (E\o pr_*\oo_\yr)^W)^G
=\ 
(E\o \Ga(\g^r,\ pr_*\oo_\yr))^{W\times  G}\label{GG2}\\
&=\ (E\o \Ga(\yr,\ \oo_\yr))^{W\times G}\ 
=\ (E\o \C[\yr]^G)^W\ =\ (E\o \C[\t])^W.\nonumber
\end{align}

Now, fix  $\t\sset\b\sset\g$
and write $i:\ \t\into\g$ for the resulting imbedding of the Cartan
subalgebra. This gives, for any $G$-representation
$L$, an identification $L^\t \iso  L^\uh$
and also a natural restriction map
$i^*:\ (L\o \C[\g])^G\to
(L^\t\o \C[\t])^W.$

\begin{lem}\label{GGlem1} The following diagram commutes
\beq{GGdiag}
\xymatrix{
\GG(L^\fz)\ar[d]^<>(0.5){\GG(\la^L)}
\ar@{=}[rrr]^<>(0.5){\eqref{GG1}}&&& (L\o \C[\g])^G
\ar[d]^<>(0.5){i^*}
\\
\GG(( L^\uh\o pr_*\oo_\yr)^W)\ar@{=}[rr]^<>(0.5){\eqref{GG2}}
&&
( L^\uh\o \C[\t])^W\ar@{=}[r]&(L^\t\o \C[\t])^W
}
\eeq
\end{lem}

\begin{proof} To prove commutativity of the diagram
 it suffices to check this after localization
to the
open dense set $\g^{rs}\sset \g^r$ of regular
semisimple elements. The latter is straightforward and is
left for the reader.
\end{proof}

By commutativity of diagram \eqref{GGdiag}, using
the equivalence of Lemma \ref{equiv2} we conclude
that the morphism $\la^L$, in \eqref{bbmap14}, is an isomorphism
if and only if  so is the 
map $i^*$ on the right  of diagram \eqref{GGdiag}.
At this point, the proof of Theorem \ref{BBprop}(ii)
is completed by the following
result of B. Broer \cite{Br}.

\begin{prop}\label{brso} 
The restriction map
$i^*: (L\o\C[\g])^G\to (L^\t\o\C[\t])^W$
 is an isomorphism, for any small representation
$L$.\qed
\end{prop}

\subsection{Proof of Theorem \ref{BBprop}(i)}\label{BBproof}
Write $\Omega^k_X$ for the sheaf of $k$-forms
on a smooth variety $X$. 
Thus,
using the identification $\g^*\cong\g$, resp.
$\t^*\cong \t$, 
for each $k\geq 0$, we have an isomorphism
$\wedge^k\g\o\oo_\g\cong\Omega^k_\g$,
resp.
$\wedge^k\t\o\oo_\t\cong\Omega^k_\t$.

Let $f_1,\ldots,f_\rk$
be a set of homogeneous generators
of $\C[\g]^G$, a free polynomial
algebra. Thus, one may think of 
the 1-forms $df_i,\ i=1,\ldots,\rk,$
as being sections of the sheaf $\g\o\oo_\g\cong\Omega^1_\g$.
A result of Kostant says that the values
$(df_1)|_x,\ldots,(df_\rk)|_x$,   of
these  sections 
at any point $x\in\g^r$,  give a basis of
the vector space $\g_x=\fz|_x$. 
Hence, the  $\rk$-tuple $(df_1,\ldots,df_\rk)$ 
 provides a basis of
sections of the sheaf $\fz$.

Next, let $\bar f_i$ be the image of $f_i$
under  the Chevalley isomorphism $\C[\g]^G\iso \C[\t/W].$
The map $h\mto (\bar f_1(h),\ldots,\bar f_\rk(h))$
yields an isomorphism $\t/W\iso\C^\rk$.
Hence, the  $\rk$-tuple $d\bar f_1,\ldots,$  $d\bar f_\rk$, of
1-forms on $\t/W$, provides a basis of
sections of the sheaf $\Omega^1_{\t/W}$.
Therefore, the  $\rk$-tuple  $\gamma^*(d\bar f_1),\ldots,\gamma^*(d\bar f_\rk)$
provides  a basis  of sections of the sheaf
$\gamma^*\Omega^1_{\t/W}$, on $\g^r$.

Thus, the assignment $\gamma^*(d\bar f_i)\mto df_i,\ i=1,\ldots,\rk,$
gives an isomorphism of sheaves $\gamma^*\Omega^1_{\t/W}\iso\fz.$
Taking exterior powers, one obtains, for each  $k\geq 0$,
an isomorphism $\Psi:\
\gamma^*\Omega^k_{\t/W}\iso\wedge^k\fz$, of 
locally free sheaves on $\g^r$.

On the other hand, we recall   that the square on the right of diagram
\eqref{cart} is cartesian and the morphism
$\vartheta$ in the diagram is finite and flat.
Hence, by flat base change, 
for any $k\geq 0$, we get $W$-equivariant isomorphisms
$ pr_*\oo_\yr=pr_*\wt\gamma^*\oo_\t
=
\gamma^*\vartheta_*\oo_\t.$
Tensoring here each term by the $W$-module
$\wedge^k\t$ and  taking $W$-invariants
in the resulting sheaves, yields a natural
 sheaf isomorphism $\dis\Phi:\
(\wedge^k\t\o pr_*\oo_\yr)^W\iso
\gamma^*[(\vartheta_*(\wedge^k\t\o  \oo_\t))^W]=
\gamma^*(\vartheta_*\Omega^k_\t)^W.
$

Combining all the above morphisms together,
we obtain the following diagram
\beq{GGlem2} 
\xymatrix{
\gamma^*\Omega^k_{\t/W}\ \ar[r]^<>(0.5){\Psi}_<>(0.5){\sim}&
\ \wedge^k\fz\  \ar[rr]^<>(0.5){\la^k}_<>(0.5){\eqref{bbmap14}}&&
\ (\wedge^k\t\o pr_*\oo_\yr)^W\  
\ar[r]^<>(0.5){\Phi}_<>(0.5){\sim}&
\ \gamma^*((\vartheta_*\Omega^k_\t)^W)
}
\eeq

\begin{lem}\label{pullback}
The composite morphism in \eqref{GGlem2} equals the pull back via the
map $\gamma$ of the canonical
morphism of sheaves $\Omega^k_{\t/W}\to(\vartheta_*\Omega^k_\t)^W$.
\end{lem}
\begin{proof} The statement amounts to showing that,
for each $i$, the composite morphism \eqref{GGlem2} sends the section
$df_i$ to the section $\gamma^*(d(\vartheta^*\bar f_i))$,
where $d(\vartheta^*\bar f_i)\in\vartheta_*\Omega^1_\t$ is viewed as a 1-form on $\t$.
 Furthermore, it suffices to check this on
the
open dense set  of regular
semisimple elements where it is clear.
\end{proof}

At this point, we  apply a result
 due to L. Solomon \cite{So} saying that the  canonical
morphism $\Omega^k_{\t/W}\to(\vartheta_*\Omega^k_\t)^W$
is in fact an isomorphism.
It follows, thanks to Lemma \ref{pullback},
that the morphism $\la^k$ in the middle of
diagram \eqref{GGlem2} must be an isomorphism.
This completes the proof of  Theorem \ref{BBprop}(i).\qed

\begin{rem} There is an alternative approach to the proof
that the morphisms in \eqref{bbmap14} are isomorphisms
as follows. First of all,  on the open set of regular semisimple
elements our claim amounts to
Lemma \ref{bbcompare}(ii). Next, one  verifies
the result in the case $\g=\sl_2$. It follows that
the result holds for any reductive Lie algebra of semisimple rank 1.
Using this, one  deduces that the morphisms 
of locally free sheaves in \eqref{bbmap14}
are isomorphisms outside a codimension 2 subset.
The result follows.

A similar strategy can also be used to obtain direct proofs of
 the above mentioned
results of Broer and Solomon, respectively.
\end{rem}

\subsection{The case of a not necessarily small representation}\label{vin}
The sheaf morphism $\la^L: L^\fz\to (L^\uh\o pr_*\oo_\yr)^W$,
in \eqref{bbmap14}, may fail to be surjective
in the case where $L$ is a  not necessarily small representation.
 Nonetheless, using a result by
Khoroshkin, Nazarov, and Vinberg   \cite{KNV}, we will
 give a description of the image of that morphism
for an arbitrary representation $L$. 

To this end, we first  apply the  construction of sections \ref{morphism_sec} and
\S\ref{bei}
in the special case of the adjoint representation $L=\g$.
Thus, starting from the
sheaf $\g_\bb$ on $\bb$, see \eqref{lie},
the construction produces 
a  sheaf
$
\g_\yr :=\uh\o \oo_\yr\ \,\bplus\,\  \big(\underset{^{\al\in R}}\oplus\ \g^\al_\yr\big).
$
This is a $G$-equivariant   locally free sheaf on $\yr$
that comes equipped with a fiberwise  Lie bracket.

For any representation
$L$, there is
a natural action morphism $\g_\yr\o L_\yr\to L_\yr.$
Taking adjoints and using the
isomorphism $\g^{-\al}_\yr\cong 
(\g^\al_\yr)^*$ yields, for any $\al\in R$ and $\la\in\X$,
an induced {\em co}action morphism
$e_\al:\ L^\la_\yr\to L^{\la+\al}_\yr\o\g^{-\al}_\yr.$
Iterating the latter morphism $k$ times
we obtain morphisms
$$e_\al^k:\ L^\la_\yr\to L^{\la+k\cdot\al}_\yr\o\g^{-k\cdot\al}_\yr,
\qquad k=1,2,\ldots.
$$

Let $\ker\al\sset\t$ be 
 the root hyperplane corresponding to a  root $\al\in R^+$.
The inverse 
image of this  hyperplane via the map $\wt\gamma: \yr\to\t,$
cf. \eqref{cart},
is a smooth $G$-stable irreducible divisor  $D_\al\sset\yr$.
Given  $k\geq1$, we use the standard notation
$L^{k\cdot\al}_\yr\o\g^{-k\cdot\al}_\yr(-k\cdot D_\al)$ for
a subsheaf of $L^{k\cdot\al}_\yr\o\g^{-k\cdot\al}_\yr$ 
formed by the sections  which have a $k$-th order zero at
the divisor ~$D_\al$.

Our  description of the
image of the morphism $\la^L$ is provided by
the following result inspired by  \cite{KNV}.

\begin{thm}\label{knv}
Let  $L$ be  a finite dimensional $\g$-representation
such that the weights of $L$ are contained in the root lattice.
Then the morphism $\la^L:\ L^\fz\to (L^\uh\o pr_*\oo_\yr)^W$,
in \eqref{bbmap14},
yields an isomorphism of $L^\fz$ with 
a subsheaf $\mathcal{L}$ of the sheaf $(L^\uh\o pr_*\oo_\yr)^W$
defined as follows
$$
\mathcal{L}\ :=\ \{s\in (L^\uh\o pr_*\oo_\yr)^W\en\big|\en e_\al^k(pr^*(s))
\in L^{k\cdot\al}_\yr\o\g^{-k\cdot\al}_\yr(-k\cdot D_\al),\ \en
\forall \al\in R^+,\  k\geq1\}.
$$
\end{thm}

\begin{rem} It is not difficult to check that, in the case where 
$L$ is a small representation, one has $\mathcal{L}
=(L^\uh\o pr_*\oo_\yr)^W$.
So, Theorem \ref{knv} reduces in this case to
Theorem \ref{BBprop}(ii).
\erem

The proof of Theorem \ref{knv} is parallel to the arguments
used in  \S\ref{reformulate}. There is an analogue of
commutative diagram \eqref{GGdiag}.
The space $(L^\uh\o\C[\t])^W$ in that diagram  is replaced
by its subspace defined in \cite[p.1169]{KNV},
resp. the space $\GG((L^\uh\o pr_*\oo_\yr)^W)$ is replaced
by $\GG(\mathcal{L})$.
The role of Broer's result from Proposition \ref{brso}
is played by \cite[Theorem 2]{KNV}.

Theorem \ref{knv} will not be used in the rest of the paper; so,
details of the proof will be given elsewhere.

\section{Geometry of the commuting scheme}\label{END}
\subsection{Another isospectral variety}\label{ss1}%xxx
We write $x=h_x+n_x$ for the Jordan decomposition of an 
element $x\in\g$.
We say that a pair $(x,y)\in\gg$ is semisimple 
if both $x$ and $y$ are semisimple elements of $\g$.
Let $\gg^\b$ be the set of pairs
$(x,y)\in\gg$ such that there exists a Borel subalgebra
that contains both $x$ and $y$.

\begin{lem}\label{lss} \vi 
Let $(x,y)\in\gg$. Then $(x,y)\in \gg^\b$ if and only if
 there exists a semisimple pair $(h_1, h_2)\in\tt$
such that, for any polynomial $f\in\C[\gg]^G$, we have
$f(x,y)=f(h_1,h_2)$.

\vii If $(x,y)\in \zz$ then
the  $G$-diagonal orbit of the pair $(h_x,h_y)$ is the unique closed
$G$-orbit contained in the closure
of  the $G$-diagonal orbit
of $(x,y)$.
\end{lem}
\begin{proof} Let
$\t\sset \b$ be Cartan and Borel subalgebras of $\g$
and let $T$ be the
maximal torus corresponding to $\t$. 
Clearly, there exists a suitable one parameter
subgroup $\gamma:\ \C^\times\to T$
such that, for any   $t\in\t$ and $n\in [\b,\b]$,
one has
$\underset{^{z\to0}}\lim\ \Ad \gamma(z)(t+n)=t$.

To prove (i), let $(x,y)\in \b\times\b$. We can write
$x=h_1+n_1,\ y=h_2+n_2$ where $h_i\in\t$ and $n_i\in[\b,\b]$.
We see from the above that the pair $(h_1,h_2)$
is contained in the closure of the
$G$-orbit of the pair $(x,y)$. Hence,
for any $f\in\C[\gg]^G$, we have
$f(x,y)=f(h_1,h_2)$.

Conversely, let $(h_1,h_2)\in\tt$ and let
$(x,y)\in \gg$ be such that
$f(x,y)=f(h_1,h_2)$ holds for  any $f\in\C[\gg]^G$.
The group $G_{h_1,h_2}$  is reductive. Hence,
the $G$-diagonal orbit of $(h_1,h_2)$
is closed in $\gg$, cf. eg. \cite{GOV}.
Moreover, this $G$-orbit is the unique
closed $G$-orbit 
contained in the closure
of the $G$-orbit of the pair $(x,y)$
since $G$-invariant polynomials on $\gg$
separate closed $G$-orbits.
By the Hilbert-Mumford criterion, we deduce
that  there exists a suitable one parameter
subgroup $\gamma:\ \C^\times\to G$
such that, conjugating the pair $(h_1,h_2)$
if necessary, on gets
$\underset{^{z\to0}}\lim\ \Ad \gamma(z)(x,y)=(h_1,h_2)$.

Now, let $\fa$ be the Lie subalgebra of $\g$
generated by the elements $x$ and $y$
and let $\t$ be a Cartan subalgebra
such that $\underset{^{z\to0}}\lim\ \Ad \gamma(z)(x,y)\in\t\times\t$.
We deduce from the above that
one has
$\underset{^{z\to0}}\lim\ \Ad \gamma(z)\big([\fa,\fa]\big)\sset [\t,\t]=0$.
This implies that any element of $[\fa,\fa]$ is nilpotent.
Hence, $[\fa,\fa]$ is a nilpotent Lie algebra,
by Engel's theorem. We conclude that $\fa$ is
a solvable  Lie algebra. Hence, there exists
a Borel subalgebra $\b$ such that $\fa\sset\b$
and (i) is proved.

To prove (ii), observe that
the elements $h_x,h_y,n_x,n_y$ generate
an abelian Lie subalgebra of $\g$.
Hence, there exists a Borel subalgebra
$\b$ that contains all of them.
Choose a Cartan subalgebra $\t\sset\b$ such that
$h_x,h_y\in\t$.
Then, the argument at the beginning of the proof shows
that the pair $(h_x,h_y)$
is contained in the closure of the $G$-diagonal
orbit of ~$(x,y).$ 
\end{proof}

Note that, by definition, we have $\gg^\b=[\mmu(\tgg)]_\red$, 
the reduced image of the
morphism $\mmu: \tgg\to\gg$, see \S\ref{tx}.
The following result implies Proposition \ref{tb}(ii).

\begin{prop}\label{nonormal} The morphism 
$\mmu\times\nnu:\ \tgg\to[\gg\times_{\gg/\!/G}\tt]_\red$
 is birational and proper.

The first projection $\gg\times\tt\to\gg$ induces a birational
finite morphism $[\gg\times_{\gg/\!/G}\tt]_\red\to\gg^\b$.
\end{prop}
\begin{proof} Let $\gg^\heartsuit$ be the set of pairs
$(x,y)\in \gg$ such that the vector space $\C x+\C y\sset \g$
spanned by $x$ and $y$ is $2$-dimensional and, moreover,  any nonzero
element of that vector space is regular in $\g$.
 It is clear
that $\gg^\heartsuit$ is a $G$-stable Zariski open and dense
subset of $\gg$. 

According to \cite{CM}, Lemma 6(i),  the set
$(\b\times\b)\cap\gg^\heartsuit$
is nonempty, for any Borel subalgebra $\b$. Hence, this set
is Zariski open and dense in $\b\times\b$.
Furthermore, it follows from  \cite{CM}, Lemma 7(i) and  Lemma 8(i),
 that for any pair $(x,y)\in \gg^\heartsuit$
there is at most one Borel subalgebra  that contains both
$x$ and $y$ (in \cite{CM}, this result is attributed to Bolsinov
\cite{Bol}).
Hence, the map $\mmu$ restricts to a bijection
$\mmu\inv(\gg^\heartsuit)\iso \gg^\b\cap\gg^\heartsuit$.
Both statements of the proposition follow from this
since the map $\mmu$ is proper.
\end{proof}

\subsection{A stratification of the isospectral
commuting variety}\label{ss2} Below, we will have to consider several
 reductive
Lie algebras at the same time. To avoid confusion,
we write $\zz(\fl)$ for the commuting
scheme of a  reductive
Lie algebra $\fl$ and use similar notation
%lll
for other objects associated with $\fl$.

Let  ${\mathfrak N}(\fl)$ be the 
{\em nilpotent} commuting variety of $\fl$,
the variety 
of pairs of commuting {\em nilpotent} elements
of $\fl$, equiped with reduced scheme structure.
It is clear that we have ${\mathfrak N}(\fl)=
{\mathfrak N}(\fl'),$ where $\fl':=[\fl,\fl]$, the derived Lie algebra
of $\fl$.
According to \cite{Pr}, the irreducible components
of ${\mathfrak N}(\fl)$ are parametrized
by the conjugacy classes of {\em distinguished}
nilpotent elements  of $\fl'$. The irreducible
component corresponding to such a conjugacy class
is equal to the closure in $\fl'\times\fl'=T^*(\fl')$
of the total space of the conormal bundle on that
conjugacy class. It follows, in particular, that
the dimension of each irreducible component of the variety
${\mathfrak N}(\fl)$ equals $\dim\fl'$.

Fix a reductive connected group $G$
and a
 Cartan subalgebra
 $\t\sset\g=\Lie G$.
Recall that the centralizer of an element of $\t$ is called
a standard Levi subalgebra of $\g$. Let $S$ be the set
of standard Levi subalgebras.
Given a  standard Levi subalgebra $\fl$,
let $\ts_\fl\sset \t,$ resp. 
$\tts_\fl\sset\tt$,
 denote the set of elements
$h\in \t$, resp. $(h_1,h_2)\in \tt,$ such that we have $\g_h=\fl$,
resp. $\g_{h_1,h_2}=\fl$.
Let  $\t_\fl$ denote the center of $\fl$.
It is clear  that $\ts_\fl$ is an irreducible Zariski open dense
subset of $\t_\fl$, resp.
$\tts_\fl$  is an irreducible  Zariski open dense
subset of $\t_\fl\times\t_\fl$.
We get a stratification $\t=\sqcup_{\fl\in S}\ \ts_\fl,$
resp. $\tt=\sqcup_{\fl\in S}\ \tts_\fl.$

%For each $\fl\in S$, there is an induced
%stratification of the center of 
%the Lie algebra $\fl$ as follows
% vector space $\t_\fl$ as follows
%$\t_\fl=\sqcup_{\{\fk\in S\ |\ \fk\supset\fl\}}\ \ts_\fk.$
%Thus, $\ts_\fl$ is the unique open
%stratum in $\t_\fl$.
%In particular, for $\fl=\t$, we have  that $\ts_\t=\t^r$ is
%a Zariski open subset of $\t$.
%Further, from the definition, we find
%\beq{tts}
%\tts_\fl\ =\ 
%\bigsqcup_{\{\fk_1,\fk_2\in S\ |\
%\t_{\fk_1}+\t_{\fk_2}=\t_\fl\}}\ 
%\ts_{\fk_1}\times\ts_{\fk_2}\ =\
% \bigsqcup_{\{\fk_1,\fk_2\in S\ |\ \fk_1\cap\fk_2=\fl\}}\ 
%\ts_{\fk_1}\times\ts_{\fk_2}.
%\eeq
%We see that $\ts_\fl\times\ts_\fl$ is
%an open dense subset of $\tts_\fl$. 
%We conclude that $\tts_\fl$ is an irreducible variety
%and, we have
%\beq{tts2}
%$\dim\tts_\fl=2\dim\ts_\fl.$

%The $G$-diagonal action on $\gg$ induces a well defined map
%\beq{f}f_\fl:\
%G\times_{N(L)} (\fl\times\fl)\to 
%\g\times\g=\gg,\quad
%(g,x_1,x_2)\mto \big(\Ad g(x_1),\,\Ad g(x_2)\big).
%\eeq

%\subsection{}\label{ss3}
Let $\xx=\xx(\g)$, the isospectral commuting variety of $\g$,
and let $p_{_\tt}:
\xx(\g)\to\tt$ be the projection.
For each standard Levi subalgebra $\fl$ of
$\g$, we put $\xx_\fl(\g):=(p_{_\tt})\inv(\tts_\fl).$
Thus, we get a partition
 $\xx(\g)=\sqcup_{\fl\in S}\,\xx_\fl(\g)$
by $G$-stable locally closed, not necessarily smooth, subvarieties.

We consider a map $G\times\gg\times\tt\to \gg\times\tt$ given by
the following assignment
\beq{map}
g\times (y_1,y_2)\times (t_1,t_2)\ \mto\
 \big(\Ad g(y_1+t_1),\,\Ad g(y_2+t_2)\big)\times (t_1,t_2).
\eeq

\begin{lem}\label{codim} For any standard Levi subalgebra $\fl$ 
with  Levi subgroup of $L\sset G$, the map \eqref{map}
induces a $G$-equivariant isomorphism
$$
\big(G\times_{ L}{\mathfrak N}(\fl)\big)\,\times\,\tts_\fl
\ \iso\ \xx_\fl.
$$

The second projection $\big(G\times_{ L}{\mathfrak N}(\fl)\big)\,\times\,
\tts_\fl\to\tts_\fl$
goes, under the  isomorphism,
to the map $p_{_\tt}:\, \xx_\fl\to\tts_\fl$.
Thus,
all irreducible
components  of the set $\xx_\fl$
have the same dimension equal to $\dim\g+\dim\t_\fl$
 and are in one-to-one correspondence
with the distinguished nilpotent conjugacy classes in
the Lie algebra $\fl$.
\end{lem}

\begin{proof} 
Let $(x_1,x_2,t_1,t_2)\in\xx_\fl$.
Lemma \ref{lss}(ii) implies that,
for any polynomial
$f\in\C[\zz]^G$, we have $f(x_1,x_2)=f(h_{x_1},h_{x_2})$.
We know that the semisimple pair $(h_{x_1},h_{x_2})$
is  $G$-conjugate to an element of $\tt$,
that
  $W$-invariant polynomials separate
$W$-orbits in $\tt$, and that the
restriction map \eqref{Jo} is surjective, \cite{Jo}.
It follows 
that the pair $(h_{x_1},h_{x_2})$ is
$G$-conjugate to the pair $(t_1,t_2)$.
Hence, the pair $(x_1,x_2)$ is
$G$-conjugate to a pair 
of the form $(t_1+y_1,t_2+y_2)$
for some $(y_1,y_2)\in {\mathfrak N}(\fl)$.
The isomorphism of the lemma easily follows from this.

Now, for any irreducible component $V$ of 
${\mathfrak N}(\fl)$, using the result of Premet mentioned
above, we find
\begin{multline*}\dim\big((G\times_{ L}V) \times\tts_\fl\big)=
\dim G-\dim L+\dim V+2\dim\t_\fl\\
=
\dim\g-\dim(\t_\fl+[\fl,\fl])+\dim[\fl,\fl]+2\dim\t_\fl=\dim\g+\dim\t_\fl.
\end{multline*}
This proves the dimension formula for irreducible components
of the variety $\xx_\fl$.
\end{proof}

We are now ready to complete the proof of Lemma
\ref{xxbasic}.

\begin{cor}\label{xdense} The set $\xx^{rs}$ is
an irreducible and Zariski
dense subset of $\xx$.
\end{cor}
\begin{proof} Let $\zz^{ss}$ be the set of semisimple
pairs $(x,y)\in \zz$ and let $\xx^{ss}_\fl:=\xx_\fl\cap p\inv(\zz^{ss}).$
 Lemma \ref{codim} implies  that,
for any Levi subalgebra $\fl$, 
any element $(x,y,t_1,t_2)\in\xx^{ss}_\fl$  is conjugate to the
element
$(x',y',t_1,t_2)$ for some $(x',y')\in\tts_\fl$.
In the special case where  $\fl=\t$, we have 
$\xx^{rs}=\xx_\t$. Hence,  Lemma \ref{codim}
implies that $\xx^{rs}$ is irreducible and, moreover, we have
 $\xx^{ss}_\fl\sset\overline{\xx^{rs}}$,
 for any standard Levi subalgebra $\fl$.

Thus, using  the isomorphism of   Lemma \ref{codim}
one more time, we see that proving 
 the Corollary reduces to showing that
${\mathfrak N}(\fl)$, the nilpotent commuting
variety of  $\fl$,
is contained in the closure of the set $\zz^{ss}(\fl)$.
But we have $\zz^{ss}(\fl)\supseteq\zz^{rs}(\fl)$
and 
the set $\zz^{rs}(\fl)$ is dense in $\zz(\fl)$,
by  Proposition
\ref{zzbasic}(i); explicitly, this
is Corollary 4.7 from \cite{Ri1}. The result follows.
\end{proof}

\subsection{Proof of Lemma \ref{irr}}\label{ss3} 
We have the
projection
$p :\xx(\g)\to\zz(\g)$ and, for
any standard Levi subalgebra $\fl\sset\g$,
put $\zz_\fl(\g)=p (\xx_\fl(\g))$
where we use the notation of Lemma \ref{codim}.
Thus, one has $\zz(\g)=\cup_{\fl\in S}\ \zz_\fl(\g).$
Note that,
for a pair of standard Levi subalgebras $\fl_1,\fl_2\sset\g$,
 the corresponding pieces $\zz_{\fl_1}(\g)$ and $\zz_{\fl_2}(\g)$
are equal whenever the Levi subalgebras
$\fl_1$ and $\fl_2$ are conjugate in $\g$;
otherwise these two pieces are disjoint.

Let $\fl$ be a   standard Levi subalgebra in $\g$.
We are interested in the dimension of the set
$\zz_\fl(\g)\sminus \zz^{rr}$.
The dimension formula of Lemma \ref{codim}
shows that the codimension of the set $\zz_\fl(\g)$ 
in $\zz(\g)$ equals $\dim(\t/\t_\fl)$.
We see that to prove Lemma \ref{irr} it suffices to show
that, for any standard Levi subalgebra $\fl\sset\g$
such that the codimension of $\t_\fl$ in $\t$
equals either $0$ or $1$,
the set $\zz_\fl(\g)\sminus \zz^{rr}$
has codimension $\geq 2$ in $\zz(\g)$.

In the first case we have $\t=\t_\fl=\fl.$
 Thus, 
one has
${\mathfrak N}(\fl)=\{0\}$, so ${\mathfrak N}(\fl)+\tts_\fl=\tts$,
where we have used simplified notation  $\tts:=\tts_\t$.
The set  $\tts\sminus \zz^{rr}$ consists of
the pairs $(h_1,h_2)\in\tts$ such that
neither $h_1$ nor $h_2$ is regular.
Therefore, each of these two elements
belongs to some root hyperplane in $\t$,
that is, belongs to a finite union of codimension
1 subspaces in $\t$.
We conclude that the set $\tts\sminus \zz^{rr}$
has codimension
$\geq 2$ in ${\mathfrak N}(\fl)+\tts_\fl$, as required.

Next, let $\dim(\t/\t_\fl)=1$.
In that case,
$\fl$ is a {\em minimal}  Levi subalgebra of $\g$.
Thus, there is a root $\alpha\in \t^*$ in
 the root system of $(\g,\t)$ such
that 
$\t_\fl=\Ker\alpha$ is a codimension 1
 hyperplane in $\t$.
We have $\fl=\t_\fl\oplus[\fl,\fl]$
where
 $[\fl,\fl]$ is an  $\sl_2$-subalgebra
of $\g$ associated with the root $\alpha$.
It is easy to see that ${\mathfrak N}(\sl_2)$
is an irreducible variety formed
by the pairs of nilpotent elements
proportional to each other
(the zero element is declared to be
proportional to any
element).
Thus, ${\mathfrak N}(\fl)+\tts_\fl$ is an irreducible variety.

To complete the proof of the lemma in this case,
we must show that the complement
of the set $U=({\mathfrak N}(\fl)+\tts_\fl)\cap \zz^{rr}$
has codimension $\geq 1$ in ${\mathfrak N}(\fl)+\tts_\fl$.
The set $U$ is a Zariski
open subset in an irreducible variety. Thus, it suffices to show that
the set $U$ is nonempty.
For this, pick  $h\in \ts_\fl$ and let $n\in \sl_2$
be any nonzero nilpotent element.
Then, $n$ is a regular element of 
the Lie algebra $[\fl,\fl]=\sl_2$;
hence  $h+n$ is a regular element of $\g$.
Therefore,  we have 
$(h+n,h+n)\in U$,
and we are done.
\qed

\subsection{Proof of 
Theorem \ref{rrsmooth}(i)}\label{smooth_pf} 
 The property of a coherent sheaf  be
Cohen-Macaulay is stable under taking  direct 
images by finite
morphisms and  taking direct
summands. Thus,  
 Theorem \ref{main_thm} implies that the sheaf $\rr$,
as well as the isotypic components $\rr^E$
for any  $W$-module $E$, is
 Cohen-Macaulay.

The scheme $\xx_\norm/W$ is reduced
and integrally
closed,
as a quotient of an integrally
closed  reduced scheme
by a finite group action, \cite{Kr}, \S3.3. 
Further, by  Lemma
\ref{xxbasic}(ii), the map   $p_\norm$ is generically a Galois
covering with $W$ being the Galois group.
It follows that the induced map $\xx_\norm/W\to\zz_\norm$
is a finite and birational morphism of normal varieties. Hence it is
  an isomorphism, that is,
 the canonical morphism
$\oo_{\zz_\norm}\to ((p_\norm)_*\oo_{\xx_\norm})^W=\rr^W$
is an isomorphism.
This yields an isomorphism
$\xx_\norm/W\cong \zz_\norm$ and implies Corollary ~\ref{zzcm}.

Observe next that the sheaf  $\rr|_{\zz^r}$ is a
 Cohen-Macaulay sheaf  on a smooth variety, hence it is locally free.
 The fiber of the corresponding algebraic vector bundle over any point 
of the open set $\zr\sset \zz^r$ affords
the regular representation of the Weyl group
$W$, by Lemma
\ref{xxbasic}(ii). The statement of  Theorem \ref{rrsmooth}(i) follows
from this by continuity since
the set $\zr$ is dense in ~$\zz^r$.

\subsection{Proof of  Theorem \ref{rrsmooth}(ii)
and of  Theorem \ref{charb}(i)}\label{txr} 

We use the notation of Definition \ref{U} and
let $\tx_1:=\mmu\inv(\zz_1)$. 
Write 
$\wt q:\ \tx_1\to \tg^r,\
(\b,x,y)\mto (\b,x)$, resp.
$q_1:\ \zz_1\to\g^r,\ (x,y)\mto x$, for  
 the natural projection, and let
 $\ppi$  be  the map from  Corollary \ref{transv}.
%Using the isomorphism $\xx_1=N_\yr,$ of Lemma \ref{conormal}, 
Thus, one has a commutative diagram
\beq{nn}
\xymatrix{
\tx_1\ \ar[d]^<>(0.5){\wt q}\ar[rr]^<>(0.5){\ppi}_<>(0.5){\sim}&&
\ \xx_1\ \ar[d]^<>(0.5){q}\ar[rr]^<>(0.5){p}&&
\ \zz_1\ \ar[d]^<>(0.5){q_1}\\
\tg^r\ \ar[rr]^<>(0.5){\pi}_<>(0.5){\sim}&&\ 
\yr\ar[rr]^<>(0.5){pr}&&\ \g^r
}
\eeq

\begin{lem}\label{carsq} The right square in diagram
\eqref{nn} is cartesian.
\end{lem}

\begin{proof}
It is immediate from Lemma \ref{centb} that
the  map $\wt q\times (p\ccirc \ppi):\
\tx_1\ \to\ \tg^r\,\times_{\g^r}\, \zz_1$
that results from the diagram
is a set theoretic bijection.
 Further, all the varieties involved in  diagram \eqref{nn}
are smooth and the map $q_1$ is a smooth morphism (the vector bundle
projection $N_\yr\to\yr$, cf. Lemma \ref{conormal}). 
It follows that $\tg^r\,\times_{\g^r}\, \zz_1$
is a smooth variety and, moreover, the
bijection above gives
 an isomorphism $\tx_1\ \iso\
\tg^r\,\times_{\g^r}\, \zz_1,$
of algebraic varieties.
We deduce that the rectangle along the perimeter
 of diagram \eqref{nn} is a cartesian square.

To complete the proof, we observe that the
map $\ppi$ in \eqref{nn} is an
isomorphism by Corollary \ref{transv},
resp.  the
map $\pi$  is an
isomorphism by Proposition \ref{tb}(i).
We conclude that 
the right square in diagram \eqref{nn} is  cartesian as well.
\end{proof}

The morphism $pr$ in diagram
\eqref{nn}  is finite and the morphism
$q_1$ is smooth. So, thanks to Lemma \ref{carsq},
we may apply smooth base change for the  cartesian  square 
on the right of
\eqref{nn}. Combining this with smooth base change 
 for  the  cartesian  square on the right of
diagram \eqref{cart}, yields
a chain of natural $G\times W\times\CC$-equivariant sheaf isomorphisms 
\beq{rrcart}
\rr|_{\zz_1}\ =\ p_*\oo_{\xx_1}\ =\ 
p_* q^*\oo_\yr\ =\ q_1^* pr_* \oo_\yr\ =\ 
q_1^*pr_*(\wt\gamma^*\oo_\t)\ =\ 
q_1^*\gamma^*(\vartheta_*\oo_\t).
\eeq

Therefore, for any $W$-representation $E$, we deduce
\beq{rrcart2}
\rr^E|_{\zz_1}\ =\  (E\o q_1^* pr_* \oo_\yr)^W\ =\ 
q_1^*\big((E\o pr_* \oo_\yr)^W\big)\ =\ 
q_1^*\gamma^*\big((E\o \vartheta_*\oo_\t)^W\big).
\eeq

Observe next that, for any 
$(x,y)\in\zz_1$, the Lie algebra $\g_x$ is abelian, hence  we have
$\g_{x,y}=\g_x\cap\g_y=\g_x$. From this, 
writing $\fz_1:=\fz_{\zz_1}$ for short, we get
a natural isomorphism $\fz_1=q_1^*\fz_{\g^r}$,
of sheaves on $\zz_1$.
We deduce that the sheaf $\fz_1$ is a
 locally free. Furthermore,  applying to functor $q_1^*(-)$
to the isomorphism of Theorem \ref{BBprop}(i)
and using \eqref{rrcart2} one
obtains the following isomorphisms
\beq{aabc}\wedge^k\fz_1\ =\ q_1^*(\wedge^k\fz_{\g^r})\
\iso\ 
q_1^*\big((\wedge^k\t\o pr_* \oo_\yr)^W\big)\ =\ 
\rr^{\wedge^k\t}|_{\zz_1}.
\eeq

Now, let $L$ be a rational $G$-module
such that the set of weights of $L$ is contained in the root lattice.
Then,   a similar argument yields a natural isomorphism
$L^{\fz_1}=q_1^*(L^{\fz_{\g^r}}).$
Therefore the sheaf
 $L^{\fz_1}$ is 
 locally free.
Moreover,  applying the functor $q_1^*(-)$
to the isomorphism of Theorem \ref{BBprop}(ii)
and using \eqref{rrcart2} again, one similarly
obtains
an isomorphism
\beq{aabcL}
L^{\fz_1}
\ \iso \ \rr^{ L^\uh}|_{\zz_1}.
\eeq

Similar considerations  apply, of course, in the case 
where the set $\zz_1$ is replaced by
the set $\zz_2$.
It follows, in particular,
that $\dis \fz_{rr}:=\fz_{\zz^{rr}}$ and  $\dis L^{\fz_{rr}}:=L^{\fz_{\zz^{rr}}}$
are locally free coherent sheaves on
$\zz^{rr}=\zz_1\cup\zz_2$.
However, it is not  clear {\em a priori},
 that the `$\zz_2$-counterparts' of   morphisms
 \eqref{aabc}-\eqref{aabcL} 
agree with those in 
 \eqref{aabc}-\eqref{aabcL} on the overlap $\zz_1\cap\zz_2$.

To overcome this difficulty, we now
 produce an independent direct construction
of canonical morphisms
\beq{directconstr}
\la^k_{rr}:\ \wedge^k\fz_{\zz^{rr}}\ \to\
\rr^{\wedge^k\t}|_{\zz^{rr}},
\quad\oper{resp.}\quad
\la^L_{rr}:\ L^{\fz_{rr}}\ \to\ \rr^{ L^\uh}|_{\zz^{rr}}.
\eeq

This will be done by
adapting the strategy of
\S\ref{morphism_sec} as follows. Let
 $\xx^{rr}=p\inv(\zz^{rr})$ and
$\tx^{rr}=\mmu\inv(\zz^{rr})$.
We know by the above that the  sheaf $\mmu^*\fz_{rr}$, resp. $\mmu^*(L^{\fz_{rr}})$, is 
locally free, being a pull-back of a locally free
 coherent sheaf on
$\zz^{rr}.$

Now let $(x,y,\b)\in \tx^{rr}.$ Then, 
$x,y\in\b$ and, moreover, we have that
either $\g_{x,y}=\g_x$ or $\g_{x,y}=\g_y$.
In each of the two cases, applying Lemma \ref{centb},
we deduce an inclusion $L^{\g_{x,y}}\sset
L^{\ccl0}$. Therefore, one gets, as in \S\ref{morphism_sec},
 a well defined morphism
$\mmu^*(L^{\fz_{rr}})\to L^\uh\o\oo_{\tx^{rr}}.$
We may further transport this morphism via  $\ppi$, the
isomorphism of  Corollary  \ref{transv}(ii).
This way, one constructs a canonical morphism
$f:\ p^*(L^{\fz_{rr}})\to L^\uh\o\oo_{\xx^{rr}}.$
The morphisms in \eqref{directconstr}
are now defined from the  morphism $f$, by adjunction,
mimicing the construction of \S \ref{morphism_sec}.

\begin{lem}\label{ZZiso} The restriction
of the morphism 
$\la^k_{rr}$, resp. 
$\la^L_{rr}$, in \eqref{directconstr},
to the open set $\zz_1\sset\zz^{rr}$ reduces
to  isomorphism  \eqref{aabc},
resp. \eqref{aabcL}. 

Similar claim holds in the case of the set $\zz_2\sset\zz^{rr}$.
\end{lem}

\begin{proof} All the sheaves involved are
$G$-equivariant and locally free.
Hence,
it suffices to check  the
statement of the lemma fiberwise,
and only
at the points of the form $(x,y)\in\t^r\times\t^r$.
In that case verification 
is straightforward and is left for the reader.
\end{proof}

Lemma  \ref{ZZiso} implies that each of the
morphisms in \eqref{directconstr}
is an isomorphism of locally free sheaves
on $\zz^{rr}.$

To complete the proof of Theorem \ref{charb}(i)
write $j:\zz^{rr}\into\zz_\norm$ be the open imbedding.
We know that the sheaf $\rr^E$ is
Cohen-Macaulay for any $W$-representation $E$.
It follows, that the canonical 
morphism $\rr^E\to j_*(\rr^E|_{\zz^r})$ is an
isomorphism, cf.
Lemma \ref{ei}. Similarly, by Lemma \ref{tors}(ii),
we have  a canonical
isomorphism $L^{\fz_{\zz_\norm}}\iso j_*(L^{\fz_{\zz^r}}).$
Thus, applying the functor $j_*(-)$
to the isomorphism $L^{\fz_{\zz^r}}\cong\rr^{ L^\uh}|_{\zz^r}$ proved earlier,
we deduce
$L^{\fz_{\zz_\norm}}=j_*(L^{\fz_{\zz^r}})\cong j_*(\rr^{ L^\uh}|_{\zz^r})
\cong \rr^{ L^\uh}$ and  Theorem \ref{charb}(i) follows.

A similar (even simplier) argument proves Theorem \ref{rrsmooth}(ii).

\begin{rem} In the case of a not necessarily small
representation $L$, one can still 
 obtain a description of the image of the corresponding morphism
$\la^L_{rr}$, in \eqref{directconstr}, in the spirit of
Theorem \ref{knv}. That description is not very useful, however,
since
 the sheaf $L^{\fz_\zz}|_{\zz^r}$ turns out  to be
{\em not} locally free, in general,  already  for
 $\g=\sl_2$.
\end{rem}

\subsection{Proof of Corollary \ref{vectbun}}\label{pfvb}
The  isomorphism $\oo_{\zz_\norm}\cong\rr^W$,
in Corollary \ref{vectbun}(i), has been already established 
at the beginning of \S\ref{smooth_pf}.

\begin{proof}[Proof of the isomorphism
$\ \kk_{\zz^r}\cong \rr^\sign|_{\zz^r}$]
Given a point $(x,y)\in\zz$,  let
$T_{x,y}\zz$, resp. $T^*_{x,y}\zz$, be the Zariski tangent,
resp. cotangent, space
to $\zz$ at $(x,y)$.
Let  $\kap_*:\ \g\oplus\g\to\g$
be the differential of
the commutator map $\kap$ at the point $(x,y)$.
Then, one has an exact sequence of vector spaces
$$
\xymatrix{
0\ar[r]& T_{x,y}\zz\ar[r]& 
\g\oplus\g\ar[rr]^<>(0.5){\kap_*}&&
\g\ar[r]&\coker(\kap_*)\ar[r]&0.
}
$$

We use  an invariant form
on $\g$ to identify $\g^*$ with $\g$ and
write $\kap_*^\top$ for the linear map
dual to the map $\kap_*$. Then,
dualizing the exact sequence above yields an exact sequence
\beq{inequ}
\xymatrix{
0&  T^*_{x,y}\zz \ar[l] & 
\g\oplus\g\ar[l]&& \g\ar[ll]_<>(0.5){\kap_*^\top}&
\Ker(\kap_*^\top) \ar[l]& 0\ar[l]
}
\eeq

Now, the map $\kap_*$ is given by the  formula
$\kap_*:\ (u,v)\mto [x,u]-[y,v]$.
Using the invariance of the bilinear form
one easily finds that the dual map
 is given by the formula
$\kap_*^\top:\ a\mto [x,a]\oplus [y,a].$
We conclude that $\Ker(\kap_*^\top)=\g_{x,y}$.

Write $\det$  for the top exterior power of a vector
space. From \eqref{inequ},
  we deduce a canonical isomorphism
$\det T_{x,y}^*\zz\ 
 =\ \det(\g\oplus\g)\o(\det\g)\inv\o\det\g_{x,y}.$
For $(x,y)\in \zz^r$, we have $\det\g_{x,y}=
\wedge^\rk\g_{x,y}=\wedge^\rk\fz_\zz|_{(x,y)}$. 
Therefore, a choice of base vector in
the 1-dimensional vector space $\det\g$ determines, for all
$(x,y)\in\zz^r$, an  isomorphism
$\det T_{x,y}^*\zz\ccong\wedge^\rk\fz_\zz|_{(x,y)}$. 
This yields an isomorphism
$\kk_{\zz^r}\ccong\wedge^\rk\fz_{\zz^r}$
of locally free sheaves.
\end{proof}

We can now complete the proof of Corollary \ref{vectbun}.
We know that $\zz_\norm$ is a Cohen-Macaulay
variety and that the sheaf $\rr$ on $\zz_\norm$
is isomorphic to its Grothendieck dual
$R\!\hhom_{\oo_{\zz_\norm}}(\rr,\ \kk_{\zz_\norm})$.
It follows that
for any $j\neq 0$ one has $R^j\!\hhom_{\oo_{\zz_\norm}}(\rr,\
\kk_{\zz_\norm})=0$
and, moreover,
there is an isomorphism
$\rr\cong$ $\hhom_{\oo_{\zz_\norm}}(\rr,\ \kk_{\zz_\norm})$.
Further, since $\oo_{\zz_\norm}$ is a direct summand
of $\rr$ the sheaf $\kk_{\zz_\norm}$
is  a direct summand
of $\hhom_{\oo_{\zz_\norm}}(\rr,\ \kk_{\zz_\norm})\cong\rr$.
Hence,  $\kk_{\zz_\norm}$ is a Cohen-Macaulay sheaf.
Similarly, the sheaf $\rr^\sign$ is also  Cohen-Macaulay.

Recall that two  Cohen-Macaulay sheaves are isomorphic
if and only if they have isomorphic restrictions to a complement of
a closed subset of codimension $\geq 2$.
The isomorphism
$\kk_{\zz_\norm}\cong \rr^\sign$ of
Corollary \ref{vectbun}(i) now follows
from the chain of isomorphisms
$\kk_{\zz_\norm}|_{\zz^r}\ \cong\ \wedge^\rk\fz_{\zz^r}\ \cong\ 
(\wedge^\rk\t\o\rr)^W|_{\zz^r} \ =\  \rr^\sign|_{\zz^r}$
where the second isomorphism holds by  Theorem \ref{rrsmooth}(ii).

Finally, the  isomorphism of part (ii) of
Corollary \ref{vectbun} follows,
thanks to  the isomorphism $\kk_{\zz_\norm}\cong
\rr^\sign$, by
 equating the corresponding
 $W$-isotypic
components on each side of
the self-duality  isomorphism
$\hhom_{\oo_{\zz_\norm}}(\rr,\ \rr^\sign)\cong\rr$
proved earlier.
\qed

\begin{rem}\label{zzregpf} The short exact sequence  \eqref{inequ}
implies that,
for any $(x,y)\in\zz$, one has
$$
\rk+\dim\g=\dim\zz\leq\dim T_{x,y}^*\zz
=2\dim\g-(\dim\g-\dim\Ker\kap_*^\top)=
\dim\g+\dim\g_{x,y},
$$
where in the first equality we have used
Proposition \ref{zzbasic}(i).
We deduce an inequality $\rk\leq \dim\g_{x,y}$. Moreover,
we see that this  inequality becomes
an equality  if and only if $(x,y)$ is
a smooth point of the scheme  $\zz$.
This proves Proposition \ref{zzbasic}(ii).
\end{rem}

\subsection{Proof of  Theorem \ref{charb}(ii)-(iii)}\label{pfvb2}
We have the following chain of  natural $W$-equivariant
algebra maps
\begin{multline*}
\C[\tt]\ =\ \C[\tt]^W\o_{\C[\tt]^W}\C[\tt]\ =\ 
\C[\zz_\red]^G\o_{\C[\zz_\red]^G}\C[\tt]\\=\ 
\C[\zz_\red\times_{\zz_\red/\!/G}\tt]^G
\ \onto\  \C[\xx]^G\ \into\  \C[\xx_\norm]^G.
\end{multline*}

Let $f$ be the composition of the above maps. 
It follows from the
isomorphism
$\zz^{rs}\cong G\times_{N(T)}\tts_\t$ that
the map $f$ induces  an isomorphism between the
fields of fractions of the algebras
$\C[\tt]$ and $\C[\xx_\norm]^G,$ respectively.
The algebra  $\C[\xx_\norm]^G$
 is, by definition, a
finitely generated $\C[\xx]^G$-module, hence, also
a finitely generated module over
the image of $f$. 

Thus, since the algebra $\C[\tt]$
is integrally closed, we obtain
$\Ga(\zz_\norm,\ \rr)^G=\C[\xx_\norm]^G=\C[\tt].$
Equating $W$-isotypic components
on each side of this isomorphism
yields the first isomorphism of Theorem \ref{charb}(ii).

To prove the second isomorphism, we compute
\begin{align*}
(L\o\C[\zz_\norm])^G\ &=\ \Ga(\zz_\norm,\ L\o\oo_{\zz_\norm})^G\\
&=\ \Ga(\zz_\norm,\ L^{\fz_{\zz_\norm}})^G\quad\text{by Lemma \ref{tors}(i)}\\
&=\ \Ga(\zz_\norm,\ \rr^{ L^\uh})^G\quad\text{by Theorem
\ref{charb}(i)}\\
&=\ ( L^\uh\o \C[\tt])^W\quad\text{by the previous paragraph}.
\end{align*}

It is immediate to check, by restricting 
to the open set of regular semisimple pairs, 
that the composition of the chain of isomorphism above goes,
via the identification $ L^\uh = L^\t$
induced by the imbedding $i: \t\into\g$,
to the  restriction homomorphism
$i^*: (L\o\C[\zz_\norm])^G\to(L^\t\o \C[\tt])^W$.
This proves part (ii) of Theorem \ref{charb}.
\medskip

To prove part (iii) we use  Theorem \ref{rrsmooth}(ii).
From that theorem,
we deduce $\Ga(\zz^r,\ \rr^{\wedge^s\t})\ccong
\Ga(\zz^r,\ \wedge^s\fz_{\zz^r})$,
for any $s\geq0$.
Further, we know that the set $\zz_\norm\sminus\zz^r$
has codimension $\geq2$ in $\zz_\norm$ and that
$\rr^{\wedge^s\t}$ is a Cohen-Macaulay sheaf on $\zz_\norm$.
It follows that the natural restriction map
induces an isomorphism
$\Ga(\zz_\norm,\ \rr^{\wedge^s\t})\iso\Ga(\zz^r,\ \rr^{\wedge^s\t}).$
The proof  is now completed by the following chain
of isomorphisms:
$$
\Ga(\zz^r,\ \wedge^s\fz_{\zz^r})^G\ccong
\Ga(\zz^r,\ \rr^{\wedge^s\t})^G\ccong
\Ga(\zz_\norm,\ \rr^{\wedge^s\t})^G\ccong
(\wedge^s\t\o\C[\tt])^W.\eqno\Box
$$

\section{Principal nilpotent pairs}\label{pnpsec}
\subsection{Filtrations and Rees modules}\label{rees_sec}
Given a vector space $E$ we refer to
 a direct sum decomposition
$E=\bplus_{i,j\geq 0}\ E^{i,j}$
as  a {\em bigrading} on $E$. Similarly,
a collection of subspaces $F_{i,j}E\ \sset\ E,\ i,j\geq0$
such that
$F_{i,j}E\sset F_{i',j'}E$ whenever
$i\leq i'$ and $j\leq j'$ will be referred to
as a {\em bifiltration} on $E$.
Canonically associated with a bifiltration
 $F_{i,j}E$, there is a pair of bigraded vector spaces
\beq{bigr}
\gr E:=\bigoplus_{i,j\geq0}\ \gr_{i,j}E,
\quad\gr_{i,j}E:=\frac{F_{i,j}E}{F_{i-1,j}E+F_{i,j-1}E},
\qquad\text{resp.}\quad
\rees E:=\bigoplus_{i,j\geq0}\ F_{i,j}E.
\eeq

We view the polynomial algebra $\C[\ta_1,\ta_2]$ as 
bigraded algebra such that $\deg\ta_1=(1,0)$ and $\deg\ta_2=(0,1).$
Below, we will often make no distinction
between  $\rees E$ and the {\em  Rees module} of $E$
 defined as
$\sum_{i,j}\ \ta_1^i\ta_2^j\cdot F_{i,j}E
\ \sset\  \C[\ta_1,\ta_2]\o E,$
a bigraded $\C[\ta_1,\ta_2]$-submodule of a free
$\C[\ta_1,\ta_2]$-module with generators $E$.
There is a canonical bigraded space
isomorphism
\beq{grE}
\gr E\ \cong\ \rees E/(\ta_1\cdot\rees E + \ta_2\cdot\rees E).
\eeq

Given   a 
vector space  $E$ equipped with a pair of ascending filtrations
$\fo_\idot E$ and  $\ft_\idot E$, there are Rees
modules $\reo E:=\sum_i\  \ta_1^i\cdot\fo_i E
\ \sset\ \C[\ta_1]\o E,$
resp. $\ret E:=\sum_j\  \ta_2^j\cdot\ft_j E\ \sset\ \C[\ta_2]\o E$,
as well as associated graded spaces $\gro_\idot  E$, resp. $\grt_\idot E$.
One may  further define
a bifiltration on $E$  by the formula
$F_{i,j}E:=\fo_i E\cap \ft_j E, \ i,j\geq 0$.
One has the corresponding Rees $\C[\tat]$-module $\rees E$.
There are canonical
isomorphisms 
$$\C[\ta_2^{\pm1}]\,\bo_{\C[\ta_2]}\,\rees E\ \cong\ 
\C[\ta_2^{\pm1}]\,\bo\, \reo E,
\quad\text{resp.}\quad
\gr E\cong  \gro(\grt E)\cong\grt(\gro E),
$$
of  graded $\C[\ta_1,\ta_2^{\pm1}]$-modules,
resp. bigraded vector spaces. Specializing the first of the above
isomorphisms
at the point
$(\ta_1,\ta_2)=(0,1)$ and using that $\reo E/\ta_1\cd\reo E\cong\gro E$ we deduce a
canonical isomorphism of graded vector spaces
\beq{part}
\rees E/\big(\ta_1\cd \rees E+(\ta_2-1)\cd \rees E\big)\ccong
\gro E.
\eeq

Now let $E=\bplus_{i,j}\ E^{i,j}$
be a bigraded vector space  itself.
Associated naturally with the bigrading, there are
two  filtrations  on $E$ defined by
$\fo_m E := \bplus_{\{i,j\mid\  i\leq m\}}\ E_{i,j}$
and $\ft_n E := \bplus_{\{i,j\mid\  j\leq n\}}\ E_{i,j},$
respectively. Let
$F_{m,n}E:=\fo_m E\cap \ft_n E=
\bplus_{\{i,j\mid\ i\leq m,\ j \leq n\}}\ E^{i,j}$
be the corresponding bifiltration.
Further, equip the $\C[\ta_1,\ta_2]$-module $\C[\ta_1,\ta_2]\o E$ with 
a standard tensor product bigrading
$(\C[\ta_1,\ta_2]\o E)^{p,q}:=
\sum_{\{0\leq i\leq p,\ 0\leq j\leq q\}}\
\C^{i,j}[\ta_1,\ta_2]\o E^{p-i, q-j}.$

\begin{lem}\label{grfilt} For a bigraded vector space $E=\bplus_{i,j}\ E^{i,j}$,
the assignment
$$\aleph:\en\ta_1^m\ta_2^n\o u_{i,j}\
\longmapsto\
\ta_1^{m+i}\ta_2^{n+j}\cdot u_{i,j},
\qquad u_{i,j}\in E^{i,j},\en i,j,m,n\geq 0,
$$
yields a bigraded $\C[\ta_1,\ta_2]$-module isomorphism 
$\
\aleph:\ \C[\ta_1,\ta_2]\o E\ \iso\ 
\rees E.$\qed
\end{lem}

%\begin{rem}\label{similar} In the case of a
%graded (as opposed to bigraded) vector space
%$E=\bplus_i\ E^i$, one has a similar isomorphism
%$\aleph:\ \C[\ta]\o E\ \iso\ 
%\rees E,\ \ta^m\o u_i\mto\ta^{m+i}\cdot u_i.$
%\erem

This lemma is clear. Later on, we will  use the following simple result

\begin{cor}\label{drin} Let $E$ be a finite dimensional
vector space equipped with a pair of ascending filtrations
$\fo_\idot E$ and $\ft_\idot E$, respectively.
Then $\rees E$ is a finite rank free
$\C[\ta_1,\ta_2]$-module.
\end{cor}

\begin{proof} 
The filtrations $\fo_\idot E$ and $\ft_\idot E$
form a pair of partial flags in $E$.
Hence, applying the  Bruhat lemma for pairs of flags,
we deduce that there exists a bigrading
$E=\oplus_{m,n}\ E^{m,n}$ such that
the original filtrations are associated,
as has been explained before  Lemma \ref{grfilt},
with that bigrading.
The corollary now follows from the lemma.
\end{proof}

\begin{rem} For a general bifiltration on a finite dimensional
vector space $E$ that does not come from
a pair of filtrations one may have $\dim(\gr E) >\dim E$, so the 
$\C[\ta_1,\ta_2]$-module $\rees E$
is not necessarily flat, in general.
\erem

Next, let  $I\sset E$ be a subspace of a vector subspace $E$
and put  $A=E/I$. Write $\overline{V}$ for the image of
a vector subspace $V\sset E$ under the
projection $E\onto E/I$.
A filtration on $E$ induces a quotient
filtration on $A$.
Therefore, given a  pair of filtrations
$\fo_\idot E$ and  $\ft_\idot E$ on $E$
one has the corresponding quotient filtrations
$\fo_\idot A=(\fo_\idot E+I)/I$ and  $\ft_\idot A=(\ft_\idot E+I)/I$ on $A$.

There are  two,
{\em potentially different}, ways to define a
bifiltration on $A$ as follows
\begin{align*}
F\min_{i,j}A&:=\ \overline{F_{i,j}E}\ =\ 
\big[(\fo_iE \cap\ft_jE)+I\big]/I\ccong (\fo_iE \cap\ft_jE)/(\fo_iE \cap\ft_jE \cap I),\\
F\max_{i,j}A&:=\ 
\overline{F_iE}\cap \overline{F_jE}\ =\ 
\big[(\fo_iE +I)\cap (\ft_j E +I)\big]/I\ =\ 
\fo_iA\cap \ft_jA.
\end{align*}

Clearly, for any $i,j$, one has 
$F\min_{i,j}A\sset F\max_{i,j}A$ where the
 inclusion
is strict, in general. Therefore,
writing $\rees\min A$, resp. $\rees\max A$,
for the Rees module associated with the bifiltration
$F\min_{i,j}A$, resp. $F\max_{i,j}A$,
one obtains a canonical,
not necessarily injective, bigraded
$\C[\ta_1,\ta_2]$-module homomorphism $\mathsf{can}:\
\rees\min A\to\rees\max A$.

\subsection{A flat scheme over $\C^2$}\label{rees_ctt}
We have the standard grading
$\C[\t]=\bplus_{i\geq0}\ \C^i[\t]$,
resp. filtration $F_m=\C^{\leq m}[\t]=\bplus_{i\leq m}\ \C^i[\t]$, where
 $\C^i[\t]$ denotes the space of 
degree $i$ homogeneous polynomials on $\t$.
Similarly, one has a bigrading
$\C[\tt]=\bplus_{i,j}\ \C^{i,j}[\tt]$
where $\C^{i,j}[\tt]:=\C^i[\t]\o\C^j[\t].$
Associated with this bigrading,
we have the pair of filtrations
$\fo_\idot\C[\tt]$ and $\ft_\idot\C[\tt]$,
respectively, and
the corresponding  bifiltration
$F_{m,n}\C[\tt]=\fo_m\C[\tt]\cap\ft_n\C[\tt]
=\bplus_{\{i\leq m,\ j\leq n\}}\
\C^{i,j}[\tt],\ m,n\geq0,$
cf. \S\ref{rees_sec}.

Recall the setting of \S\ref{pnp_sec}. Thus,
we have 
the semisimple pair $\bh=(h_1,h_2)\in\tt$ associated
with the principal nilpotent pair $\bbe=(e_1,e_2)$. 
The coordinate ring of the finite 
subscheme  $\wbh\sset \tt$ has the form
$\C[\wbh]=\ctt/I_\bh$ where $I_\bh\sset\C[\tt]$ is
an ideal generated by  the elements
$\{f-f(\bh),\ f\in\ctt^W\}$.
The quotient algebra $\C[\wbh]=\C[\tt]/I_\bh$ inherits a pair,
$\fo_\idot\C[\wbh]$ and $\ft_\idot\C[\wbh]$ of
quotient filtrations. Associated with these filtrations,
there are bifiltrations $F_{i,j}\max{^{\,}}\C[\wbh]$,
resp. $F_{i,j}\min{^{\,}}\C[\wbh]$.

We put $\fy:=\Spec\rees\max\,\C[\wbh].$ This is an affine 
scheme that comes equipped with a $W\times\CC$-action and 
with a $\CC$-equivariant morphism $\wp:\ \fy\to\C^2$
induced by the canonical
algebra imbedding $\C[\ta_1,\ta_2]\into\ra$.

Let $\vartheta: \t\to \t/W$ be the quotient morphism
and write $\C[\vartheta\inv(W\cdot h_2)]$
for the coordinate ring of the {\em scheme theoretic} fiber
of $\vartheta$ over the orbit $W\cdot h_2$
viewed as a closed point of $\t/W$. 

\begin{lem}\label{rees} \vi The scheme $\fy$ is a reduced,
flat and finite  scheme over $\C^2$.

There are natural $W$-equivariant algebra  isomorphisms
\begin{align}
 \C[\wp\inv(0,0)]&\ccong
\gr\max\C[\wbh]=\bigoplus_{m,n\geq 0}\
\frac{F\max_{m,n}{^{\,}}\C[\wbh]}{F\max_{m-1,n}\C[\wbh]+F\max_{m,n-1}\C[\wbh]},
\label{rees1}\\
\C[\wp\inv(0,1)]&\ccong \gro\C[\wbh]\ccong
\C[\vartheta\inv(W\cdot h_2)],\label{rees2}\\
\C[\wp\inv(1,1)]&\ccong \C[\wbh].\label{rees3}
\end{align}
\vskip2pt

\vii We have $\ F\max_{\bd_1,\bd_2}\C[\wbh]=\C[\wbh]$,
where  $\bd_s=\# R^+_s,\ s=1,2,$
see \S\ref{pnp_sec}.
\end{lem}

\begin{proof} The 
isomorphism in \eqref{rees1} follows from \eqref{grE}, resp.
the isomorphism in \eqref{rees3}
follows from definitions.
The latter isomorphism implies,
by  $\CC$-equivariance,
that we have $\wp\inv(\CC)\ccong \CC\times \wbh.$

Further, the map $\wp$ is flat by
Corollary \ref{drin}. Therefore, $\fy$  is
a Cohen-Macaulay scheme such
that $\wp\inv(\CC)$ is a reduced scheme.
It follows that the scheme $\fy$  is  reduced.

To complete the proof of the lemma, we consider
the Levi  subalgebra $\g^1=\g_{h_2}$
and its Weyl group $W_1$, the
  subgroup of $W$ generated
by reflections with respect to the 
set $R^+_1\sset \t^*$ formed by the roots 
$\alpha\in R^+$ such that $\alpha(h_2)=0$.

Let $I_1\sset \C[\t]$ denote
the ideal of the orbit $W_1\cdot h_1\sset\t$
viewed as a reduced
finite  subscheme of $\t$.
The element $h_1$ has  trivial isotropy group
under the $W_1$-action, \cite[Proposition ~3.2]{pnp}.
Therefore, a standard argument
based on the fact that $\C[\t]$ is a free
$\C[\t]^{W_1}$-module shows that the ideal
$\gr^F I_1$ equals   $(\C[\t]^{W_1}_+)$, the ideal
generated by $W_1$-invariant
homogeneous polynomials of positive degree.
We deduce a chain of
graded algebra isomorphisms
\beq{wwgg}\gr^F\C[W_1\cdot h_1]\ccong\gr^F\C[\t]/\gr^F I_1\ =\
\C[\t]/(\C[\t]^{W_1}_+)\ =\
\C[(\vartheta_1)\inv(0)]
\eeq
where we have used the notation
$\vartheta_1:\ \t\to\t/W_1$ for the quotient morphism.
Thus, there is an isomorphism
$\Spec\gr^F\C[W_1\cdot h_1]\cong(\vartheta_1)\inv(0)$
of (not necessarily reduced) schemes.

Let  $pr_2:\ \wbh\to W\cdot h_2,\ w(\bh)\mto w(h_2)$ be the second
projection. We may view
 the fiber $pr_2\inv(t),\ t\in W\cdot h_2,$  as a subset
of $\t\times \{t\}\cong\t$.
Fix $m\geq 0$ and let
$f\in \fo_m\C[\wbh]$.
Then, by  definition of the filtration
$\fo_\idot\C[\wbh]$, for any $t\in W\cdot h_2$,
there exists a polynomial $\wt f_t\in \C^{\leq m}[\t]$
such that we have $\wt f_t|_{pr_2\inv(t)}=f|_{pr_2\inv(t)}$.
Conversely,  let
$f\in \C[\wbh]$ be an element such that, for any $t\in W\cdot h_2$,
there exists a polynomial $\wt f_t\in \C^{\leq m}[\t]$
such that  $\wt f_t|_{pr_2\inv(t)}=f|_{pr_2\inv(t)}.$
Then, using  Langrange
interpolation formula, one shows that $f\in \fo_m\C[\wbh]$.
Thus, we have established a natural $W\times\C^\times$-equivariant isomorphism
$$
\Spec(\gro\C[\wbh])\ccong
W\times_{W_1}\big(\Spec\gr^F\C[pr_2\inv(h_2)]\big),
$$
of (not necessarily reduced) schemes.

Note further that we have $pr_2\inv(h_2)=W_1\cdot h_1$.
Moreover, the scheme
$\Spec\gr^F\C[W_1\cdot h_1]$
is isomorphic to $(\vartheta_1)\inv(0)$
thanks  to \eqref{wwgg}.
Thus, one obtains $W\times\C^\times$-equivariant isomorphisms
\beq{this}
\Spec(\gro\C[\wbh])\ccong
W\times_{W_1}\big((\vartheta_1)\inv(0)\big)
\ccong \vartheta\inv(W\cdot h_2).
\eeq

The isomorphisms in \eqref{rees2} follow
 since we have
$$\C[\wp\inv(0,1)]=\frac{\rees\max{^{\,}}\C[\wbh]}{
\ta_1\cdot\rees\max{^{\,}}\C[\wbh]+(\ta_2-1)\cdot\rees\max{^{\,}}\C[\wbh]}
=\gro\C[\wbh]=\C[\vartheta\inv(W\cdot h_2)].$$
Here, the first isomorphism  holds by definition,
the second isomorphism is  \eqref{part},
and the last isomorphism is \eqref{this}.

To prove part (ii),
we recall that the coinvariant algebra $\C[\t]/(\C[\t]^{W_1}_+)$
is a graded algebra which is
known to be concentrated in degrees $\leq \bd_1$.
We deduce, using isomorphisms \eqref{wwgg} and \eqref{this}, 
 that the graded algebra $\gro\C[\wbh]$
is also  concentrated in degrees $\leq \bd_1$.
Thus, we have $\fo_{\bd_1}\C[\wbh]=\C[\wbh].$
By symmetry, we get $\ft_{\bd_2}\C[\wbh]=\C[\wbh].$
Part (ii) of the lemma follows.
\end{proof}

The group $\CC$ acts on $\C^2$ and on $\tt$.
This makes 
 $\C[\C^2\times\tt]=
\C[\tat]\o\ctt$  a bigraded algebra
 with respect
to the natural bigrading on a tensor product.
Let $I\bim=\bplus_{i,j}\ I\bim^{i,j}\sset
\C[\tat]\o\ctt$ be a  bihomogeneous  ideal
generated by the
set 
$$\{\ta_1^m\ta_2^n\cdot
 f_{m,n}(\bh)-f_{m,n}\en|\en
f_{m,n}\in  (\C^{m,n}[\tt])^W,\ m,n\geq0\},$$
of bihomogeneous  elements.
It is clear that we have
$\C[\C^2\times\tt]/I\bim=\C[\C^2\times_{\tt/W}\tt]$,
where the fiber product on the right involves
the map
$\C^2\to\tt/W,\ (\tau_1,\tau_2)\mto(\ta_1\cd h_1,\
\tau_2\cd h_2)\,{\opp{mod}}\,W$.

According to the definition of the map $\aleph$
of  Lemma
\ref{grfilt},
for any $f_{m,n}\in  \C^{m,n}[\tt]$, we find
$\dis\alo(\ta_1^m\ta_2^n\cdot f_{m,n}(\bh)-f_{m,n})
=
\ta_1^m\ta_2^n\cdot(f_{m,n}(\bh)-f_{m,n}).$
The right hand side here
clearly belongs to the subspace
$\ta_1^m\ta_2^n\cdot(I_\bh\cap F_{m,n}\C[\tt])$.
Hence,  for any $i,j\geq0$, one has
an inclusion
$\aleph(I\bim^{i,j})\sset \ta_1^i\ta_2^j\cdot(I_\bh\cap F_{i,j}\ctt)$.
These inclusions insure that
the map $\aleph$ descends to  a well
defined map
$\aleph:\ \C[\C^2\times\tt]/I\bim\to \rees\ctt/\rees I_\bh=
\rees \min{}^{\,}\C[\wbh]$. 

Thus, we obtain
the following chain of bigraded
$W$-equivariant $\C[\ta_1,\ta_2]$-algebra maps
\beq{hh1}
\xymatrix{
 \C[\C^2 \times_{\tt/W} \tt]=
\C[\C^2\times\tt]/I\bim\ 
\ar@{->}[r]^<>(0.5){\aleph}&
\ \rees \min\,\C[\wbh] \
\ar[r]^<>(0.5){\mathsf{can}}&
\ \rees\max\,\C[\wbh]=\C[\fy].
}
\eeq

\subsection{A 2-parameter deformation of $\bbe$}\label{map_sec} It follows from definitions that
the nilpotent elements $e_1,e_2$ and the
semisimple elements $h_1,h_2$ satisfy the following
 commutation relations, cf. \cite[Theorem 1.2]{pnp}
\beq{eh}
[e_1,e_2]=0=[h_1,h_2],\qquad
[h_i,e_j]=\delta_{ij}\cdot e_i,\quad i,j\in\{1,2\}.
\eeq

The above relations imply that  the
elements $e_1 +\ta_1\cdot h_1$ and
$e_2 +\ta_2\cdot h_2$ commute for any $\ta_1,\ta_2\in\C$.
Therefore, one can define the following map
that will play an important role in the arguments below
\beq{vka}
\vka:\ \C^2\to \zz,\quad (\tau_1,\ta_2)\ \mto\
\vka(\tau_1,\ta_2)\ =\ (e_1 +\ta_1\cd h_1,\ 
e_2 +\ta_2\cd h_2).
\eeq

\begin{lem}\label{conjug} For any $f\in\C[\zz]^G$ and $\ta_1,\ta_2\in\C$,
one has
$\dis f(\vka(\tau_1,\tau_2))\,=\,
f(\ta_1\cd h_1,\ \ta_2\cd h_2).
$
\end{lem}
\begin{proof} This easily follows from
 Lemma \ref{lss}.
An alternative   way of proving the  lemma
is based on a useful formula $\exp\ad(\ta\cdot e_i)\, (h_j)=
h_j-\delta_{i,j}\cdot\ta\cdot e_i,\ i=1,2$.
Define a holomorphic (non-algebraic)  map
$\gamma:\ \CC\to G$ as follows
$\gamma(\ta_1,\ta_2)=\exp(\frac{1}{\,\ta_1}e_1+\frac{1}{\,\ta_2}e_2)$.
Then, using the formula, we find (we note that $h_i=0$ holds only if  $e_i=0$):
\begin{align}
&\tau\cd h_i=\Ad \exp(\mbox{$\frac{1}{\tau}$}e_i)(e_i+\tau\cd
 h_i)
\qquad\forall\ta\neq0,\ i=1,2;\label{ta}\\
&\quad\text{\qquad hence, we have}\quad
(\ta_1\cd h_1,\
\tau_2\cd h_2)=
\Ad \gamma(\tau_1,\tau_2)(\vka(\tau_1,\tau_2)),
\qquad (\tau_1,\tau_2)\in\CC.\nonumber
\end{align}

The last equation in \eqref{ta} clearly implies the lemma.
\end{proof}

 From Lemma \ref{conjug} we deduce that, for any
$(\tau_1,\tau_2)\in\C^2,$ in $\zz/\!/G=\tt/W$ one has
$\dis
\vka(\tau_1,\tau_2)\,{\opp{mod}}\,G=
(\tau_1\cd h_1,\ \tau_2\cd h_2)\,{\opp{mod}}\,W$.
Thus, one can introduce  a map $\wt\vka$ that
fits into the following  diagram of cartesian
squares
$$
\xymatrix{
\C^2 \times_{\tt/W}\tt\ar@{=}[r]&
\C^2\, \times_\zz\, (\zz\times_{\zz/\!/G}\tt)\ 
\ar[d]\ar[rr]^<>(0.5){{\wt\vka}}&&
\zz\times_{\zz/\!/G}\tt\ar[d]\ar@{->>}[rr]^<>(0.5){pr_2}&&\tt\ar[d]\\
&\C^2\ar[rr]^<>(0.5){\vka}&&
\zz\ar@{->>}[rr]^<>(0.5){\text{\em proj}}&&\zz/\!/G=\tt/W.
}
$$

The composite of the maps in \eqref{hh1} is a homomorphism of coordinate rings
that gives a certain $W\times\CC$-equivariant morphism
$\fy\ \to\ \C^2 \times_{\tt/W} \tt.$
Thus, we obtain the following chain of
$W\times\CC$-equivariant morphisms 
\beq{hh2}
\xymatrix{
\fy\  \ar[r]&
\  \C^2\, \times_{\tt/W}\,\tt\  \ar@{=}[r]&
\  \C^2\, \times_\zz\, (\zz\times_{\zz/\!/G}\tt)\  
\ar[rr]^<>(0.5){{\wt\vka}}&&\   \zz\times_{\zz/\!/G}\tt.
}
\eeq

The composite morphism in \eqref{hh2}
 factors through the reduced scheme $\xx$,
since the scheme $\fy$ is reduced
by Lemma \ref{rees}.
Thus, we have constructed a $W\times\CC$-equivariant morphism
$\ups$ that fits into a commutative diagram
\beq{ydiag}
\xymatrix{
\fy\ \ar@{->>}[d]^<>(0.5){\wp}\ar[rr]^<>(0.5){\ups}&&
\ \xx\ \ar@{->>}[d]^<>(0.5){p }\\
\C^2\ \ar[rr]^<>(0.5){\vka}&&\ \zz_\red\ 
}
\eeq

The proof of Theorem \ref{n!} is based on the following result
concerning  the structure of the map $\ups$  over $\C^2\sminus\{\bbo\},$
the complement of the origin $\bbo=(0,0)\in\C^2$.

\begin{prop}\label{cruc} \vi The image of  the map
$\vka$ is contained in $\zz^r$.
\vskip2pt

\noindent
\vii For any  $(x,y)\in \vka(\C^2\sminus\{\bbo\})$
we have:
\vskip2pt

$\,\quad$\npb{The fiber
$p \inv(x,y)$ is contained
in the smooth locus of the variety
$\xx$}
\vskip2pt

$\,\quad$\npb{The  map $p $ is flat over $(x,y)$.}
\vskip2pt

\noindent
\viii
The  map $\wp\times\ups$ 
yields the following isomorphism  of schemes
over $\C^2\sminus\{\bbo\}$:
$$\xymatrix{
\wp\inv(\C^2\sminus\{\bbo\})\
\ar[rr]^<>(0.5){\wp\times\ups}_<>(0.5){\sim}
&&\ (\C^2\sminus\{\bbo\})\times_\zz\xx.
}
$$
\end{prop}
\begin{proof} Part (i) is clear for $(\ta_1,\ta_2)=(0,0)$
since  we have $\vka(\bbo)=(e_1,e_2)=\bbe\in\zz^r$.
Thus, for the rest of the proof  we may assume that
$(x,y)=\vka(\ta_1,\ta_2)$ for some
$(\ta_1,\ta_2)\in\C^2\sminus\{\bbo\}$.

One has an open covering
$\C^2\sminus\bbo= U_1\cup U_2$ where
$U_i=\{(\ta_1,\ta_2)\in\C^2\mid \ta_i\neq0\}.$
We will consider the case where $(\ta_1,\ta_2)\in U_2$, the
other case  being totally similar.
Thus, we have
$(x,y)=(e_1+\ta_1\cdot h_1,\ e_2+\ta_2\cdot h_2)$
for some $\ta_2\neq0$.

If $\ta_1\neq0$ then the element
$(x,y)$ is $G$-conjugate to the element
$(\ta_1\cdot h_1,\ \ta_2\cdot h_2),$
by formula \eqref{ta}. In that case, one has
$(x,y)\in\zz^{rs}$ and all the statements
of parts (i)-(ii) of  Proposition \ref{cruc}
are clear.

It remains to consider the case $\ta_1=0$.
Thus, we have $(x,y)=(e_1, e_2+\ta_2\cdot h_2)$
where $\ta_2\neq0$. Applying the formula
in the first line of \eqref{ta}, we deduce
that  the element
$(x,y)$ is $G$-conjugate to the element
$(e_1, \ta_2\cdot h_2)$. Further, using the $\C^\times$-action,
we may assume without loss of generality that $\ta_2=1$.

To complete the proof of parts (i)-(ii), we write
$\gg=\C^2\o\g,$ resp.
$\gg\times\tt=\C^2\o(\g\times\t).$
The natural $GL_2$-action on $\C^2$
induces, via the
action on the first tensor factor, a  $GL_2$-action on 
$\gg$, resp. on $\gg\times\tt$, such that
the $\CC$-action considered earlier corresponds
to the action of the maximal torus in
$GL_2$ formed by diagonal matrices.
The scheme $\zz\sset\gg$,
resp. $\xx\sset\gg\times\tt$, is clearly $GL_2$-stable.
Further, it is  clear that there is an element 
$g\in GL_2$ that takes
the point $(e_1, h_2)\in\zz$
to the point $(e_1+h_2,h_2)$.
Now, according to \cite[Proposition 3.2.3]{pnp},
the element $e_1$ is a principal nilpotent in the
Levi subalgebra $\g_{h_2}$. It follows that
$e_1+h_2$ is a regular element of $\g$.
Hence, we have $(e_1+h_2,h_2)\in\zz_1$, cf. Definition \ref{U}.
It follows that $(e_1+h_2,h_2)$ is a smooth point
of $\zz$. Moreover,  the fiber $p \inv(e_1+h_2,h_2)$
is contained in the smooth locus of $\xx$
and the map $p $ is flat over
the point $(e_1+h_2,h_2)$, by Lemma \ref{conormal}.
All statements of Proposition \ref{cruc}(i)-(ii)
follow from this using the $GL_2$-action.

We see from part (ii)  that $U_2\times_\zz\xx$
 is a  reduced scheme, flat  over
$U_2$.
Proving part (iii) amounts to
 showing that the map
$\wp\inv(U_2)\ \to\
U_2\times_\zz\xx$
induced by  $\wp\times\ups$
is an isomorphism.
The map in question is a 
$\CC$-equivariant morphism of
flat schemes over $U_2=\C\times\C^\times$.
Therefore, this map is an isomorphism if and only if it induces
an isomorphism 
\beq{psii}
\Psi:\ \, \wp\inv(0,1)\ \iso\  \{(0,1)\}\times_\zz\xx=p\inv(e_1, e_2+h_2),
\eeq
of the corresponding fibers over a single
point $(0,1)\in \C\times\C^\times$.

Observe next that the  argument used in the proof
of part (ii), based on the $GL_2$-action,
yields
an isomorphism of schemes $p\inv(e_1, e_2+h_2)\ccong
p\inv(e_1+h_2, h_2)$. Further,
we know that $e_1+h_2$ is a regular element of $\g$.
Therefore, taking the fibers at $e_1+h_2$
of the locally free sheaves involved in the
chain of isomorphism \eqref{rrcart},
yields an
algebra isomorphism
$\C[p\inv(e_1+h_2,h_2)]=\rr_{(e_1+h_2,h_2)}\cong
\C[\vartheta\inv(W\cdot h_2)]$.
On the other hand, by Lemma \ref{rees}, we have
$\C[\wp\inv(0,1)]\cong \C[\vartheta\inv(W\cdot h_2)]$.
Thus,
we obtain the following chain of  $W$-equivariant algebra isomorphisms
$$
\C[\{(0,1)\}\times_\zz\xx]=
\C[p\inv(e_1, e_2+h_2)]\cong
\C[p\inv(e_1+h_2,h_2)]\cong
\C[\vartheta\inv(W\cdot h_2)]
\cong
\C[\wp\inv(0,1)].
$$

We claim that the composite of the above isomorphisms
is equal to the algebra map $\Psi^*:\
\C[\{(0,1)\}\times_\zz\xx]\to
\C[\wp\inv(0,1)]$ induced by the  morphism
$\Psi$ in \eqref{psii}. To see this, one
observes that 
the algebra $\C[\{(0,1)\}\times_\zz\xx]$  is a quotient
of the algebra $\C[\tt]$.
Hence, it suffices to verify our claim for linear functions on $\tt$,
the generators of the algebra $\ctt$.  Checking the latter is straightforward
and is left
 to the reader.

We conclude that the morphism
$\Psi$ in \eqref{psii} is itself an isomorphism,
and part (iii) of Proposition \ref{cruc} follows.
\end{proof}

\subsection{Proof of Theorem \ref{n!}}\label{npf}
 Part (i) of Proposition \ref{cruc} implies that the map $\vka$ gives
a closed imbedding $\vka:\ \C^2\into\zz^r$.
Therefore,  we may view $\vka$ as a map $\C^2\to\zz_\norm$
and form an associated fiber product
$\C^2\times_{\zz^r}\xx_\norm$.
From diagram \eqref{ydiag} we obtain, by base change,
the following commutative
diagram of $W\times\CC$-equivariant morphisms of schemes over $\C^2$
\beq{fxy}
\xymatrix{
\fy\ \ar@{->>}[d]^<>(0.5){\wp}\ar[rr]^<>(0.5){\wp\times\ups}&&
\ \C^2\times_{\zz^r}\xx\ \ar@{->>}[d]&&
\ \C^2\times_{\zz^r}\xx_\norm\ 
\ar[ll]_<>(0.5){\Id_{\C^2}\times\psi}\ar@{->>}[d]^<>(0.5){p_\vka}\\
\C^2\ \ar@{=}[rr]^<>(0.5){\Id}&&\ \C^2\ \ar@{=}[rr]^<>(0.5){\Id}&&
\ \C^2.
}
\eeq

Restricting this diagram 
 to the open set $\C^2\sminus\{\bbo\}$
one obtains a diagram of isomorphisms
\beq{openiso}
\xymatrix{
\wp\inv(\C^2\sminus\{\bbo\})\ar[rr]^<>(0.5){\wp\times\ups}_<>(0.5){\sim}&&
(\C^2\sminus\{\bbo\})\times_{\zz^r}\xx&&
(\C^2\sminus\{\bbo\})\times_{\zz^r}\xx_\norm.
\ar[ll]_<>(0.5){\Id_{\C^2}\times\psi}^<>(0.5){\sim}
}
\eeq
Here the map $\wp\times\ups$ on the left is an  isomorphism
thanks to part (iii) of  Proposition \ref{cruc}
and the map $\Id_{\C^2}\times\psi$ on the right
induced by the normalization map $\psi: \xx_\norm\to\xx$ is an  isomorphism
thanks to  part (ii) of Proposition \ref{cruc}.

The map  $p_\norm:\ \xx_\norm\to\zz_\norm$ being  flat,
by flat base change  we get that
$(p_\vka)_*\oo_{\C^2\times_{\zz^r}\xx_\norm}$
$\cong
\vka^*[(p_\norm)_*\oo_{\xx_\norm}]=
\vka^*\rr$ is a locally free sheaf on $\C^2$.
We conclude that $\wp_*\oo_\fy$ and
$\vka^*\rr$ are $W\times\CC$-equivariant  locally free sheaves of
 $\oo_{\C^2}$-algebras 
and we have a $W\times\CC$-equivariant  algebra isomorphism
$(\wp_*\oo_\fy)|_{\C^2\sminus\{\bbo\}}\
\cong\ (\vka^*\rr)|_{\C^2\sminus\{\bbo\}}$
induced by  diagram \eqref{openiso}. Therefore,
 the isomorphism
can be extended 
(uniquely) across the origin $\bbo\in\C^2$. The resulting
 isomorphism 
$\vka^*\rr\
\iso\ \wp_*\oo_\fy$
is automatically    $W\times\CC$-equivariant
and respects the
 $\oo_{\C^2}$-algebra structures.
Restricting the latter  isomorphism to the
fibers at the origin
 yields a bigraded
$W$-equivariant algebra isomorphism
$\rr_\bbe\
\iso\ \C[\wp\inv(\bbo)].$
On the other hand, by \eqref{rees1},
we have $\C[\wp\inv(\bbo)]=\gr\max{^{\,}}\C[\wbh]$.
The isomorphism of Theorem \ref{n!} follows.

The last claim in the theorem is a consequence of
Lemma \ref{rees}(ii).
\qed

\subsection{Proof of Theorem \ref{trace}}\label{pftr}
First of all, we introduce some notation
and define several vector spaces
associated with the pair 
$\fo_\idot\ctt$ and $\ft_\idot\ctt$
of filtrations on $\ctt$.

We may (and will) view elements of the vector space 
$\sym\tt$  as constant
coefficient
differential operators on $\tt=\t\times\t$.
Thus, an element $u\in \sym^i\t\o\sym^j\t$
is a bihomogeneous differential operator
of order $i$ with respect to the first, resp.
of order $j$ with respect to the second,
factor in the cartesian product $\t\times\t=\tt$.
For any $m\geq 0$, we write $S^m:=\sym^m\t$ and let
$\ssym:=\prod_{i,j\geq 0}\ (\sym^i\t\o\sym^j\t)$.

For any $i,j\geq 0$,
the
assignment $u\times f\mto u(f)(0)$
gives a perfect pairing
$(S^i\o S^j)\,\times\,\C^{i,j}[\tt]\
\to\C.$
We obtain  canonical
isomorphisms
\beq{ssym}\big(\ctt\big)^*\ = \ \big(\bigoplus\nolimits_{i,j}\ \C^{i,j}[\tt]\big)^*
\ccong \prod\nolimits_{i,j\geq 0}\ \big(\C^{i,j}[\tt]\big)^*
\ =\ \prod\nolimits_{i,j\geq 0}\ (S^i\o S^j)\ =\ \ssym.
\eeq
Thus, one has a  perfect pairing
$\ssym\times\ctt\to\C$.
It is instructive to think of an element $u\in \ssym$ as a constant
coefficient
differential operator on $\tt$ "of infinite order".
Accordingly, we
write the above pairing as $u\times f\mto u(f)(0)$
similarly to the finite order case. 
Given a vector subspace $H\sset\ctt$,
let $H^\perp\sset \ssym$ denote the
annihilator of $H$ with respect the pairing.

Below, we will use the following simplified
notation $\fo_i:=\fo_i\ctt,$ resp.
$\ft_j:=\ft_j\ctt$, and
$F_{i,j}=F_{i,j}\ctt=\fo_i\cap\ft_j.$ 
Further, let $I=I_\bh$ so, we have $\ctt/I=\C[\wbh]$. 
Then, for any $m,n\geq 0$, there are
natural imbeddings
$$
\xymatrix{
F_{m,n}=\fo_m\cap\ft_n\ 
\ar@{^{(}->}[r]&
\ [\fo_{m}+I]\cap[\ft_n+I]\ 
\ar@{^{(}->}[r]&
\ \fo_{m}+\ft_n+I\ 
\ar@{^{(}->}[r]&
\ \ctt.
}
$$

These imbeddings  induce  the following chain of natural
 linear
maps
\begin{multline}\label{abc}
\gr_{m,n}\ctt=
\frac{F_{m,n}}{F_{m,n-1}+ F_{m-1,n}}=
\frac{\fo_m\cap\ft_n}{
\fo_{m-1}\cap\ft_n\ +\ \fo_m\cap\ft_{n-1}}\\
\xymatrix{\ar[r]^<>(0.5){a}&\ \gr\max_{m,n}\C[\wbh]}
=\frac{[\fo_{m}+I]\cap[\ft_n+I]}{
[\fo_{m-1}+I]\cap[\ft_n+I]\ +\ [\fo_m+I]\cap[\ft_{n-1}+I]}\\
\xymatrix{\ar[r]^<>(0.5){b}&}
\
\frac{\ctt}{\fo_{m-1}+\ft_{n-1}+I}\ .
\end{multline}

For each $m,n\geq 0$, let
$$\so^{>m}:=
\prod\nolimits_{\{i>m,\ j\geq 0\}}\ S^i\o S^j,
\quad\opp{resp.}\quad
\st^{>n}:=
\prod\nolimits_{\{i\geq 0,\ j> n\}}\ S^i\o S^j.
$$

Clearly, we have $(\fo_i)^\perp=\so^{>i},$
resp.
$(\ft_j)^\perp=\st^{>j}.$
Hence, we get $F_{i,j}^\perp= (\fo_i\cap\ft_j)^\perp
=
\so^{>i}+\st^{>j}$ and $\ [\fo_{i}+I]^\perp=\so^{>i}\cap I^\perp$,
resp.  $\ [\ft_j+I]^\perp=\st^{>j}\cap I^\perp$. Further,
we find

\begin{multline}
\left(\frac{\ctt}{\fo_{i-1}+\ft_{j-1}+I}\right)^*=(\fo_{i-1}+\ft_{j-1}+I)^\perp=
\so^{>i-1}\cap\st^{>j-1}\cap I^\perp;\\
\left(\frac{F_{m,n}}{F_{m,n-1}+ F_{m-1,n}}\right)^*=
\frac{(\so^{>m-1}+\st^{>n})
\,\cap\,
(\so^{>m}+\st^{>n-1})}{\so^{>m}+\st^{>n}}
\ =\ 
S^m\o S^n.\label{SSS}
\end{multline}

Therefore, dualizing the maps in \eqref{abc}
 one gets the following linear maps
\begin{multline}
\so^{>m-1}\cap\st^{>n-1}\cap I^\perp
\ \xymatrix{\ar[r]^<>(0.5){b^*}&
\ 
(\gr\max_{m,n}\C[\wbh])^*
\ 
\ar[r]^<>(0.5){a^*}&\ (\gr_{m,n}\ctt)^*}\ =\ 
S^m\o S^n. \label{for2}
\end{multline}
%}
%\\
%=\ 
%\frac{(\so^{>m-1}+\st^{>n})
%\,\cap\,
%(\so^{>m}+\st^{>n-1})}{\so^{>m}+\st^{>n}}
%%\ =\ 
%S^m\o S^n
%\ =\
%(\gr_{m,n}\ctt)^*. \label{for2}
%\end{multline}

\begin{proof}[Proof of Theorem \ref{trace}]
We write $\check \alpha\in\t$
for the coroot corresponding to a root
$\alpha\in R^+$.
Recall the notation  $R^+_s\sset R^+$
for the set of positive roots of the
Levi subalgebra $\g^s,\ s=1,2,$ and put
 $\del_s:=\prod_{\al\in R^+_s}\ \check \alpha
\,\in\,S^{\bd_s}$.
Further, let $V_s=\C[W]\cdot\del_s\sset S^{\bd_s},\ s=1,2,$
be the  $W$-submodule generated by the 
element $\del_s$. It is known, that $V_s$ is a simple $W$-module
and, moreover, this  $W$-module occurs in
$S^{\bd_s}$ with multiplicity one,
see \cite{Mc}.
Therefore, there is a canonically defined
copy of the simple $W\times W$-submodule
$V_1\o V_2$ inside  $S^{\bd_1}\o S^{\bd_2}.$
Dually,  there is a canonically defined
copy of the simple $W\times W$-submodule
$\check V_1\o \check V_2$ inside $\C^{\bd_1,\bd_2}[\tt]$
where $\check V_s\sset \C^{\bd_s}[\t],\ s=1,2,$ stands 
for the contragredient $W$-module
(in fact, one has $V_s\cong\check V_s,$ since
 any simple $W$-module is known to be  selfdual).

Following \cite[\S 4]{pnp}, we consider
the element $\Del:=\sum_{w\in W}\
\sign(w)\cdot e^{w(\bh)}\in \ssym$. More explicitly, we let
\beq{Del}
\Del=\sum_{i,j\geq 0}\ \Del_{i,j}\en\opp{where}\en
 \Del_{i,j}:=\frac{1}{i!\cdot j!}\sum_{w\in W}\
\sign(w)\cdot w(h_1^i\o h_2^j)\ \in\ S^i\o S^j.
\eeq

Using the Taylor formula,  for any polynomial $f$ on $\tt$, we find
$e^{w(\bh)}(f)(0)=f(w(\bh)).$ Hence, we get
$\Del(f)(0)=\sum_{w\in W}\ \sign(w)\cdot f(w(\bh))$.
It follows that the linear function
$\C[\tt]\to\C,\ f\mto\Del(f)(0)$
 annihilates the ideal
$I=I_\bh$ in other words, we have $\Del\in I^\perp.$ 
Further, by \cite[Lemma 4.3]{pnp},
one has $\Del_{i,j}=0$ whenever
$i<\bd_1$ or $j<\bd_2$.
We conclude that
$\Del\in \so^{>\bd_1-1}\cap\st^{>\bd_2-1}\cap I^\perp$.
Thus, there is a well defined element
$a^*(b^*(\Del))\in S^{\bd_1}\o S^{\bd_2}$,
cf. \eqref{for2}.
In addition, it is clear from \eqref{SSS}-\eqref{Del} that
 we have an equation:
\beq{cong}
\Del_{\bd_1,\bd_2}
\ =\
a^*(b^*(\Del))
\ \opp{mod}
\,(\so^{>\bd_1}+\st^{>\bd_2}).
\eeq

From now on, we assume that
 $\bbe$ is a {\em non-exceptional} 
principal nilpotent pair. 
Then, according to \cite[Theorem 4.4]{pnp},
one has an isomorphism
$V_2\cong V_1\o\sign$ of $W$-modules.
It follows that $(V_1\o V_2)^\sign$ is a 1-dimensional
vector space, moreover, according to
{\em loc cit}, the element $\Del_{\bd_1,\bd_2}$ is a
nonzero element of that
vector space.
Dually,  $(\check V_1\o\check  V_2)^\sign$ is a
 1-dimensional vector space and  $\bde$  
is a
nonzero element of that
vector space.

The canonical  perfect pairing
$(V_1\o V_2)\times (\check V_1\o\check  V_2)\to \C$
 is  $W$-invariant
with respect to the $W$-diagonal
action. Therefore, this pairing yields a perfect pairing between 
$(V_1\o V_2)^\sign$ and $(\check V_1\o\check  V_2)^\sign$,
 the corresponding sign-isotypic components.
These are 1-dimensional vector spaces, with
$\bde$ and $\Del_{\bd_1,\bd_2}$ being respective base vectors.
Hence, one must have $\Del_{\bd_1,\bd_2}(\bde)(0)\neq 0$.
Thus, writing
$\dis\langle\!\langle-,-\rangle\!\rangle:\
(\gr\max_{\bd_1,\bd_2} \C[\wbh])^*\times (\gr\max_{\bd_1,\bd_2} \C[\wbh])\to \C$
for the canonical pairing and using  \eqref{cong}, we deduce
 $$\langle\!\langle b^*(\Del),\ a(\bde)\rangle\!\rangle=
\langle\!\langle a^*(b^*(\Del)),\ \bde\rangle\!\rangle=
\Del_{\bd_1,\bd_2}(\bde)(0)\neq0.$$
The map $a$ in \eqref{abc} is clearly
$W$-equivariant (with respect to the action induced by
the $W$-diagonal
action on $\C[\t]\o\C[\t]$).
Thus, we have shown that $a(\bde)$
is a nonzero element of the vector space
$\dis(\gr\max_{\bd_1,\bd_2}\C[\wbh])^\sign $.
In particular, we get
 $(\gr\max_{\bd_1,\bd_2}\C[\wbh])^\sign\neq 0$.

We know
that $\gr\max \C[\wbh]$ is a graded 
algebra such that $\gr\max_{i,j}\C[\wbh]=0$
whenever $i>\bd_1$ or $j>\bd_2$, by Theorem \ref{n!}.
Also, we have  $\dim(\gr\max \C[\wbh])^\sign=1$ and
Theorem  \ref{n!} says that the $W$-equivariant
projection
$\gr\max \C[\wbh]\onto (\gr\max \C[\wbh])^\sign$
gives a nondegenerate trace on
the algebra $\gr\max \C[\wbh]$.
Therefore, since
 $(\gr\max_{\bd_1,\bd_2}\C[\wbh])^\sign\neq 0,$
it follows that we must have
$\dis\gr\max_{\bd_1,\bd_2} \C[\wbh]=
(\gr\max_{\bd_1,\bd_2} \C[\wbh])^\sign=(\gr\max \C[\wbh])^\sign$.
In particular, the vector space
$\gr\max_{\bd_1,\bd_2} \C[\wbh]$ is 1-dimensional,
with $a(\bde)$ being a base vector.
The theorem follows.
\end{proof}

\subsection{The algebra $\rr_\bbe$}\label{pnpappl} We use  the notation of the previous
section and keep the assumption that  $\bbe$ is
 a
non-exceptional principal nilpotent pair.

\begin{proof}[Proof of Corollary \ref{normlem}]
Let $\oo_{\bbe,0}$ denote the local ring of the
algebra $\C[\xx]$ at the point $(\bbe,0)\in\gg\times\tt$.
The variety $\xx$ is normal at $(\bbe,0)$ if and only if
the canonical imbedding $\oo_{\bbe,0}\to 
\oo_{\bbe,0}\,\o_{\C[\xx]}\,\C[\xx_\norm]$ is an
isomorphism. The latter holds,
by the Nakayama lemma,  if and only if
the natural map $\C[p\inv(\bbe)]\to
\C[p_\norm\inv(\bbe)]=\rr_\bbe$ is surjective.
This yields the equivalence
$\mathsf{(i)}\ \Leftrightarrow\ \mathsf{(ii)}$
of the Corollary,
since the restriction map
$\ctt\to \C[p\inv(\bbe)]$ is surjective.

The equivalence
$\mathsf{(ii)}\ \Leftrightarrow\ 
\mathsf{(iii)}$ is a direct consequence of Theorem \ref{n!}.
\end{proof}

\begin{proof}[Proof of Corollary \ref{kostka}]
Let $I=I_\bh$ be the ideal of the reduced subscheme
$\wbh\sset\tt$ and choose ${d}>\!>0$ such that
$\ft_{{d}}\C[\wbh]=\C[\wbh]$
(by 
Lemma \ref{rees}(ii), one can take ${d}=\bd_2$).
Thus,  we have $\ft_{{d}}\ctt +I=\ctt$.
Also, by definition, one has
$\fo_0\ctt=1\o\C[\t]$. Therefore,
we find
$$
F_{0,{d}}\C[\wbh]=\frac{((1\o\C[\t])
+I))\cap(\ft_{{d}}\ctt +I)}{I}
=
\frac{(1\o\C[\t]) +I}{I}\ccong \frac{1\o\C[\t]}{(1\o\C[\t]) \cap I}.
$$

The ideal $(1\o\C[\t]) \cap I$ corresponds,
via the identification $1\o\C[\t]=\C[\t]$,
to the ideal in $\C[\t]$ of the image of the orbit
$\wbh$ under the second projection $\tt\to\t$.
This image equals the set $W\cdot h_2\sset\t$.
We conclude that the rightmost space
in the displayed formula above is isomorphic to
$\C[W\cdot h_2]$. Thus, using that formula, we obtain
$W$-module isomorphisms
$$
\grt(F_{0,{d}}\C[\wbh])
\ccong F_{0,{d}}\C[\wbh]\ccong \C[W\cdot h_2]\ccong
\C[W/W_1].
$$

 This yields the first isomorphism
of the corollary. The second
isomorphism is proved similarly. 
\end{proof}

As an immediate consequence of the (proof of) Theorem \ref{trace},
we have 

\begin{lem}\label{vanish} There is a nonzero constant
$c\in\C$ such that  $\bde(w\bh)=c\cdot\sign(w),\ \forall w\in W$.
\end{lem}
\begin{proof} Any  $W$-alternating function on
the orbit $\wbh$ is a constant multiple of   the function
$\sign_{W\cdot\bh}:\ w(\bh)\mto\sign(w)$.
The  function $\bde|_{W\cdot\bh}$ is 
$W$-alternating, hence it  must be proportional to the function
$\sign_{W\cdot\bh}$.
We know that the class of  $\bde$  in $\gr\max_{\bd_1,\bd_2}\C[\wbh]$
is nonzero, by Theorem \ref{trace}.
Hence, the function $\bde|_{W\cdot\bh}$ is not identically zero
and the result follows.
\end{proof}

Let $f$ be  a holomorphic function  on $\tt$.
The pull-back of $f$ via
the composite $\xx_\norm\to\xx\to\tt$
is  a  $G$-invariant holomorphic function on $\xx_\norm$.
One may view that function 
as a  holomorphic section, $f^\sharp$, of the
sheaf $\rr$, resp. view the restriction of that function 
 to the
subscheme $p_\norm\inv(\bbe) \into
\xx_\norm$ as an element
$\rese(f)\in  \rr_\bbe=\C[p_\norm\inv(\bbe)]$,
the value of the section $f^\sharp$
at the point $\bbe$.

We know, thanks to Theorem \ref{trace}
and Lemma \ref{vanish}, that the  class  in $\gr\max\C[\wbh]$ of 
the function $\sign_{W\cdot\bh}$ 
gives a base vector of the 1-dimensional vector
space $(\gr\max\C[\wbh])^\sign=\gr\max_{\bd_1,\bd_2}\C[\wbh]$.
Let $u\in \rr_\bbe^\sign$ be
the image of that base vector
under the isomorphism $(\gr\max\C[\wbh])^\sign\iso \rr_\bbe^\sign,$
of Theorem \ref{n!}.
Thus, $u$ is a base vector  of the 
 1-dimensional vector space  $\rr_\bbe^\sign$.

It is clear that, for any   $W$-alternating function $f$, on $\tt$, 
there is a constant  $c(f)\in\C$ such that one has
$\rese(f)=c(f)\cdot u$.

\begin{proof}[Proof of Corollary \ref{section}]
Let $\prsign_\rr:\ \rr_\bbe\onto\rr^\sign_\bbe$,
resp. $\prsign_{W\cdot\bh}:\ \C[\wbh]\onto\C[\wbh]^\sign$,
be the projection to the  sign
isotypic component. Further, we will use a  natural identification
$\zeta: \ \ctt=\gr\ctt$, as bigraded algebras,
resulting from the bigrading 
on $\ctt=\bplus_{i,j}\ \C^{i,j}[\tt]$
on the algebra $\ctt$ itself.
We have a diagram

\beq{ares}
\xymatrix{
\C[\tt]\ \ar@{=}[d]_<>(0.5){\zeta}\ar[rr]^<>(0.5){\rese}&&
\ \rr_\bbe\ \ar@{=}[d]_<>(0.5){\text{Theorem \ref{n!}}}
\ar[rr]^<>(0.5){\pr^\textit{sign}_\rr}&&
\ \rr^\sign_\bbe \ar@{=}[d]_<>(0.5){\text{Theorem \ref{trace}}}\\
\gr\ctt\ \ar[rr]^<>(0.5){a}&&\
\gr\max\C[\wbh]\ \ar[rr]^<>(0.5){\pr^\textit{sign}_{W\cdot\bh}}&&
\ \gr\max_{\bd_1,\bd_2}\C[\wbh].
}
\eeq
Here, the map $a$ is induced by the projection $\ctt\onto\C[\wbh]$, cf. \eqref{abc}.

Going through the construction of the
isomorphism of Theorem \ref{n!}, see esp. formulas \eqref{hh1} and
\eqref{hh2}, shows that the left square in diagram \eqref{ares} commutes.
The right square commutes trivially.

For any  $W$-alternating function $f$ on $\tt$
we have  $f|_{W\cdot\bh}=f(\bh)\cdot\sign_{W\cdot\bh}$.
Assume, in addition, that $f\in  (\C^{\bd_1,\bd_2}[\tt])^\sign,$
a homogeneous polynomial of bidegree
$(\bd_1,\bd_2)$. Then, 
the equation $f|_{W\cdot\bh}=f(\bh)\cdot\sign_{W\cdot\bh}$,
in $\C[\wbh]$, implies an
equation
$a(\zeta(f))=f(\bh)\cdot\sign_{W\cdot\bh},$  in $\gr\max_{\bd_1,\bd_2}\C[\wbh]$.

We transport the latter equation via the isomorphism
of Theorem \ref{n!}. Thus,  using commutativity  of
 diagram \eqref{ares} and our definition
of the base vector $u\in\rr^\sign$,
we deduce
\beq{deduce}
\rese(f)=f(\bh)\cdot u, \quad\forall f\in (\C^{\bd_1,\bd_2}[\tt])^\sign.
\eeq

Next, by definition one has $\bbs_\bbe=\Delta_\bbe^\sharp$.
Hence, 
$\bbs_\bbe(\bbe)=\rese(\bde)$. Thus,
 applying formula \eqref{deduce}
to the polynomial $\bde$, we get
$\bbs_\bbe(\bbe)= \bde(\bh)\cdot u.$

Now, Lemma \ref{vanish} implies that  $\bde(\bh)\neq0$.
It follows that $\bbs_\bbe(\bbe)\neq0$.
\end{proof}

\begin{rem} We have shown in the course of the proof of Theorem
\ref{trace}
that,
we have
$\Del_{\bd_1,\bd_2}(\bde)(0)\neq0$. This provides an alternative 
proof that  $\bde(\bh)\neq0$,
thanks to
a simple formula
$\Del_{\bd_1,\bd_2}(\bde)(0)=(\bd_1!\bd_2!)^2\cdot\bde(\bh).$
\end{rem}

\begin{proof}[Proof of Propositions \ref{exp}]
For each $\bx\in\tt$, we let $\bbde_\bx(-):=\bbde(\bx,-)$,
a $W$-alternating
holomorphic function on $\tt$.
Unraveling definitions, one finds that
$\bbss|_{\zz\times\{\bbe\}}= s^\sharp\boxtimes u$,
where $s^\sharp$ is a  $G$-invariant holomorphic section of the
sheaf $\rr^\sign$ that corresponds to 
the holomorphic function $s: \tt\to\C,\ \bx\mto c(\bbde_\bx)$.

Fix $\bx=(x_1,x_2)\in\tt$. For each $i,j\geq 1$, we define
a  homogeneous polynomial on $\tt$  of bidegree $(i,j)$ by 
the formula
$$\bbde^{ij}_\bx(y_1,y_2)\
=\ \mbox{$\frac{1}{i!j!}$}\sum_{w\in W}\ \sign(w)\cdot
 (\langle x_1, w(y_1)\rangle)^i
(\langle x_2, w(y_2)\rangle)^j.
$$ 

Writing $\by=(y_1,y_2)\in\tt$, one obtains an obvious expansion
$$\bbde_\bx(\by)=
\sum_{w\in W}\ \sign(w)\cd
e^{\lll\bx,w(\by)\rrr}
=\sum_{w\in W}\ \sign(w)\cd e^{\langle x_1,w(y_1)\rangle}\cd
e^{\langle x_2,w(y_2)\rangle}=
\sum_{i,j\geq0}\ \bbde^{ij}_\bx(y_1,y_2).
$$

Observe next that the  composite map $\pr^\sign_{W\cdot\bh}\ccirc a$, in the bottom row of
 diagram \eqref{ares},
annihilates the homogeneous component $\C^{i,j}[\tt]$ unless we have 
$(i,j)=(\bd_1,\bd_2)$. 
From the commutativity of the diagram  we deduce that
the  map $\pr^\sign_\rr\ccirc\rese$, in the top
row of the diagram, annihilates all
the  components $\C^{i,j}[\tt],\ (i,j)\neq(\bd_1,\bd_2)$. 
It follows  that we have
$\pr^\sign_\rr(\rese(\bbde_\bx))\ =\ \pr^\sign_\rr(\rese(\bbde^{\bd_1,\bd_2}_\bx))$.
Since each of 
the functions $\bbde^{\bd_1,\bd_2}_\bx$ and $\bbde_\bx$ is
 $W$-alternating,
 we  deduce $\rese(\bbde_\bx) =\rese(\bbde^{\bd_1,\bd_2}_\bx)$.
Hence,
applying
 formula \eqref{deduce}, we obtain
$$\rese(\bbde_\bx)=\rese(\bbde^{\bd_1,\bd_2}_\bx)=
\bbde^{\bd_1,\bd_2}_\bx(\bh)\cdot u,$$
where $u$ is the base vector of $\rr_\bbe^\sign$ 
 defined earlier in this section.

Now, using the definition of the polynomial $\bde$, we compute
$$\bbde^{\bd_1,\bd_2}_\bx(\bh)=
\mbox{$\frac{1}{\bd_1!\bd_2!}$} \sum_{w\in W}\ \sign(w)\cd (\langle x_1,w(h_1)\rangle)^{\bd_1}
(\langle x_2,w(h_2)\rangle)^{\bd_2}\
=\ \mbox{$\frac{1}{\bd_1!\bd_2!}$}\cd\bde(\bx).
$$

Recall finally  the function $s:\ \bx\mto c(\bbde_\bx)$,
where the quantity $c(\bbde_\bx)$
is determined from the equation $\rese(\bbde_\bx)=c(\bbde_\bx)\cdot u.$
Thus, using the above formulas, we find 
$$s(\bx)=c(\bbde_\bx)=\bbde^{\bd_1,\bd_2}_\bx(\bh)=
\mbox{$\frac{1}{\bd_1!\bd_2!}$}\cdot\bde(\bx).$$
Since $\bbs_\bbe=\Delta_\bbe^\sharp$,
for the corresponding sections of 
the sheaf $\rr$ we deduce an equation
$s^\sharp=\frac{1}{\bd_1!\bd_2!}\cdot\bbs_\bbe$.
Also, according to the proof of Corollary \ref{section},
we have $\bbs_\bbe(\bbe)= \bde(\bh)\cdot u.$

Combining everything together, we obtain
$$\bbss|_{\zz\times\{\bbe\}}=
s^\sharp \boxtimes u = \frac{1}{\bd_1!\cd\bd_2!}\cdot\bbs_\bbe\boxtimes u
=\frac{1}{\bd_1!\cd\bd_2!\cd\bde(\bh)}\cdot\bbs_\bbe\boxtimes
\bbs_\bbe(\bbe),$$
proving the first equation of Proposition \ref{exp}.
The proof of the second  equation is similar.
\end{proof}

\section{Relation to work of M. Haiman}\label{hai_sec}

\subsection{A vector bundle on the Hilbert scheme}\label{psec}
We keep the notation of \S\ref{mhai}, so
$G=GL_n$ and $V=\C^n$. We have $\g=\End_\C V$.
For any $(x,y)\in \zz$, one may view
the vector space $\g_{x,y}$  as an associative
subalgebra of $\End_\C V$.
Let $G_{x,y}\sset G$ be
the isotropy  group of the pair $(x,y)$ under the
$G$-diagonal action.
The group $G_{x,y}$ may be identified
with the group of invertible elements of
the associative algebra $\g_{x,y}$ hence, this group is connected.

Let $\C\langle x,y\rangle $ denote an associative
subalgebra of $\End_\C V$ generated by the elements $x$ and $y$.
Recall that a vector $v\in V$ is said to be 
a {\em cyclic vector} for a pair $(x,y)\in\zz$ 
if one has $\C\langle x,y\rangle v=V$.
Let $\zz^\circ$ be the set of pairs $(x,y)\in\zz$
which have a cyclic vector.

Part (i) of the following result is due to
Neubauer and Saltman  and part (ii)
is well-known,  cf.  \cite[Theorem 1.1]{NS}.

\begin{lem}\label{cyc} \vi A pair $(x,y)\in\zz$ is regular
if and only if one has
 $\C\langle x,y\rangle=\g_{x,y}$.

\vii The set  $\zz^\circ$
is a Zariski open subset of $\zz^r$.
For  $(x,y)\in\zz^\circ$,
all cyclic vectors for $(x,y)$ form
a single  $G_{x,y}$-orbit, which is the unique
Zariski open dense $G_{x,y}$-orbit in $V$. \qed
\end{lem}

Let  $G$ act on $\zz\times V$ by
$\,g: (x,y,v)\mto (gxg\inv,gyg\inv,gv)$. We
 introduce a variety of triples
\beq{ssc}\ssc=\{(x,y,v)\in \g\times\g\times V\en\big|\en
[x,y]=0,\  \C\langle x,y\rangle v=V\}.
\eeq

This is a Zariski open $G$-stable subset of $\zz\times V$.
Lemma \ref{cyc} shows that
$\ssc$ is  smooth and that
the $G$-action on $\ssc$ is free.
It is known that
there exists a universal  geometric quotient morphism
$\rho:\ \ssc\to\ssc/G$. Furthermore, the variety $\ssc/G$ 
may be identified with $\hilb:=\hilb^n(\C^2)$, 
the Hilbert scheme of $n$ points in $\C^2$,
 see
\cite{Na}.

The first projection $\zz\times V\to\zz$ restricts
to a $G$-equivariant map $\delta: \ssc\to \zz^\circ$.
We let 
${\mathcal G}_{\zz^\circ}$,
the universal stabilizer group scheme on $\zz^\circ,$
cf. \S\ref{bei}, act along the fibers
of $\delta$ by
$(\gamma,x,y):\ (x,y,v)\mto (x,y,\gamma v)$.
Lemma \ref{cyc}(i) implies that ${\mathcal G}_{\zz^\circ}$
is a smooth commutative
group scheme, furthermore,
the ${\mathcal G}_{\zz^\circ}$-action
 makes $\ssc\to \zz^\circ$ a $G$-equivariant
$\ZZ_{\zz^\circ}$-torsor over ${\zz^\circ}$. In particular,
$\delta$ is a smooth morphism.

We use the notation 
of \S\ref{dsec} and
put $\pp=(\rho_*\delta^*(\rr|_{\zz^\circ}))^G$.
There is a $W$-action on $\pp$ inherited from the one
on $\rr$. 
By equivariant descent, one has a canonical
isomorphism $\rho^*\pp=\delta^*(\rr|_{\zz^\circ})$,
cf. diagram \eqref{square} below.
The sheaf $\rr|_{\zz^\circ}$
is locally free (Theorem \ref{rrsmooth}) and self-dual  (Theorem
\ref{main_thm}). Thus,
 we deduce the following result

\begin{cor}\label{P} The sheaf $\pp$ 
is a locally free   coherent
sheaf of commutative $\oo_{\hilb}$-algebras 
equipped with a natural $W$-action by algebra automorphisms.
The fibers of 
the corresponding algebraic vector bundle 
on $\hilb$ afford the regular representation of 
the group ~$W$ and the projection  $\pp\onto\pp^\sign$ induces a
nondegenerate trace
on each fiber.\qed
\end{cor}

\subsection{The isospectral Hilbert scheme}
\label{hilb}
Let $\ha:=\Spec_{\hilb}\pp$ be
 the relative spectrum of $\pp$,
a sheaf  of algebras on the Hilbert scheme.
By Corollary \ref{P}, the scheme $\ha$  comes equipped
with a  flat and finite morphism  $\eta: \ha\to\hilb$
and with a $W$-action along the fibers
of $\eta$. 
We conclude that
 $\ha$ is a reduced Cohen-Macaulay and  Gorenstein variety.

One can interpret the construction of
the scheme $\ha$ in more geometric terms as follows.
Let $\xc:=p\inv_\norm({\zz^\circ}),$
an open subvariety in $\xx_\norm$.
We  form a fiber product $\xx_\norm\times_{\zz^\circ}\ss$ and
consider
the following commutative diagram
where $\wt p$ denotes the second projection, cf. also \cite[\S8]{Ha2},

\beq{square}
\xymatrix{
\xc\ \ar[dr]_<>(0.5){p_\norm}&\ \xc\times_{\zz^\circ}\ssc\
\ar@{->>}[l]_<>(0.5){\wt\delta}
\ar@{->>}[dr]^<>(0.5){\wt p}\ar@{.>}[rr]^<>(0.5){h}_<>(0.5){\cong}
&&\ \ssc\times_{\hilb}\ha\
\ar[dl]_<>(0.5){\wt\eta}\ar@{->>}[r]^<>(0.5){\wt\rho}
&\ \ha
\ar@{->>}[dl]^<>(0.5){\eta}\\
&\ {\zz^\circ}\ \ar@{->>}[dr]_<>(0.5){\eqref{Jo}}
 &\ \ssc\ \ar@{->>}[l]^<>(0.5){\delta}\ar@{->>}[r]_<>(0.5){\rho}&
\ \hilb\ \ar@{->>}[dl]^<>(0.5){^{\;\ \text{Hilbert-Chow}}}&\\
&& \ \tt/W\ &&
}
%\begin{floatingfigure}
\eeq

We let $G$ act on $\ssc\times_{\hilb}\ha$ through the first factor, 
resp. act {\em diagonally} on
 $\xc\times_{\zz^\circ}\ssc$.
The morphisms $\delta$ and $\rho$ in diagram \eqref{square} are
smooth, resp.  the morphisms $p_\norm$ and $\eta$ are finite and flat.
The morphism $\wt\delta$, resp.
$\wt\rho, \wt p,$ and $\wt\eta$, 
is obtained from $\delta$, resp. from $\rho, p_\norm,$ and $\eta$,
by base change.
Hence, flat base change yields
\beq{basechange}\delta^*(\rr|_{\zz^\circ})=\delta^*(p_\norm)_*\oo_{\xc}=
\wt p_*\oo_{_{\xc\times_{\zz^\circ}\ssc}},\quad
\text{resp.}\quad
\rho^*\pp=\rho^*\eta_*\oo_\ha=
\wt\eta_*\oo_{_{\ssc\times_{\hilb}\ha}}.
\eeq

Since $\delta^*(\rr|_{\zz^\circ})=\rho^*\pp$
we deduce  a canonical isomorphism
$\wt p_*\oo_{_{\xc\times_{\zz^\circ}\ssc}}
\cong\wt\eta_*\oo_{_{\ssc\times_{\hilb}\ha}}$
of $G$-equivariant  sheaves of $\oo_\ssc$-algebras.
This means that there is a canonical $G$-equivariant  isomorphism of schemes
$h:\ \xc\times_{\zz^\circ}\ssc$ $\iso
\ssc\times_{\hilb}\ha$, the dotted arrow in diagram
\eqref{square}. 

The map
$\wt\rho$ makes $\ssc\times_{\hilb}\ha$ a
$G$-torsor over $\ha$. Therefore, the composite
$\wt\rho\ccirc h$ makes $\xc\times_{\zz^\circ}\ssc$
a $G$-torsor over $\ha$ hence, we have $\ha=
(\xc\times_{\zz^\circ}\ssc)/G$, a
 geometric quotient by $G$.
From \eqref{basechange}, we obtain
$\dis\rho_*\delta^*(\rr|_{\zz^\circ})=
\rho_*{\wt p}_*\oo_{{\xc\times_{\zz^\circ}\ssc}}=
\rho_*{\wt \eta}_*\oo_{{\ssc\times_{\hilb}\ha}}
=\eta_*{\wt \rho}_*\oo_{{\ssc\times_{\hilb}\ha}}.$
Therefore, taking $G$-invariants, we get
$\pp=[\rho_*\delta^*(\rr|_{\zz^\circ})]^G=
[\eta_*{\wt \rho}_*\oo_{{\ssc\times_{\hilb}\ha}}]^G=
\eta_*\oo_{\ha}.$

The  $W$-action on $\xc$ induces one on
$\xc\times_{\zz^\circ}\ssc$.
This $W$-action commutes with the $G$-diagonal
action, hence, descends
to $(\xc\times_{\zz^\circ}\ssc)/G$. 
The resulting  $W$-action may be identified with
the one that comes from the $W$-action on
the sheaf $\pp$, see Corollary \ref{P}.
The composite map
$\xc\times_{{\zz^\circ}}\ssc\to\xc\to [{\zz^\circ}
\times_{_{{\zz^\circ}/\!/G}}\tt]_\red
\to\tt$
descends to a $W$-equivariant map $\si:\ \ha\to\tt$.
Thus, we have a  diagram
\beq{ha}
\xymatrix{\xc\ &
\ \xc\times_{{\zz^\circ}}\ssc\ \ar@{->>}[l]_<>(0.5){\wt\delta}
\ar@{->>}[rr]^<>(0.5){\wt\rho\ccirc h}&&\ 
\  (\xc\times_{{\zz^\circ}}\ssc)/G=\ha  \ar[rr]^<>(0.5){\eta\times\si}&&
\hilb\times_{\tt/W}\tt.
}
\eeq

Following Haiman, one defines the
{\em isospectral Hilbert scheme} to be
$[\hilb\times_{\tt/W}\tt]_\red,$
a reduced fiber product.

\begin{prop}\label{haim} 
The map $\eta\times\si$ on the right of \eqref{ha}
factors through an isomorphism 
$$\ha\iso [\hilb\times_{\tt/W}\tt]_\norm.$$
In particular, the normalization of the 
isospectral Hilbert scheme is Cohen-Macaulay
and Gorenstein.
\end{prop}
\begin{proof} It is clear 
that $\wt\rho\ccirc h(\wt\delta\inv(\xr))$  is a Zariski
open subset of $\ha$.
Lemma \ref{xxbasic}(ii) implies
readily that the map $\eta\times\si$ restricts to
an isomorphism $\wt\rho\ccirc h(\wt\delta\inv(\xr))$
$\iso\hilb\times_{\tt/W}\ttr$;
in particular, $\eta\times\si$ is a birational isomorphism.
The image of the map $\eta\times\si$
 is contained
in $[\hilb\times_{\tt/W}\tt]_\red$ since the scheme $\ha$ is reduced.

Further, 
Corollary \ref{P} and  Lemma \ref{irr} imply
that the scheme $\ha$ is  Cohen-Macaulay 
and smooth in codimension 1.
We conclude that $\ha$  is a normal scheme which is birational
and finite over $[\hilb\times_{\tt/W}\tt]_\red$.
This yields the isomorphism of the proposition.
\end{proof}

As a consequence of Haiman's work, since $\pp=\eta_*\oo_{\ha}$
we deduce
\begin{cor}\label{procesi} The vector bundle
$\pp$ on $\hilb$ is isomorphic to the
{\em Procesi bundle}, cf. \cite{Ha1}.
\qed
\end{cor}

\begin{rem}\label{mmap}
For any $(x,y)\in \zz$, one has a surjective evaluation
homomorphism $\C[z_1,z_2]\onto\C\langle x,y\rangle,
\ P\mto P(x,y)$. The kernel of
this homomorphism is  an ideal $I_{x,y}\sset\C[z_1,z_2]$.
If $(x,y)\in\zz^r$ then
we have $\C[z_1,z_2]/I_{x,y}=\C\langle x,y\rangle=\g_{x,y},$
 by Lemma \ref{cyc}(i). Hence, $I_{x,y}$ has codimension $n$ in
 $\C[z_1,z_2]$. Therefore, the assignment $(x,y)\mto I_{x,y}$
gives a well defined morphism $\varrho:\ \zz^r\onto\hilb$
such that one has $(\varrho|_{_{\zz^\circ}})\ccirc \delta=\rho$.
It follows that $\varrho^*\pp$, the pull back of the
sheaf $\pp$, is a $G\times W$-equivariant locally free
sheaf of $\oo_{\zz^r}$-algebras. Furthermore, we have canonical
isomorphisms $\delta^*(\varrho^*\pp)=
\rho^*\pp=\delta^*\rr$ of
$G\times W\times{\mathcal G}_{\zz^\circ}$-equivariant
sheaves of $\oo_{\ss}$-algebras.
By equivariant descent, this yields
a canonical isomorphism 
$(\varrho^*\pp)|_{\zz^\circ}\cong \rr|_{\zz^\circ}$.
 Since $\zz^r\sminus\zz^\circ$ is a codimension
$\geq2$ subset in the smooth variety $\zz^r$
and the sheaves $\varrho^*\pp$ and $\rr|_{\zz^r}$ are both
locally free, we deduce
that the above isomorphism extends canonically
to an isomorphism $\varrho^*\pp\cong\rr|_{\zz^r}$
of $G\times W$-equivariant  locally
free sheaves  of $\oo_{\zz^r}$-algebras.
In other words, there is a $G\times W$-equivariant isomorphism
$\zz^r\times_{\hilb}\ha\cong p\inv_\norm(\zz^r)$
of schemes over ~$\zz^r$.
\end{rem}
\subsection{Proof of Theorem \ref{rres_thm}}\label{VT}
%\subsection{}\label{vt0} 
Given a $G$-variety $X$
and  a rational representation $E$ of $G$
we put $\mmp_G(X,E)=\big(\C[X]\o E\big)^G$,
the vector space of $G$-equivariant polynomial maps $X\to E$.
If $X\to Y$ is a $G$-torsor, we
let $\L_Y(E)$ be the sheaf of sections of
an associated vector bundle $X\times_G E \to Y$
on $Y$.
By definition,  one has $\Ga(Y, \L_Y(E))=\mmp_G(X,E).$

We may apply the above to the geometric quotient
morphism $\rho: \ss\to\hilb$. The tautological
bundle on the Hilbert scheme is defined to be 
$\VV:=\L_{\hilb}(V)$, a  rank
$n$ vector bundle  associated
with the vector representation of $G=GL(V)$ in $V$.

\begin{rem}\label{vt0}
According to Lemma \ref{cyc},
for any $(x,y,v)\in \ss$, one has a canonical
vector space isomorphism $\g_{x,y}\to V,\
a\mto a(v).$ This gives a canonical isomorphism
$\delta^*(\fz_{_{\zz^\circ}})\iso \rho^*\VV$, of
$G$-equivariant sheaves on $\ss$.
\erem

%Note also
%that Haiman's character formula is based
%on the Atiyah-Bott fixed point theorem
%together with the cohomology vanishing
%$H^j(\hilb,\ \pp\o\VV^{\o m})=0$ for all $j>0$.
%I don't see how one could adapt my proof of the above Theorem
%to prove such a   cohomology vanishing
%independently of Haiman.

For any $m\geq0$, one has 
 a vector bundle $\L_\ha(V^{\o m})$
on $\ha$
  associated with the $G$-repre-sentation  $V^{\o m}$ and
with the $G$-torsor $\wt\rho\ccirc h:\
\xx^\circ\times_{\zz^\circ}\ss\to\ha$.
Clearly, we have $\L_\ha(V^{\o m})=
\eta^*(\L_{\hilb}(V^{\o m}))=\eta^*\VV^{\o m}$
where  $\eta: \ha\to\hilb$ is the map from \eqref{square}.
Hence,  the projection formula yields
$$
\mmp_G(\xx^\circ\times_{\zz^\circ}\ssc,\ V^{\o m})=
\Ga(\ha,\ \L_\ha(V^{\o m}))=
\Ga(\hilb,\ \eta_*\eta^*\VV^{\o m})
=\Ga(\hilb,\ \pp\o\VV^{\o m}),
$$
where we have used that 
$\pp=\eta_*\oo_\ha$. Let $\Phi_1$ be the composite of the above isomorphisms.

Write $u\mto (x_u,y_u),\ x_u,y_u\in\g$ for the projection
$\xx_\norm\to\zz_\red$. Then,
by definition, we have ${\xx^\circ\times_{\zz^\circ}\ssc}$
${=\{(u,v),\ u\in \xx_\norm,\ v\in V\mid 
\C[x_u, y_u]v=V\}.}$
Thus, there is
a natural open imbedding
 $\alpha: \xx^\circ\times_{\zz^\circ}\ssc\into \xx_\norm\times V$
and one has the corresponding restriction morphism
\beq{rres}
\alpha^*:\
\mmp_G(\xx_\norm\times V,\  V^{\o m})\too
\mmp_G(\xx^\circ\times_{\zz^\circ}\ssc,\  V^{\o m}).
\eeq

Let $\C^\times\sset GL_n$
be the group of scalar matrices.
The scalar matrix
$z\cdot\Id\in GL_n,\ z\in\C^\times$,
acts trivially on $\zz_\red$ hence, for any $\ell\geq0$, the
element $z\cdot\Id$
 acts on the subspace
$\C[\zz_\red]\o\C^\ell[V]\sset
\C[\zz_\red\times V]$ via
multiplication by $z^{-\ell}$.
It follows that  the natural inclusion $\C^m[V]\into \C[V]$ induces
an isomorphism 
$\C^m[V]\o V^{\o m}\iso\big(\C[V]\o V^{\o m}\big)^{\C^\times}\!\!.$
The group $G=GL_n$ being generated by the subgroups $\C^\times$ and $SL_n$,
we obtain a chain of  isomorphisms
$$
\big(\C[\xx_\norm]\o \C^m[V]\o V^{\o m}\big)^{SL_n}\ \iso\ 
\mmp_G(\xx_\norm,\ \C[V]\o V^{\o m})\ \iso\ 
\mmp_G(\xx_\norm\times V,\  V^{\o m}).
$$
Write $\Phi_2$ for the composite isomorphism.

The isomorphism of Theorem \ref{rres_thm}(i) is defined as 
the following composition
\begin{multline*}
\big(\C[\xx_\norm]\o \C^m[V]\o V^{\o m}\big)^{SL_n}\
{\stackrel{\Phi_2}{\underset{^\sim}\too}}\ 
\mmp_G(\xx_\norm\times V,\  V^{\o m})\\
\stackrel{\alpha^*}\too
\mmp_G(\xx^\circ\times_{\zz^\circ}\ssc,\  V^{\o m})\
{\underset{^\sim}{\stackrel{\Phi_1}\too}}\
\Ga(\hilb,\ \pp\o\VV^{\o m}).
\end{multline*}

All the above maps are clearly $W\times S_m$-equivariant
bigraded 
$\C[\tt]$-module
morphisms. 
We see that proving the theorem reduces to 
the following result

\begin{lem}\label{lem1} The restriction map $\alpha^*$ in 
\eqref{rres} is an isomorphism.
\end{lem}

\begin{proof} The  map $\alpha^*$ is injective
since the set $\xx^\circ\times_{\zz^\circ}\ssc$
is Zariski dense in $\xx_\norm\times V$.

To prove surjectivity of $\alpha^*$ 
recall that a linear operator $x: V\to V$ has a cyclic
vector if and only if $x\in\g^r$.
Also,  for any  $(x_1,x_2)\in \zz$ and  $v\in V$,
either of the two
equations, ${\C[x_1]v=V}$ or
$\C[x_2]v=V$, implies ${\C[x_1,x_2]v=V}$.
It follows that we have $\zz^{rr}=\zz_1\cup\zz_2\sset\zz^\circ$
where  $\zz_i,\ i=1,2,$ are the open sets
 introduced in  Definition \ref{U}. 

Next, let $\ss_i,\ i=1,2,$
be the open subset of $\ss$ formed by
the triples  $(x_1,x_2, v)\in\zz\times V$
 such that we have $\C[x_i]v=V$.
Restricting the map $\delta: \ss\to\zz^\circ,\
(x_1,x_2,v)\mto (x_1,x_2)$
gives well defined maps
$\ss_i\to\zz_i$.
This way, we get
 a commutative diagram of open imbeddings
\beq{xxv}
\xymatrix{
\xx_i\times_{\zz_i}\ss_i\
\ar@{^{(}->}[r]
\ar@{_{(}->}[d]^<>(0.5){\alpha_i}&
\ \xx^{rr}\times_{\zz^{rr}}\ss\ 
\ar@{^{(}->}[rr]^<>(0.5){\beta_\ss}
\ar@{_{(}->}[d]^<>(0.5){\alpha^{rr}}&&
\ \xx^\circ\times_{\zz^\circ}\ss\ 
\ar@{_{(}->}[d]^<>(0.5){\alpha}\\
(\xx_i\times V)
\ \ar@{^{(}->}[r]&
\ \xx^{rr}\times V\ 
\ar@{^{(}->}[rr]^<>(0.5){\beta_V}&&
\ \xx_\norm\times V
}
\eeq

Recall that the  morphism  $\delta: \ss\to\zz^\circ$ is smooth
and the set $\xx_\norm\sminus\xx^{rr}$
has codimension $\geq 2$ in $\xx_\norm$,
by Lemma \ref{irr}. 
Hence,
the complement  of the set $\xx^{rr}\times_{\zz^{rr}}\ss$ has codimension
 $\geq 2$ in $\xx^\circ\times_{\zz^\circ}\ss$,
resp. the complement  of the set $\xx^{rr}\times V$ has codimension
 $\geq 2$ in $\xx_\norm\times V$.
The varieties $\xx^\circ\times_{\zz^\circ}\ss$ and
$\xx_\norm\times V$ are normal.
Therefore, restriction of functions yields
the following bijections:
\beq{bij1}
(\beta_\ss)^*:\
\C[\xx^\circ\times_{\zz^\circ}\ss]
\iso
\C[\xx^{rr}\times_{\zz^{rr}}\ss],
\en\text{resp.}\en
(\beta_V)^*:\
\C[\xx_\norm\times V]\iso
\C[\xx^{rr}\times V].
\eeq

Similarly,  the inclusion
$\alpha_i$ in diagram \eqref{xxv}
induces a   restriction map
\beq{jres}
\alpha_i^*:\ 
\mmp_G(\xx_i\times V,\ V^{\o m})
\too
\mmp_G(\xx_i\times_{\zz_i}\ss_i,\ V^{\o m}), \quad i=1,2.
\eeq
We will show in Lemma
\ref{lem2} below that this map  is surjective.

To complete the proof of  Lemma \ref{lem1},
let
$f:\ \xx^{rr}\times_{\zz^{rr}}\ss\to V^{\o m}$
be a  $G$-equivariant map.
Then, thanks to the surjectivity of \eqref{jres},
there exist maps $f_i\in \mmp_G(\xx_i\times V,\ V^{\o m}),\
i=1,2,$
such that one has
$f|_{\xx_i\times_{\zz_i}\ss_i}=
f_i|_{\xx_i\times_{\zz_i}\ss_i}$.
It follows that the restrictions of the maps $f_1$ and $f_2$
to the set $(\xx_1\times V)\cap(\xx_2\times V)$
agree. Therefore, the map $f$ can be
extended to a regular
map $\xx^{rr}\times V=
(\xx_1\times V)\cup(\xx_2\times V)\to
V^{\o m}$.
Thus, we have proved (modulo
Lemma \ref{lem2}) the surjectivity of 
the restriction map  $(\alpha^{rr})^*$
induced by the vertical imbedding $\alpha^{rr}$
in the middle of diagram \eqref{xxv}.
This, combined with isomorphisms \eqref{bij1}, implies  the surjectivity of 
the map $\alpha^*$ and Lemma
\ref{lem1} follows.
\end{proof}

\subsection{}\label{alpha} To complete the proof of the theorem,
it remains to prove the following

\begin{lem}\label{lem2}
The map $\alpha_i^*,\ i=1,2,$ in \eqref{jres}  is
surjective.
\end{lem}
\begin{proof}
We may restrict ourselves
to the case $i=1$,
the case  $i=2$ being similar.

Let  $\yy\sset\g\times\t$ be the closed subvariety
considered in \S\ref{ysec}, and let
$\yr=\{(x,t)\in \yy\mid x\in \g^r\}.$
Write
$q:\ \xx\to\yy,\ (x,y,t_1,t_2)\mto (x,t_1)$
for the projection. According to Lemma \ref{conormal}
we have that $q\inv(\yr)=N_\yr$,
 the total space of the conormal bundle on $\yr$.
Thus, we obtain
$\dis\xx_1=\big\{(x,y,t_1,t_2)\in
\xx,\  x\in\g^r\big\}=q\inv(\yr)=N_\yr$.

It will be convenient to introduce the following set
$$Y:=\{(x,t,v)\in \g\times\t\times V\mid
(x,t)\in \yy\en\&\en \C[x]v=V\}.
$$

Note that, for any $(x,t,v)\in Y$, the element $x$
is automatically regular,
as has been already observed earlier.
Therefore,   we have
$Y\sset \yr\times V$ and the projection $(x,t,v)\mto (x,t)$
gives a smooth and surjective
 morphism $Y\to \yr$.
Unraveling the definitions 
we obtain 
\begin{multline*} 
\xx_1\times_{\zz_1} \ss_1=\{(x_1,x_2,t_1,t_2,v)\in
\xx\times V\mid
 \C[x_1]v=V\}
=
q\inv(\yr)\times_\yr Y=N_\yr\times_\yr Y.
\end{multline*}

Using these identifications, we
see that the map $\alpha_1$ from \eqref{xxv}
fits into a cartesian square
\beq{big}
\xymatrix{
 N_\yr\times_\yr Y\ \ar[d]^<>(0.5){pr_Y}
\ar@{^{(}->}[rr]^<>(0.5){\alpha_1}
&&\ 
N_\yr\times V\ \ar[d]^<>(0.5){q_V:=q\times\Id_V}\\
Y\ 
\ar@{^{(}->}[rr]&&
\ \yr\times V
}
\eeq

Let $\ddd:=(\yr\times V)\sminus Y$.
This is an irreducible divisor in
$\yr\times V$, the principal divisor associated with the
function $(x,t,v)\mto \langle{\operatorname{vol}},
\ v\wedge x(v)\wedge
x^2(v)\wedge\ldots\wedge x^{n-1}(v)\rangle$
where ${\operatorname{vol}}\in\wedge^n V^*$ is a fixed nonzero
element.  
Using the cartesian  square in \eqref{big},
we see that $(N_\yr\times V)\sminus 
(N_\yr\times_\yr Y)=(q_V)\inv(\ddd)$   is an irreducible divisor
 in 
$N_\yr\times V$.
Now,
view $\ddd$ as a subset of $\g\times\t\times V$.
Then, $\ddd^{rs}:=\ddd\cap (\g^{rs}\times\t\times V)$ is an open  dense subset of
$\ddd$ hence $(q_V)\inv(\ddd^{rs})$ is an open  dense subset of
$(q_V)\inv(\ddd)$.

Let $T$
be the maximal torus of diagonal matrices in $GL_n$.
Thus, the
 set $\t\cap\g^r$ may be identified
with the set of matrices
$t={\operatorname{diag}}(z_1,\ldots,z_n)\in \C^n$ 
such that $z_i\neq z_j\ \forall i\neq j$.
Let $x=t={\operatorname{diag}}(z_1,\ldots,z_n)\in\t\cap\g^r$.
For  $v=(v_1,\ldots,v_n)\in \C^n=V$,
 the triple $(x,t,v)$ is contained in $\ddd^{rs}$ if and only if 
$v$ is
not a cyclic vector for  the linear operator
$x$ which holds if and only if there exists an
$i\in[1,n]$ such that $v_i=0$.
Let $V_i\sset V$ denote the hyperplane formed by
the elements with the vanishing $i$-th coordinate.

To complete the proof of the lemma, 
let $f\in\mmp_G(\xx^\circ\times_{\zz^\circ}\ss,\
V^{\o m})$.
We must show that $(\alpha_1)^*(f)$, viewed as a map $N_\yr\times_\yr Y
\to V^{\o m}$, has no singularities at $(N_\yr\times V)\sminus 
(N_\yr\times_\yr Y),$  a divisor in a smooth variety.
The set 
$(q_V)\inv(\ddd^{rs})$ is an open dense subset 
of that  divisor.
Hence, by $G$-equivariance, it suffices to prove
that, for  $x=t\in\t\cap \g^{rs}$ and a fixed element
$(x,y,t,t')\in q\inv(x,t)=N_\yr$,
the rational map $f_V: V\to V^{\o m}$
given by the assignment
$v\mto f(x,y,t,t',v)$
has no poles at  the divisor ~$\cup_i\ V_i\ \sset\ V$.

It will be convenient to
choose a $T$-weight basis $\{u_\gamma\}$ of the vector space $V^{\o m}$.
Thus, for each $\gamma$ and any 
diagonal matrix ${\operatorname{diag}}(z_1,\ldots,z_n)\in T$,
we have
${\operatorname{diag}}(z_1,\ldots,z_n)u_\gamma=
z_1^{m(\gamma,1)}\cdots z_n^{m(\gamma,n)}\cdot u_\gamma$
where $m(\gamma,i)\in\Z_{\geq0}$.
Expanding the function $f_V$ in the basis $\{u_\gamma\}$ one can write
$f_V(v)={\sum_\gamma f^\gamma_V(v)\cdot u_\gamma}$
where $f^\gamma_V$ are 
Laurent polynomials of the form
$$
v=(v_1,\ldots,v_n)\ \mto\ f^\gamma_V(v)=
\sum_{k_1,\ldots,k_n\in\Z}\
a^\gamma_{k_1,\ldots,k_n}
\cdot v_1^{k_1}\cdots v_n^{k_n},
\quad a^\gamma_{k_1,\ldots,k_n}\in\C.$$

Now, the map $f$ hence also the map $(\alpha_1)^*(f)$ is $G$-equivariant.
Since $(x,y)\in\tt$ we deduce that  $f_V$ is a $T$-equivariant map.
Thus, for each $\gamma$ and all  ${\operatorname{diag}}(z_1,\ldots,z_n)\in T$,
 we must have
\begin{multline*}
\sum_{k_1,\ldots,k_n\in\Z}\
a^\gamma_{k_1,\ldots,k_n}
\cdot (z_1v_1)^{k_1}\cdots (z_nv_n)^{k_n}\cdot u_\gamma
\ =\
f^\gamma_V[{\operatorname{diag}}(z_1,\ldots,z_n)v]
\\
=\ {\operatorname{diag}}(z_1,\ldots,z_n) [f^\gamma_V(v)]\
=\ \sum_{k_1,\ldots,k_n\in\Z}\
a^\gamma_{k_1,\ldots,k_n}
\cdot v_1^{k_1}\cdots v_n^{k_n}\cdot
[z_1^{m(\gamma,1)}\cdots z_n^{m(\gamma,n)}\cdot 
u_\gamma].
\end{multline*}

The above equation yields
$k_i=m(\gamma,i)$ whenever $a^\gamma_{k_1,\ldots,k_n}\neq 0$.
Therefore, 
the function $f_V$ takes the following simplified form
\beq{fV} f_V(v_1,\ldots,v_n)=
\sum_\gamma\ a^\gamma\cdot  v_1^{m(\gamma,1)}\cdots
v_n^{m(\gamma,n)}\cdot u_\gamma,
\quad a^\gamma\in\C.
\eeq

Since
the weights of the $T$-action in $V^{\o m}$ 
are nonnegative, i.e.,  $m(\gamma, i)\geq 0$,
the right hand side 
of \eqref{fV}  has no poles at  the divisor ~$\cup_i\ V_i$
 and the lemma follows.
\end{proof}

\subsection{Proof of Proposition \ref{aa}}\label{gg}
Let $v_o=(1,1,\ldots,1)\in \C^n=V$.
Restriction of polynomial functions
via the imbedding $\zeta:\
\tt=\tt\times\{v_o\}\into \zz_\red\times V$ gives
an algebra homomorphism
\beq{resa}
\zeta^*:\ \C[\zz_\red]\o \C[V]=\C[\zz_\red\times V]
\too \C[\tt],\quad
f\mto f|_{\tt\times\{v_o\}}.
\eeq

Given a $G$-variety $X$ and an integer $k\geq 0$,
let
$\C[X]^{\det^k}=\{f\in\C[X]\mid
g^*(f)=(\det g)^k\cdot f,$ $\forall g\in G\}$
be the subspace of $\det^k$-semi-invariants
of the group $G=GL_n$.
Looking at the action of the group $\C^\times\sset GL_n$,
of scalar matrices, shows that one has $\C[\zz_\red\times V]^{\det^k}=
\big(\C[\zz_\red]\o\C^{kn}[V]\big)^{SL_n}.$

Following \cite{GG}, we 
introduce an affine variety
\begin{equation}\label{ss}
 \bar{\ss} := \{ (x,y,v,v^*)\in \g\times\g\times
V\times V^* \mid [x,y]+v\o v^*=0\}.
\end{equation}

The assignment
$(x,y,v)\mto (x,y,v,0)$ gives a $G$-equivariant
closed imbedding $\iota:\
\zz_\red\times V\into \bar{\ss}$. Pull-back of functions via the
imbedding $\zeta$, resp.
$\iota$, yields linear maps
\beq{ab}
\xymatrix{
\C[\bar{\ss}]^{\det^k}\ \ar@{->>}[r]^<>(0.5){\iota^*}&
\ 
\ \C[\zz_\red\times V]^{\det^k}=
\big(\C[\zz_\red]\o\C^{kn}[V]\big)^{SL_n}\
\ar[r]^<>(0.5){\zeta^*}&\ \C^{kn}[\tt],
}\quad k\geq 0.
\eeq

The composite map
in \eqref{ab} was considered in \cite{GG}.
According to Proposition A2 from 
\cite[Appendix]{GG}, the image of the
map $\zeta^*\ccirc\iota^*$ is equal to $A^k$ and,
moreover, this map yields an isomorphism
$\C[\bar{\ss}]^{\det^k}\iso A^k$. 
Note further that the  restriction map $\iota^*$ in \eqref{ab}  is surjective
since 
the group $G$ is reductive. It follows
that each of the two maps in \eqref{ab}
must be an isomorphism. This yields the 
statement of the proposition.
\qed

\section{Some applications}\label{some}
\subsection{}\label{red_sec} For each $\Ad G$-orbit $O\sset \g$,
the closure $\overline{N_O}$ of the total
space of the  conormal bundle on $O$
is a Lagrangian subvariety in $T^*\g$.
We define
\beq{zznil}
\zz^\nil:=\{(x,y)\in\gg\mid [x,y]=0,\ x\in\NN\}=
\cup_{O\sset\NN}\
\overline{N_O},\eeq
where the union on the right is taken
over the (finite) set of  nilpotent
$\Ad G$-orbits in $\g$.

Let  $\gamma:\ \g\to\g/\!/G=\t/W$ be  the adjoint quotient morphism.
We introduce the  following maps
\beq{nil}
\th_\zz:\ \zz_\red\to
\g/\!/G=\t/W,\en (x,y)\mto \gamma(y)
\quad\text{resp.}\quad
\th_\xx:\ \xx\to\t,\en (x,y,t_1,t_2)\mto t_2.
\eeq

The proposition below may be viewed as an analogue of the crucial 
flatness result
in Haiman's proof of the $n!$ theorem, see Proposition 3.8.1 and
Corollary 3.8.2 in  \cite{Ha1}.  Proposition \ref{flat} was
also obtained independently by I. Gordon \cite{Go} who
used it in his proof of positivity of 
the Kostka-Macdonald polynomials.

\begin{prop}\label{flat}  The 
composite  morphism
$\wt\th_\zz:\ \zz_\norm\to\zz_\red\stackrel{\th_\zz}\to \g/\!/G,$
resp. $\wt\th_\xx:\ \xx_\norm\to\xx\stackrel{\th_\xx}\to\t$
is  flat.
Furthermore, the scheme theoretic zero fiber
$\wt\th_\zz\inv(0)$,
resp. $\wt\th_\xx\inv(0)$,
is a Cohen-Macaulay scheme, not necessarily reduced in general.
\end{prop}
\begin{proof} The dilation action
of $\C^\times$ on $\t$ descends to a
contracting $\C^\times$-action on $\t/W$. This
 makes  the map $\wt\th_\zz$  a
$\C^\times$-equivariant morphism. 
Hence,
 for any $t\in\t/W$, one obtains
$$\dim (\wt\th_\zz)\inv(t)= \dim(\th_\zz\inv(t))\leq\dim (\th_\zz)\inv(0)=
\dim\zz^\nil=\dim\g=\dim\zz-\dim\t/W,
$$

Here, the inequality holds thanks to the semi-continuity
of fiber dimension, the second equality is a consequence
of the  (set theoretic) equation $(\th_\zz)\inv(0)=\zz^\nil$,
and the third  equality is a consequence of  \eqref{zznil}.
The scheme $\zz^\norm$ being
Cohen-Macaulay, we conclude that the map 
 $\wt\th_\zz$ is flat and each scheme theoretic fiber
of that map is Cohen-Macaulay, cf. \cite[\S 16A, Theorem 30]{Ma}.

The proof of the corresponding statements involving the
variety $\xx$ is similar and is left to the reader.
\end{proof}

\subsection{}\label{LSsec}
We take $W$-invariant
 global sections of the sheaves on each side of formula
\eqref{mmbimod}. The functor $\Ga(\g\times\t,-)^W$
being exact, one obtains an isomorphism
of left $\dd(\g)$-modules
\beq{winv}
\Ga(\g\times\t,\mm)^W=\dd(\g)/\dd(\g)\cd\ad\g.
\eeq

\begin{proof}[Proof of Corollary \ref{LS}]
The Hodge filtration on
$\mm$  is $W$-stable and it induces,
by restriction to $W$-invariants,
 a  filtration $F_\idot$ on $\Ga(\g\times\t,\mm)^W$.
The  functor $\Ga(\g\times\t,-)^W$
clearly commutes with taking an associated graded module.
Therefore,  
we deduce 
\begin{multline*}
\gr^F[\dd(\g)/\dd(\g)\cd\ad\g]=
\gr^F[\Ga(\g\times\t,\,\mm)^W]=
\big[\gr^\hodge\Ga(\g\times\t,\, \mm)\big]^W\\=\Ga(\g\times\t,\ \gr^\hodge\mm)^W=
\C[\xx_\norm]^W=\C[\zz_\norm],
\end{multline*}
where the fourth equality holds
by  Theorem \ref{mthm} and the fifth equality is a
consequence of
the isomorphism $\C[\zz_\norm]=\C[\xx_\norm]^W$,
see Corollary \ref{vectbun}(i).
\end{proof}

We recall that  the vector space $\dd(\g)/\dd(\g)\cdot\ad\g$
has a natural right action of the algebra $\AA$,
cf. \S\ref{hc_filt}. The isomorphism in
\eqref{winv} intertwines the natural left
$\dd(\t)^W$-action on
$\Ga(\g\times\t,\mm)^W$ and
the right $\AA$-action on $\dd(\g)/\dd(\g)\cdot\ad\g$ via the isomorphism 
$\dd(\t)^W\cong\AA^{op}$ induced by the map $\Xi$, see 
\eqref{mmbimod}.
Therefore, the filtration $F_\idot$
makes $\dd(\g)/\dd(\g)\cdot\ad\g$ a filtered 
 $(\dd_\g,\AA)$-bimodule. In particular,
one may view $\dd(\g)/\dd(\g)\cdot\ad\g$ as a right filtered $(\sym\g)^G$-module
via the map $(\sym\g)^G\to\AA$ given by
the composition of natural maps $(\sym\g)^G\into \dd(\g)^G\onto
\dd(\g)^G/[\dd(\g)\cd\ad\g]^G$.

\begin{proof}[Proof of Corollary \ref{ls_cor}]
By Corollary \ref{LS}, we have $\gr^F\big[\dd(\g)/\dd(\g)\cd\ad\g\big]=
\C[\zz_\norm]$. Therefore, Corollary \ref{zzcm} implies that 
$\gr^F\big[\dd(\g)/\dd(\g)\cd\ad\g\big]$ 
is a Cohen-Macaulay $\C[\gg]$-module; moreover,
this module is flat over the algebra $(\sym\g)^G$, by
Proposition \ref{flat}.
Now,
a result of Bjork \cite{Bj} insures
that $\dd(\g)/\dd(\g)\cd\ad\g$
is a Cohen-Macaulay $\dd(\g)$-module
which is, moreover, flat over $(\sym\g)^G$.
\end{proof}

\bigskip
\centerline{\large{\sc{Index of Notation}}}
\medskip

{\small
\begin{multicols}{2}
$T^*X,\ X_\red,\ X_\norm,\ \psi,\ \oo_X,\ \dd_X,\ \kk_X,\ 
\quad$ \ref{not}

$G,\ T,\ \g,\ \t,\ W,\ \mathbf{r},\ \sign
\quad$ \ref{not}

$\DD,\ \rad,\ \ad,\ \mm\quad$ \ref{r_sec}

$\ggr^\hodge\quad$ \ref{r_sec},\, \ref{hc_hodge},\, \ref{filt_sec}.

$\gg,\ \zz,\ \tt,\ \kap,\ \res,\
\langle-,-\rangle \quad$ \ref{r_sec}

$v^\top \quad$ \ref{r_sec},\, \ref{hc_filt}

$\xx \quad$ \ref{r_sec},\, \ref{xxdef}

$R^+\quad$ \ref{actions}, \ref{canfilt}

$p\quad$ \ref{dsec},\, \ref{xxdef}

$p_\norm,\ \rr,\ \rr^E,\ E^*,\ \g_x,\ \g_{x,y},\
\rr_\bx,\
\g^r,\ \zz^r\quad$ \ref{dsec}

$ \fz\quad$ \ref{dsec},\, \ref{unistab}

Small representation \quad \ref{small_intro}

$L^{\mathfrak a},\ L^\fz\quad$ \ref{small_intro}

$S_n\quad$ \ref{small_intro},\, \ref{mhai}

$\bbe,\ \bullet_\bbe,\ h_s,\ \bh,\ \g^s,\ 
R^+_s,\ W_s,\ 
\bd_s\quad$ \ref{pnp_sec}

$\C^{\leq m}[\t],\ \fo_\idot,\,\ft_\idot,
\quad$ \ref{pnp_sec},\, \ref{rees_sec},\, \ref{rees_ctt}

$\gr^F\C[\wbh]\quad$ \ref{pnp_sec}

Exceptional principal nilpotent pair\quad \ref{bde_sec}

$\check{h}_s, \ \bde,\ \bbs_\bbe,\ \bbs_\bbe(\bbe),\
\bbde,\ \bbss\quad$ \ref{bde_sec}

$\hilb^n(\C^2),\ \wt\hilb^n(\C^2),\ \zz^\circ,\ \pp,\
A^k\quad$ \ref{mhai}, \ref{psec}

$\VV\quad$ \ref{mhai},\, \ref{vt0}

$\g^{rs},\ \zz^{rs},\ \xx^{rs},\ \g/\!/G,\ \zz/\!/G,\ 
\yy,\ \yy^{rs},\ \yr,\ N_\yr\quad$ 
\ref{ysec}

$p_{_\tt}\quad$ \ref{xxdef}

%Lagrangian subvariety\quad \ref{ysec}

%Good filtration, support cycle \quad \ref{ddfilt}

$q,\ F^\ord,\ \gr^\ord,\ \ggr^F,\
\dd_{X\to Y},\ \int_f,\ \int^R_f
\quad$ \ref{ddfilt}

$\AA,\ \Xi,\ \I,\ \J,\ F^\ord\mm,\ \J
\quad$ \ref{hc_filt}

$\t^r,\ \del_\t,\ dx,\ dh, \
\bj,\ \bj_{!*}\oo_\yr\quad$ \ref{dx}

$\Lmod{{\mathsf F}\dd_X},\ D^b_{coh}({{\mathsf F}\dd_X}),\
\ggr(M,F)\quad$ \ref{hodged}

Hodge module, Hodge filtration \quad \ref{hodged}

$\mh(X),\ {\mathbb D},\ {\mathsf F}{\mathbb D},\
\coh Z,\ \dcoh(Z),\ \dgr,\ 
\quad$ \ref{hodged}

$F^\hodge\quad$ \ref{hc_hodge},\, \ref{filt_sec}

$\H^k_Z,\ \zz_i,\ \zz^{rr},\ \xx_i,\ \xx^{rr},\
\quad$ \ref{main_res} 

$\bb,\ \b,\ \tg,\ \tgg,\ \nu,\ \nnu,\
\mu,\ \mmu,\ \T_X,\ 
\quad$ \ref{tx}

$\pi,\ \omega\quad$ \ref{tb}

${\mathfrak q}\quad$ \ref{tb},\, \ref{dima_sec}

$\eps,\ \La\quad$ \ref{sympl},\, \ref{filtmmG}

$\pr_{\La\to T^*\bb},\ \pr_{\La\to\gg\times\tt},\
\N,\ \wt\N,\ \Phi,\ \Psi,\ \kkap
\quad$ \ref{sympl}

$\tx,\ \tx^{rr}
\quad$ \ref{tx_sec}

$\iim\quad$ \ref{tx_sec},\, \ref{filt_sec}

$\ppi\quad$ \ref{transv},\, \ref{xyzsec}

$\aa,\ \H^j,\ \eu,\ \partial_{\kkap^*\eu}\quad$ \ref{dima_sec}

$\vvpi,\ \vsi\quad$ 
\eqref{fdiag}

${\mathcal G}_X,\ \fz_X,\ L^{\fz_X}\quad$ \ref{unistab}

$R,\ \X,\ \ccl,\ L_\b^{\ccl\la},\ L_\bb^{\ccl\la},\
L_\bb^\la,\ \g^\al_\bb,\ \uh,\ L^\uh\quad$ 
\ref{canfilt}

$pr,\ \tg^r,\ \gamma,\ \wt\gamma,\ \vartheta,
\ L^\la_\tg,\ \g^\al_\tg\quad$ 
\ref{morphism_sec}

$L^\la_\yr,\ \la^L,\
\la^k\quad$ 
\ref{bei}

$L^\la_\yr,\ \g^\al_\yr,\ e^k_\al
\quad$ \ref{vin}

$h_x,\ n_x,\ \gg^\b\quad$ \ref{ss1}

$\zz(\fl),\ {\mathfrak N}(\fl),\, S,\
\t_\fl,\ \ts_\fl,\ \tts_\fl,\ \xx_\fl(\g)
\quad$ 
\ref{ss2}

$\zz_\fl(\g)\quad$ \ref{ss3}

$\rees E,\ \reo,\ \ret,\ \gro,\ \grt,\
\aleph,\ F\min,\ F\max,\ \mathsf{can}
\quad$ \ref{rees_sec}

$\C^{i,j}[\tt],\ I_\bh,\ \fy,\ \wp,\ I\bim
\quad$ \ref{rees_ctt}

$\vka,\ \wt\vka,\ \ups\quad$ \ref{map_sec}

$S^m,\ \ssym,\ H^\perp,\ \del_s,\ V_s,\
\check V_s,\ \nabla,\ \nabla_{i,j}\quad$ \ref{pftr}

$f^\sharp,\ \rese,\ u,\ \pr^\sign_\rr,\ \pr^\sign_{W\cdot\bh},\
\bbde_\bx\quad$ \ref{pnpappl}

$G_{x,y},\ \C\langle x,y\rangle,\
\ssc,\ \hilb,\ \rho,\ \delta\quad$ \ref{psec}

$\ha,\ \xx^\circ,\ \eta,\ \wt\delta,\ \wt\rho,\
\wt p,\ h,\ \varrho\quad$ \ref{hilb}

$\mmp_G,\ \L_Y(E),\ \al,\ \al_i,\ \Phi_i,\ 
\be_\ssc,\ \be_V\quad$ \ref{vt0}

$Y,\ \ddd,\ 
\quad$ \ref{alpha}

$\zz^\nil,\ \th_\zz,\ \th_\xx,\ \wt\th_\zz,\ 
\wt\th_\xx
\quad$ \ref{red_sec}
%iind
\end{multicols}
}

%bbb
\small{
\bibliographystyle{plain}

}

\end{document}